\documentclass[pdflatex,sn-mathphys-num]{sn-jnl}% Math and Physical Sciences Numbered Reference Style
\usepackage{graphicx}%
\usepackage{multirow}%
\usepackage{amsmath,amssymb,amsfonts}%
\usepackage{amsthm}%
\usepackage{mathrsfs}%
\usepackage[title]{appendix}%
\usepackage{xcolor}%
\usepackage{textcomp}%
\usepackage{manyfoot}%
\usepackage{booktabs}%
\usepackage{algorithm}%
\usepackage{algorithmicx}%
\usepackage{algpseudocode}%
\usepackage{listings}%
\usepackage{multirow}
\usepackage{subfigure}
\usepackage{etoolbox}
\newtheorem{theorem}{Theorem}[section]% meant for sectionwise numbers
\newtheorem{example}{Example}%
\newtheorem{remark}{Remark}%

\newtheorem{lemma}{Lemma}[section]%
\newtheorem{assumption}{Assumption}[section]%

\usepackage{caption}

\newenvironment{proofs}{\noindent \textit{Proof }}{\hfill$\square$}
\usepackage{geometry}
\def\yx#1{{\color{red} #1}}
\def\sM{\mathcal{M}} 
\def\bR{\mathbb{R}}

\begin{document}

\title[On the resolution of $\ell_1$-norm minimization via a two-metric adaptive projection method]{On the resolution of $\ell_1$-norm minimization via a two-metric adaptive projection method}


\author[1]{\fnm{Hanju} \sur{Wu}}%\email{xxx}

\author*[1,2]{\fnm{Yue} \sur{Xie}}\email{yxie21@hku.hk}


\affil[1]{\orgdiv{Department of Mathematics}, \orgname{The University of Hong Kong}, \orgaddress{\street{Pokfulam}, \city{Hong Kong}%,\postcode{100190}, \state{State}, \country{Country}
}}

\affil[2]{\orgdiv{Institute of Data Science}, \orgname{The University of Hong Kong}, \orgaddress{\street{Pokfulam}, \city{Hong Kong}%, \postcode{10587}, \state{State}, \country{Country}
}}

%\affil[3]{\orgdiv{Department}, \orgname{Organization}, \orgaddress{\street{Street}, \city{City}, \postcode{610101}, \state{State}, \country{Country}}}


\abstract{\fontsize{11}{15}\selectfont
In this work, we propose an efficient two-metric adaptive projection method for solving the $\ell_1$-norm minimization problem. Our approach is inspired by the two-metric projection method, a simple yet elegant algorithm proposed by Bertsekas for bound/box-constrained optimization problems. The low per-iteration cost of this method, combined with the ability to incorporate Hessian information, makes it particularly attractive for large-scale problems, and our proposed method inherits these advantages. 
Previous attempts to extend the two-metric projection method to $\ell_1$-norm minimization rely on an intermediate reformulation as a bound-constrained problem, which can lead to numerical instabilities in practice, in sharp contrast to our approach.  Our algorithm features a refined partition of the index set,  an adaptive projection, and a novel linesearch rule.  It can accommodate singular Hessians as well as inexact solutions to the Newton linear system for practical implementation.  We show that our method is theoretically sound - it has global convergence.  Moreover, it is an active-set method capable of manifold identification: the underlying low-dimensional structure can be identified in a finite number of iterations, after which the algorithm reduces to an unconstrained Newton method on the identified subspace. Under an error bound condition,  the method attains a locally superlinear convergence rate. Hence, when the solution is sparse, it achieves superfast convergence in terms of iterations while maintaining scalability, making it well suited for large-scale problems.
We conduct extensive numerical experiments to demonstrate the practical advantages of the proposed algorithm over several competitive methods from the literature, particularly in large-scale settings. }




\maketitle

\fontsize{11}{15}\selectfont
\section{Introduction}
\label{introduction}

In this paper, we consider the $\ell_1$-norm minimization problem:
\begin{equation}\label{l1min}
        \min_{x \in \mathbb{R}^n} \psi(x) \triangleq f(x) + h(x),
\end{equation}
where $f: \mathbb{R}^n \to \mathbb{R}$ is continuously differentiable, and $h(x) = \gamma \| x \|_1$. 

The $\ell_1$-norm regularization is widely used in machine learning and statistics to promote sparsity in the solution. Many algorithms have been proposed to solve the $\ell_1$-norm minimization problem, such as the proximal gradient method \citep{beck2009fast}, the accelerated proximal gradient method \citep{nesterov2013gradient}, the alternating direction method of multipliers (ADMM) \citep{boyd2011distributed}, and the coordinate descent method \citep{wright2015coordinate}. 
However, these first-order methods at best have linear convergence rate and may converge slowly near the optimal point \citep{tseng2010approximation,tseng2009coordinate,hong2017linear}. 
Second-order methods are more desirable when seeking highly accurate solution.
When $f$ is a twice continuously differentiable function, it is natural to expect
algorithms to have faster convergence rates by exploiting the Hessian information of $f$
at each iterate $x_k$. 
Some successful second-order algorithms with superlinear convergence rate have been proposed to solve the general composite optimization problem in the form of problem \eqref{l1min}, such as the successive quadratic approximation (SQA) methods \citep{yue2019family,lee2023accelerating,mordukhovich2023globally}, the semi-smooth Newton methods \citep{hu2022local,xiao2018regularized,li2018highly,milzarek2014semismooth} and some active set methods \citep{bareilles2023newton,doi:10.1137/24M1697001} based on the notions of {\it partial smoothness} and {\it identifiability} \citep{lewis2018partial,lewis2024identifiability,lewis2022partial}.

In this work, we propose a two-metric adaptive projection method to address problem \eqref{l1min}. The original two-metric projection method is proposed by Bertsekas and Gafni \citep{bertsekas1982projected,gafni1984two} to solve the bound-constrained optimization problem:
\begin{equation}\label{bdmin}
    \begin{array}{rl}
        \min\; f(x) \quad \text {subject to} \; x \geq 0.
    \end{array}
\end{equation}
Two-metric projection method is a kind of active set method which can be described as follows:
\begin{align}\label{alg: 2mproj}
    x_{k+1} \triangleq P(x_k - t_k D_k \nabla f(x_k)),
\end{align}
where $D_k$ is a positive definite matrix, and $P(\cdot)$ denotes the Euclidean projection onto the nonnegative orthant. Note that $P(\cdot)$ has a simple closed form: given any vector $v \in \mathbb{R}^n$, $i$th component of $P(v)$ is $\max\{ v^i, 0\}$. 
Therefore, the method has a low per-iteration cost if $D_k$ is provided. When $D_k$ is constructed using the Hessian information, the algorithm potentially converges fast. However, there is an issue regarding \eqref{alg: 2mproj}: starting from the current iterate $x_k$, the objective function value may not decrease for any small stepsize $t_k > 0$ even when $D_k$ is chosen as inverse of the Hessian. 
This means that a naive projected Newton step may fail. To solve this issue, we need to divide the index set into two parts in each iteration:
\begin{align}
    \label{index}
    I_k^+ & \triangleq \{ i \mid 0 \le x_k^i \le \epsilon_k, \nabla_i f(x_k) > 0 \}, \\
    I_k^- & \triangleq \{ 1,\ldots,d \} \setminus I_k^+,
\end{align}
and choose $D_k$ such that it is diagonal regarding $I_k^+$, i.e.:
\begin{align}
    \label{Dk-rules}
    D_k[i,j] = 0, \forall i,j \in I_k^+ \mbox{ and } i \neq j.  
\end{align}
The problem is fixed via this approach.  {In particular, $I_k^+$ includes the coordinates where the iterate is sufficiently close to the boundary and the negative gradient points to the boundary.} Moreover, if the non-diagonal part of $D_k$ is constructed based on the Hessian inverse, it can be shown that under the assumption of second-order sufficient conditions, 
the sequence of iterates identifies the active constraints in finite number of iterations and is attracted to a nondegenerate critical point (critical point satisfies strict complementarity) at a superlinear rate. 
%Recently, the worst-case complexity property of the two-metric projection method in nonconvex setting has been studied in \citep{wu2024study}: the algorithm can locate an $\epsilon$-approximate
%first-order optimal point in $\mathcal{O}(\epsilon^{-3})$ iterations. By properly scaling $D_k$, the complexity result is enhanced to $\mathcal{O}(\epsilon^{-2})$.

Note that the $\ell_1$-norm minimization problem (\ref{l1min}) is equivalent to the following bound-constrained optimization problem \eqref{eq-bound} when splitting $x$ into non-negative variables representing positive and negative components of $x$. 
\begin{align}
    \label{eq-bound}
    \begin{array}{rl}
    \min\limits_{y, z \in \mathbb{R}^n} & F(y,z) \triangleq f\left(y-z\right)+\gamma \sum_{i}\left(y_{i} +z_{i} \right) \\    
    \text { s.t. } 
        & y \geq 0, \quad z \geq 0.
\end{array}
\end{align}
Two-metric projection can be applied to this
bound-constrained reformulation to resolve problem \eqref{l1min}.
Unfortunately, this approach has a flaw.  To obtain fast convergence, we often use Newton's method: build non-diagonal part of $D_k$ using the Hessian inverse. However, the corresponding submatrix of $\nabla^2 F(y,z)$ can be singular even $\nabla^2 f$ is globally positive definite.  For example, suppose that $f(x) \triangleq (x-1)^2/2$ and $h(x) \triangleq |x| $.  Then $F(y,z) = (y-z-1)^2/2$,  $\nabla F = \begin{pmatrix}
y-z \\ -y+z + 2 \\
\end{pmatrix}$ and $\nabla^2 F = \begin{pmatrix}
1 & -1 \\ -1 & 1 \\
\end{pmatrix}$ is singular.  Suppose that $y_k = 0, z_k > \epsilon_k > 0$, then $I_k^+ = \emptyset$, $I_k^- = \{1,2\}$.  Therefore, we cannot use inverse of $\nabla^2 F$ to build nondiagonal part of $D_k$.  What is worse,  if we try to solve the equation
\begin{align*}
\nabla^2 F (y_k,z_k) \cdot p = \nabla F (y_k,z_k),
\end{align*}
we will find that the solution does not exist.  Therefore, approximation approach of the Hessian leads to numerical instability.  More detailed and generalized discussions are given in Section~\ref{sec: prelim}.
%This is mainly due to the index selection in the two-metric projection method, which may lead to a singular sub-Hessian matrix $D_k$ in practice.
Also note that the global convergence and local superlinear convergence rate of the two-metric projection method require the eigenvalues of $D_k$ to be uniformly bounded away from zero and infinity (and hence the eigenvalues of the Hessian must also be uniformly bounded from below and above), and the resulting linear system must be solved exactly. These conditions are rarely satisfied in practice.

Such limitations motivate us to develop a more robust and stable alternative with a sound convergence analysis framework applicable to a wide range of real-world problems.
We propose a novel two-metric adaptive projection method that solves problem \eqref{l1min} directly without resorting to the intermediate bound-constrained reformulation \eqref{eq-bound}, thereby circumventing the aforementioned flaw.  We consider a refined partition of the index set, incorporate the adaptive projection, and design a novel linesearch rule.  Moreover, we do not require the Hessian to be uniformly positive definite nor the linear system to be solved exactly, enabling efficient and practical implementation for solving problem~\eqref{l1min}.  Despite these practical improvements, the algorithm preserves the ability to identify the active manifold (the sparsity structure of the optimal point).
This identification property is computationally valuable because once the active manifold is identified, problem~\eqref{l1min} reduces locally to a low-dimensional smooth optimization problem, and only a low-dimensional linear system needs to be solved per iteration.
After manifold identification, our algorithm {\it automatically} reduces to a Newton method on the active manifold with superlinear convergence to the optimal solution.  {It is worth mentioning that such transition happens spontaneously without testing whether the active manifold is already identified, nor adding a safeguard via some first-order algorithm, in contrast to some existing algorithms in literature \citep{bareilles2023newton,lee2023accelerating}.  In practice,  it is hard to verify that the active manifold is identified with certainty.  In \citep{sun2019we},  the complexity of active manifold identification is derived,  but it requires prior knowledge of the optimal solution and such complexity can be too conservative to apply.}

Our convergence analysis relies on the error-bound (EB) condition, which is weaker than second-order sufficient conditions and holds for a broad class of objective functions \citep{zhou2017unified}. The work \citep{yue2019family} develops an EB-based analysis framework with inexactness conditions for successive quadratic approximation (SQA) methods, which is subsequently applied to semi-smooth Newton methods \citep{hu2022local}.
\citep{lee2023accelerating} established convergence results under local sharpness conditions, a special case of the Kurdyka-\L ojasiewicz (KL) condition.
Notably, in the convex setting,  the KL condition with exponent $1/2$ is equivalent to the EB condition \citep{liao2024error}.
\citep{mordukhovich2023globally} proved local superlinear convergence for their SQA method under metric subregularity of the subdifferential mapping, which is equivalent to the EB condition \citep{drusvyatskiy2018error}.
Thus, the local superlinear convergence guarantees of existing second-order methods are established under conditions equivalent to or generalizing the EB condition.

{\bf Contribution} We summarize our main contributions in this paper as follows:
\begin{enumerate}  \fontsize{11}{15}\selectfont
  \item We propose a novel two-metric adaptive projection method that solves the $\ell_1$-norm minimization problem \eqref{l1min} directly. The proposed method is straightforward to implement and maintains a low per-iteration computational cost. It avoids the drawback that arises when applying the original two-metric projection method to the equivalent bound-constrained formulation \eqref{eq-bound} and removes stringent requirements impeding practicality.
  \item We provide a comprehensive theoretical analysis of the proposed algorithm, establishing global convergence, manifold identification, and local superlinear convergence rates under assumptions applicable to a broad problem class including the  logistic regression and LASSO problems. 
  \item We conduct extensive numerical experiments on large-scale logistic regression and LASSO problems to demonstrate the efficiency and robustness of the proposed algorithm, in comparison with several competitive state-of-the-art methods.
\end{enumerate}

{\bf Organization} The remainder of this paper is organized as follows: In Section~\ref{sec: prelim} we review the relevant background knowledge. 
In Section~\ref{sec: tmap} the proposed algorithm is formally introduced in detail. Section~\ref{sec: conv} presents the theoretical analysis.  In Section~\ref{sec: num} we discuss the experimental results. In Section~\ref{sec: con} we conclude the paper and discuss the future work.
Finally, we show the proofs of some theoretical results in Appendix~\ref{sec: appendixA}.
%In Appendix~\ref{sec: appendixB} we discuss the definition of the identifiable manifold and provide the proof of the equivalence between error bound conditions on the identifiable manifold and the whole space.
%In Appendix~\ref{sec: appendixC} we cite the result of equivalence of identifiability and partial smoothness plus nondegeneracy.

\textbf{Notation} Suppose that $X^*$ denotes the optimal solution set of problem \eqref{l1min}. $\| \cdot \|$ denotes the $\ell_2$-norm for vectors or operator norm for matrices. For a matrix $A$ (vector $a$), let $[A]_{\mathcal{I} \times \mathcal{I}}$ ($[a]_{\mathcal{I}}$) denote the submatrix (subvector) of $A$ ($a$) w.r.t index set $\mathcal{I}$. $a^i$ denote the $i$th component of $a$.
$|\mathcal{I}|$ denotes the cardinality of index set $\mathcal{I}$.
$\mathbf{1}_k$ denotes the vector of all elements equal to 1 in $\mathbb{R}^k$.
For any point $x$ and set $S$, let ${\rm dist}(x,S) \triangleq \inf\{ \| x - y \| \mid y \in S \}$. For a convex function $f: \mathbb{R}^b \to \mathbb{R}$, $\partial f$ denotes its subdifferential set, i.e., $\partial f(x) \triangleq \{ g \mid f(y) \ge f(x) + g^T(y-x),\; \forall y \in \mathbb{R}^b \}$. 
For a smooth function $f$ on $\mathbb{R}^n$, $\nabla f$ and $\nabla^2 f$ denote its Euclidean gradient and Hessian, respectively. $\nabla_i f$ denotes the $i$th coordinate of $\nabla f$.
For a smooth function $f$ on a manifold $\mathcal{M}$, $\operatorname{grad}_{\mathcal{M}} f$ and $\operatorname{Hess}_{\mathcal{M}} f$ denote its Riemannian gradient and Hessian, respectively.  $F|_{\sM}$ denotes the restriction of function $F$ on manifold $\sM$.

\section{Preliminaries}\label{sec: prelim}
In this section we introduce some basic concepts to be used in the algorithm design and analysis in this paper. We will also explain in detail the limitation of the original two-metric projection method when applied to the equivalent bound-constrained formulation.

\noindent\textbf{Optimality conditions and critical points} If $x^*$ is an optimal point of problem \eqref{l1min} then,
    \begin{equation}
        \label{optimal-condition}
        0 \in \nabla f(x^*) + \partial h(x^*)
    \end{equation}
    (see \citep{attouch2010proximal}). Specifically,  when $h(x) = \gamma \| x \|_1$, (\ref{optimal-condition}) is equivalent to
    \begin{equation} \label{optcon}
        \nabla_{i} f\left(x^{*}\right)=\left\{\begin{array}{ll}
-\gamma, & (x^{*})^i>0, \\
a \in[-\gamma, \gamma], \quad & (x^{*})^i=0, \\
\gamma, & (x^{*})^i <0 .
\end{array}\right.
    \end{equation}
If $x^*$ satisfies the first-order necessary conditions \eqref{optimal-condition}, we say that $x^*$ is a {\it critical (stationary) point} of problem \eqref{l1min}.  Note that when $f$ is a convex function, the necessary optimality condition \eqref{optimal-condition} is also sufficient (see \cite{wright2022optimization}). 

\noindent\textbf{Proximity operator}
For a function $h$, the {\it proximity operator} is defined as the set-valued mapping
\begin{equation*}
    \operatorname{prox}_{h}(x) \triangleq  \underset{u \in \operatorname{dom} h}{\arg \min }\left\{h(u)+\frac{1}{2}\|u-x\|^{2}\right\} .
\end{equation*}
For $h(x) = \gamma \| x \|_1$, the proximity operator has a closed form:
\begin{equation*}
    \left(\operatorname{prox}_{\alpha h}(x)\right)^{i}=\left\{\begin{array}{ll}
        x^i - \alpha \gamma, & x^i > \alpha \gamma, \\
        0, & |x^i| \leq \alpha \gamma, \\
        x^i + \alpha \gamma, & x^i < -\alpha \gamma,
    \end{array}\right.
\end{equation*}
which is also known as the {\it soft-thresholding} operator.

\noindent\textbf{Equivalent optimality conditions}
Note that $x^*$ is a critical point iff for any $\alpha > 0$,
\begin{align*}
    x^* = \operatorname{prox}_{\alpha h} (x^* - \alpha \nabla f(x^*)).
\end{align*}
More details can be found in \citep{beck2017first}.

\noindent\textbf{Active set and active manifold}
The ideas of active set and active manifold were first introduced in constrained optimization, and have been extended to nonsmooth optimization. In our context,  for any $x \in \mathbb{R}^n$, the {\it active set} is defined as 
\[\mathcal{A}(x) \triangleq \left\{i \mid x^i=0\right\},\]
and the {\it active manifold} at $x \in \bR^n$ is defined as
\[
        \mathcal{M}_{\mathcal{A}(x)} \triangleq\left\{x \in \mathbb{R}^{n}: x^{i}=0 \text { for } i \in \mathcal{A}(x)\right\}.
        \]
Suppose that $x^*$ is an optimal point of \eqref{l1min},  the active manifold associated with $x^*$ is denoted as $\mathcal{M}_{*} \triangleq \mathcal{M}_{\mathcal{A}(x^*)}$. Note that $\mathcal{M}_{*}$ is a subspace of $\mathbb{R}^n$,  and although $\psi(x)$ may not be smooth around $x^*$ in $\mathbb{R}^n$,  it is locally smooth around $x^*$ on $\mathcal{M}_{*}$.
Therefore, we can consider the calculus on the manifold $\mathcal{M}_{*}$ around $x^*$. %More details about optimization on manifolds can be found in \citep{absil2008optimization,boumal2023intromanifolds}.

%We specify the formula for some concepts in Riemannian optimization used in this paper. More details about these concepts can be found in \citep{absil2008optimization, boumal2023intromanifolds}. 
\noindent{\bf Calculation of Riemannian gradient/Hessian}
For $\mathcal{M}_{*}$ defined in our context, $T_{x} \mathcal{M}_{*}$ denotes the tangent space of $\mathcal{M}_{*}$ at $x \in \mathcal{M}_{*}$. Then we have
    \begin{equation*}
        T_{x} \mathcal{M}_{*}=\mathcal{M}_{*}.
    \end{equation*}
$\mathcal{M}_{*}$ can be viewed as an embedded submanifold of $\mathbb{R}^n$. Then for $x \in \mathcal{M}_*$, $\eta \in T_x \mathcal{M}_*$ and a function $F$ which is second-order smooth, the Riemannian gradient and Hessian of $F$ at $x$ w.r.t.  $\mathcal{M}_*$ can be written as follows,  
    \begin{align}\label{def: Rgrad}
            \operatorname{grad}_{\mathcal{M}_{*}} F(x) & =\operatorname{proj}_{x}(\nabla F(x)), \\
            \label{def: RHes}
        \operatorname{Hess}_{\mathcal{M}_{*}} F(x)[\eta] & =\operatorname{proj}_{x} \left(\nabla^{2} F(x) \eta \right),
    \end{align}
where $\operatorname{proj}_x$ is the Euclidean projection operator onto $T_x \mathcal{M}_{*}$ (In case when $F$ is not smooth around $x$ on $\mathbb{R}^n$, we can replace $\nabla F$ and $\nabla^2 F$ with $\nabla \tilde F$ and $\nabla^2 \tilde F$,  respectively,  where $\tilde F$ is a second-order smooth extension of $F|_{\mathcal{M}_*}$).

More details about the manifold and the Riemannian optimization on the manifold can be found in \citep{absil2008optimization,boumal2023intromanifolds} and Appendix~\ref{app: prelim}.

\noindent\textbf{Equivalent bound-constrained formulation of problem \eqref{l1min}}
%Let $x^{+} \triangleq \max (0, x) $ and $ x^{-} \triangleq -\min (0, x) $. 
Note that by splitting the variable into the difference of two nonnegative vectors,  we can reformulate problem \eqref{l1min} as \eqref{eq-bound}.
%\begin{align}
%    \label{eq-bound}
%    \begin{array}{rl}
%    \min\limits_{y, z \in \mathbb{R}^n} & F(y,z) \triangleq f\left(y-z\right)+\gamma \sum_{i}\left(y_{i} +z_{i} \right) \\    
%    \text { s.t. } 
%        & y \geq 0, \quad z \geq 0.
%\end{array}
%\end{align}
Let $x \triangleq y - z$ in \eqref{eq-bound}.  The gradient and Hessian of $F(y,z)$ are given by
\begin{align*}
    \nabla^2 F(y , z) & = \begin{bmatrix}
       \nabla^2 f(x) & -\nabla^2 f(x) \\
       -\nabla^2 f(x) & \nabla^2 f(x)
      \end{bmatrix}, \quad
      \nabla F(y, z) = \begin{bmatrix}
       \nabla f(x) + \gamma {\mathbf{1}_n} \\
       - \nabla f(x) + \gamma {\mathbf{1}_n}
        \end{bmatrix}.
\end{align*}
The authors in \citep{schmidt2007fast,schmidt2010graphical} first noticed this equivalence and proposed \textbf{ProjectionL1} (the two-metric projection method) to solve problem \eqref{eq-bound}. Note that $\nabla^2 F(y,z)$ is singular even if $\nabla^2 f(x)$ is positive definite.

\noindent\textbf{Limitation of ProjectionL1 when solving \eqref{eq-bound}}
%To illustrate the drawback, we write the index set \eqref{index} in two-metric projection method for \eqref{eq-bound} specifically. 
To illustrate the drawback, we present the formula when applying the two-metric projection method to \eqref{eq-bound}.  Recall the two-metric projection method in \eqref{alg: 2mproj}-\eqref{Dk-rules}. Suppose that $(y_k; z_k)$ denotes the $k$th iterate.  At $k$th iteration,  $I_{k,y}^{-}$ and $I_{k,z}^{-}$ denote the ``$I_k^-$ part'' corresponding to $y_k$ and $z_k$ respectively.
Let $x_k \triangleq y_k - z_k$. 
Then given an $\epsilon_k > 0$, we have
\begin{align}\label{def: ind.2m.1}
    I_{k,y}^- & \triangleq \{i\mid 0 \leq (y_{k})^i\leq \epsilon_k, \nabla_i f(x_k)+\gamma \leq 0 \} \cup \{i\mid (y_{k})^i > \epsilon_k \}, \\
    \label{def: ind.2m.2}
    I_{k,z}^- & \triangleq \{i\mid 0 \leq (z_{k})^i\leq \epsilon_k, -\nabla_i f(x_k)+\gamma \leq 0 \} \cup \{i\mid (z_{k})^i > \epsilon_k \}.
\end{align}
Note that $I_{k,y}^-$ and $I_{k,z}^-$ are subsets of $\{1,2,\ldots,n\}$. $I_{k,y}^- \cap I_{k,z}^- \ne \emptyset$ can happen.
%happens if $(x_{k}^{+})^i > \epsilon_k, (x_{k}^{-})^i\leq \epsilon_k, -\nabla_i f(x_k)+\gamma \leq 0$ or $(x_{k}^{-})^i > \epsilon_k, (x_{k}^{+})^i\leq \epsilon_k, \nabla_i f(x_k)+\gamma \leq 0$.

To get fast covergence,  non-diagonal part of $D_k$, i.e., $[D_k]_{I_k^- \times I_k^-}$ is chosen as inverse of the corresponding Hessian submatrix $[\nabla^2 F(y_k,z_k)]_{I_k^- \times I_k^-}$.  However, this submatrix can be singular even if $\nabla^2 f$ is positive definite.  What is worse, in this case the corresponding linear system ($ [\nabla^2 F(y_k,z_k)]_{I_k^- \times I_k^-} p = [\nabla F(y_k,z_k)]_{I_k^-}$) does not have a solution for any $\gamma > 0$, $(y_k,z_k)$, and the function $f$. More specifically, the linear system can be written as follows using $I_{k,y}^-$, $I_{k,z}^-$, $p \triangleq (p_y;p_z)$,
%According to the construction of $D_k$ in \eqref{Dk-rules}, if we use the Hessian information of $F(y_k, z_k)$ to build $D_k$ , we have to solve the following Newton equation in the two-metric projection method:
\begin{align}
    \begin{bmatrix}
       [\nabla^2 f(x_k)]_{I_{k,y}^- \times I_{k,y}^-} & -[\nabla^2 f(x_k)]_{I_{k,y}^-\times I_{k,z}^-}\\
       -[\nabla^2 f(x_k)]_{I_{k,z}^- \times I_{k,y}^-} & [\nabla^2 f(x_k)]_{I_{k,z}^- \times I_{k,z}^-}
      \end{bmatrix}
      \begin{bmatrix}
        p_y \\
        p_z
       \end{bmatrix} = 
       \begin{bmatrix}
        [\nabla f(x_k)]_{I_{k,y}^-} +\gamma {\mathbf{1}}_{|I_{k,y}^-|} \\
        - [\nabla f(x_k)]_{I_{k,z}^-} +\gamma {\mathbf{1}}_{|I_{k,z}^-|}
       \end{bmatrix}.  \label{newton-eq}
\end{align}

Suppose that $I_{k,y}^- \cap I_{k,z}^- \ne \emptyset $,  and $i \in I_{k,y}^- \cap I_{k,z}^-$. Then the linear system (\ref{newton-eq}) include the following two equations,
\begin{align*}
    & [[\nabla^2 f(x_k)]_{i \times I_{k,y}^{-} }] p_y - [[\nabla^2 f(x_k)]_{i \times I_{k,z}^{-} }] p_z = [\nabla f(x_k)]_{i} + \gamma,\\
     -&[[\nabla^2 f(x_k)]_{i \times I_{k,y}^{-} }] p_y + [[\nabla^2 f(x_k)]_{i \times  I_{k,z}^{-} }] p_z = -[\nabla f(x_k)]_{i} + \gamma,
\end{align*}
which implies that $\gamma = 0$. This contradicts the assumption that $\gamma > 0$. Therefore, when $I_{k,y}^{-} \cap I_{k,z}^{-} \ne \emptyset $, the linear system (\ref{newton-eq}) is unsolvable.  This defect in the algorithm design is intrinsic. Note that we may add perturbation to the singular Hessian or use other approximation approaches for salvage. However, since the original linear
system may not have a solution, the approximation approaches suffer from numerical instability.  An argument leveraging contradiction will show that when the perturbation goes to $0$, the solution must diverge.  More specifically,  suppose that \eqref{newton-eq} does not have a solution and for any $t = 1,2,...$, $[D_{k,t}]_{I_k^- \times I_k^-} p_{k,t} = [\nabla F(y_k,z_k)]_{I_k^-}$.   Let $[D_{k,t}]_{I_k^- \times I_k^-} {\longrightarrow} [\nabla^2 F(y_k,z_k)]_{I_k^- \times I_k^-}$ as $t \to \infty$. Then we must have $\| p_{k,t} \| \to \infty$.  Otherwise there exists a cluster point of $\{ p_{k,t} \}_t$ denoted as $p_k^*$ which is the solution of \eqref{newton-eq} by continuity. This is a contradiction.

From the above discussion, it is clear that the two-metric projection method has an intrinsic flaw when it is applied to solve the equivalent bound-constrained problem \eqref{eq-bound}.  Such a flaw will cause numerical instability if close approximation of the Hessian matrix becomes necessary, for e.g.,  when solution of high accuracy is desired within a reasonable amount of time. Such a limitation motivates us to consider a more robust and stable alternative.
%From the above discussion, it is clear that the two-metric projection method only suitable for problem \eqref{bdmin} and
%directly applying this method to solve the equivalent bound-constrained problem \eqref{eq-bound} will lead to numerical instability.
%This is mainly due to the fact that
%we have increased the problem size from $n$ to $2n$ and the index selection strategy in the two-metric projection method.

\section{Two-Metric Adaptive Projection Method}\label{sec: tmap}
In this section,  we propose a novel two-metric adaptive projection method that solves problem~\eqref{l1min} directly, without resorting to the intermediate bound-constrained reformulation~\eqref{eq-bound}, thereby circumventing the issue discussed in Section~\ref{sec: prelim}. 

The outline of our algorithm is as follows:
\begin{align}\label{alg: 2map0}
x_{k+1} \triangleq P_k(x_k - t_k p_k  ).
\end{align}
In \eqref{alg: 2map0}, $P_k$ is an adaptive operator since its definition depends on the index partition at iteration $k$. %(more accurately,  the definition depends on $f$, $\gamma$, and algorithm parameter $\epsilon$);
$p_k \in \bR^n$ is the descent direction which can be constructed via the gradient and the Hessian information. $t_k$ is the stepsize to be determined by a backtracking linesearch scheme with a new rule.  Next we will explain how to divide the index set and calculate each element of \eqref{alg: 2map0}.

\begin{itemize}
\fontsize{11}{15}\selectfont
\item[(i).] {\bf Partition of $\{1,...,n\}$}.  We design a new index partition strategy to directly solve \eqref{l1min}.  Given $x_k \in \mathbb{R}^n$, accuracy level $\epsilon > 0$, denote
\begin{equation*}
    \pi _k \triangleq \| x_k - \operatorname{prox}_{h} (x_k - \nabla f(x_k)) \|\text{ and } \epsilon_k \triangleq \min\{ \epsilon, \pi_k \}. 
\end{equation*}
%Let $g_{k} \triangleq \nabla f\left(x_{k}\right)$.  
Then we have
\begin{align}
    I_k^{-+} & \triangleq \left\{ i: \begin{array}{l}
         x_k^i > \epsilon_k \mbox{ or} \\
         0 \le x_k^i \le \epsilon_k, \nabla_i f(x_k) \le - \gamma
    \end{array} \right\}, \label{index set: I_k-+} \\
    I_k^{--} & \triangleq \left\{ i: \begin{array}{l}
         x_k^i < -\epsilon_k \mbox{ or} \\
         -\epsilon_k \le x_k^i \le 0, \nabla_i f(x_k) \ge \gamma
    \end{array} \right\},\label{index set: I_k--} \\
    \label{def: ind.2map.-}
    I_k^- & \triangleq   I_k^{-+} \cup I_k^{--} \\
    \label{def: ind.2map.+}
    I_k^+ & \triangleq \left\{ i : \begin{array}{l} | x_k^i | \le \epsilon_k, | \nabla_i f(x_k) | < \gamma \;\mbox{or}\\
    -\epsilon_k \le x_k^i < 0, \nabla_i f(x_k) \le - \gamma \; \mbox{or} \\
    0 < x_k^i \le \epsilon_k, \nabla_i f(x_k) \ge \gamma
    \end{array} \right\} = (I_k^-)^c.
\end{align}
If we compare these to the definitions of $I_{k,y}^-$ and $I_{k,z}^-$ in \eqref{def: ind.2m.1}, \eqref{def: ind.2m.2}, and let $y_k=\max (0, x_k) $, $ z_k=-\min (0, x_k) $ and $\epsilon_k$ be the same, then we have
\begin{equation*}
    I_k^{-+} \subseteq I_{k,y}^-, \quad I_k^{--} \subseteq I_{k,z}^-.
\end{equation*}
It is worth noting that we always have $I_k^{-+} \cap I_k^{--} = \emptyset $ in our method.
\item[(ii).] {\bf Calculation of $p_k$}. Given the index sets $I_k^+$ and $I_k^-$ defined in \eqref{def: ind.2map.-} and \eqref{def: ind.2map.+}, we let
\begin{align}\label{def: pk}
p_k \triangleq D_k( \nabla f(x_k) + \omega_k + r_k).
\end{align}
In \eqref{def: pk},  $D_k$ is positive definite. Inspired by the two-metric projection method, we scale the gradient component in $I_k^+$ using a diagonal matrix and the component in $I_k^-$ using a possibly non-diagonal one. More specifically,  let
\begin{align}\label{def: Dk}
\begin{aligned}
& [D_k]_{I_k^- \times I_k^-} \triangleq (H_k + \mu_k I)^{-1}, \; [D_k]_{I_k^+ \times I_k^+} \triangleq I,  \; [D_k]_{I_k^+ \times I_k^-}  \triangleq 0, \; [D_k]_{I_k^- \times I_k^+}  \triangleq 0, \\
& H_k \succeq 0, \; \| H_k \| \le \lambda_H,  \; \forall k,
\end{aligned}
\end{align}
and $\mu_k$ is a positive regularization parameter defined as
\begin{equation}\label{def: muk}
    \mu_k \triangleq c \left\| \begin{bmatrix}
        [x_k - \operatorname{prox}_h(x_k - \nabla f(x_k)]_{I_k^+} \\
        [\nabla f(x_k) +\omega_k]_{I_k^-}
       \end{bmatrix} \right\|^\delta,  \; c > 0, \; \delta \in (0,1).
\end{equation}
We implicitly assume that $\mu_k > 0$ and this is reasonable. Indeed, if $\mu_k = 0$, then $\| x_k - {\rm prox}_h(x_k - \nabla f(x_k)) \| = 0$ and $x_k$ is a stationary point. In practice,  the algorithm would have stopped.  In the original two-metric projection method,  it is assumed that the sequence of $D_k$ has a uniform upper bound on their operator norms.  This assumption prohibits using the inverse of a close approximation of a singular Hessian to build $D_k$, which is relevant in a host of applications featuring singular Hessians.  In contrast, our requirements on $D_k$ \eqref{def: Dk} allow $\| D_k \|$ to blow up, for e.g., when $H_k$'s are singular and $\mu_k \to 0$.  This addresses the singularity issue during the implementation and potentially enlarges the problem class which our methods can efficiently solve to high accuracy. 

$\omega_k$ in \eqref{def: pk} can be seen as an approximation of the subgradient of $\partial h(x_k)$ and defined as follows: for any $i$,
\begin{equation}\label{def: wk}
    \omega^i_{k} \triangleq
    \left\{\begin{array}{ll}
    \gamma, & i \in I_k^{-+}\\
   -\gamma, & i \in I_k^{--}\\
        0, & i \in I_k^+
    \end{array}\right. .
\end{equation}
In fact,  $ [\omega_{k}]_{I_k^-}$ belongs to the subdifferential of $\| [x_k]_{I_k^-} \|_1$, and if $\epsilon_k = 0$, then $\omega_k \in \partial \| x_k \|$.  

Last, $r_k$ in \eqref{def: pk} is the error when computing $p_k$ by potentially solving a linear system and satisfies the following:
\begin{equation}
    \label{residual}
\| [r_k]_{I_k^-} \| \leq \tau \min \left\{\mu_k \|[p_{k}]_{I_k^-}\|, \| [\nabla f(x_k) +\omega_{k}]_{I_k^-} \| \right\}, \;  [r_k]_{I_k^+}  \triangleq 0.
\end{equation}
Note that \eqref{residual} is checkable. Calculation of the $I_k^+$ part is relatively easy so it is exact.
\item[(iii).] {\bf Adaptive projection}.  The adaptive projection operator $P_k$ in \eqref{alg: 2map0} is defined as:
\begin{equation}
    \label{def: Pk}
    (P_{k}(v))^i =
    \left\{\begin{array}{ll}
            \max\{ v^i, 0\}, & i \in I_k^{-+}\\
       \min\{v^i,0\}, & i \in I_k^{--}\\
    {\rm prox}_{t_k \gamma | \cdot | }(v^i), & i \in I_k^+
    \end{array}\right..
\end{equation}
In a nutshell,  {\it the two-metric adaptive projection method combines two types of updates: on the $I_k^-$ part, it takes a gradient step scaled by a matrix and projects onto the corresponding orthant; on the $I_k^+$ part, it takes a proximal gradient step. }
\item[(iv).] {\bf Linesearch}.  At $k$th iteration, the trajectory of the trial point as a function of the stepsize $t$ can be written as:
\begin{equation}
    \label{update-x}
    x_{k}(t)= P_k(x_k - t p_k ),
\end{equation}
where $P_k$ and $p_k$ are defined in \eqref{def: pk} - \eqref{def: Pk}.

Denote $g_k \triangleq \nabla f(x_k)$.  Let  $\bar{g}_{k}$, $\bar{p}_{k}$, $\bar{\omega}_{k}$ and $\bar{r}_k$ be the subvectors of  $g_{k}$, $p_{k}$, $\omega_{k}$ and $r_k$  w.r.t index set  $I_{k}^{-}$.
Let $x_k^+$, $x_k^+(t)$, $g_k^+$ denote the subvectors of $x_k$, $x_k(t)$, $g_k$ w.r.t index set $I_k^+$, respectively.
 
The gradient map at iteration $k$ is defined as follows:
\begin{equation}\label{gradient map}
        G_{t}(x_k^+; g_k^+) \triangleq \frac{x_k^+ - {\rm prox}_{t \gamma \| \cdot \|_1 }(x_k^+ - t g_k^+) }{t} =\frac{x_k^+ - x_k^+(t)}{t},
\end{equation}
where the second equality holds due to the formula for $x_k(t)$.  Then, the stepsize $t_k$ is determined by a backtracking linesearch to ensure sufficient descent of the objective function.
\begin{equation}
    \psi \left(x_{k}\right)-\psi \left(x_{k}\left(t_k\right)\right) 
            \geq \sigma t_k (1-\tau) \mu_k \|\bar{p}_{k}\|^2 +\sigma t_k\|G_{t_k}(x_k^+; g_k^+) \|^{2}
\end{equation}
\end{itemize}

The algorithm can be formally stated as follows. In the next section, we will further show that Algorithm~\ref{2m-proj} is well-defined, convergent and enjoys fast convergence rate under certain assumptions.

\begin{algorithm}
\caption{Two-metric Adaptive Projection}
    \label{2m-proj}
\begin{algorithmic}[1]
\Require accuracy level $\epsilon > 0$, step acceptance parameter $\sigma \in(0,1)$, backtracking parameter $\beta \in(0,1)$, parameter $\tau \in(0,1)$ and initial point $x_{0}$.  %Stopping tolerance $\texttt{tol} > 0$.
\For{$k=0,1,\ldots,$}
    \State Compute $g_{k} \triangleq \nabla f\left(x_{k}\right)$, $\pi_k \triangleq \| x_k - \operatorname{prox}_{h} (x_k - g_k) \| $ and $\epsilon_k \triangleq \min\{\epsilon,  \pi_k \}$.
%    \If{ $\epsilon_k \le \texttt{tol} $ 
%    }
%    \State STOP and output $x_k$. 
%    \EndIf
    \State Construct index sets $I_k^+$, $I_k^-$,  $I_k^{-+}$, $I_k^{--}$ using \eqref{index set: I_k-+} to \eqref{def: ind.2map.+}.
    \State Compute step $p_{k}$ and scalar $\mu_k$ via \eqref{def: pk} to \eqref{residual}.
    \State Define $x_{k}(t)$ via \eqref{def: Pk} and \eqref{update-x}. Define $G_t(x_k^+; g_k^+)$ as in \eqref{gradient map}. 
    \State Let $t_k = 1$
    \If{$\psi \left(x_{k}\right)-\psi \left(x_{k}\left(t_k\right)\right) 
            < \sigma t_k (1-\tau) \mu_k \|\bar{p}_{k}\|^2 +\sigma t_k\| G_t(x_k^+; g_k^+) \|^{2}$}
    \State $t_k = \beta t_k$
    \Else
    \State break
    \EndIf
    \State $x_{k+1} = x_{k}\left(t_k \right)$
\EndFor
\end{algorithmic}
\end{algorithm}

\section{Convergence property}\label{sec: conv}
 In this section, we discuss the theoretical guarantees/advantages of Algorithm~\ref{2m-proj}. These results demonstrate its
 correctness and provide perspectives of why it works well in practice. The following assumption will be used throughout this section.
\begin{assumption}\label{As: basic}
Suppose that the following holds for problem \eqref{l1min} and Algorithm~\ref{2m-proj}:
    \begin{enumerate} \fontsize{11}{15}\selectfont
        %\item $h$ is convex and lower semi-continuous (i.e. closed);
        \item $\psi(x)$ has bounded sublevel sets. In particular,  the sublevel set $\mathcal{L}_{\psi}\left(x_{0}\right) \triangleq \{ x \mid \psi(x) \le \psi(x_0) \}$ is bounded.
        %\item $\psi(x)$ is bounded below by $\underline{\psi}$;
        \item $f$ is $L_g$-Lipschitz smooth over a convex open neighbourhood containing the sublevel set $\mathcal{L}_{\psi}\left(x_{0}\right)$ and trial points of Algorithm~\ref{2m-proj}.
    \end{enumerate}
\end{assumption}

\begin{remark}
(i).
    $\psi = f+h$ is continuous so every sublevel set is closed thus compact under Assumption~\ref{As: basic}. 
    Therefore,  the optimal solution set $X^*$ is nonempty and compact.  $\psi$ is bounded below by some constant $\underline{\psi}$.  (ii).  Assumption~\ref{As: basic} holds for a host of applications including $\ell_1$-norm regularized logistic regression and LASSO.  Specifically,  if $f$ is lower bounded by $\underline{f}$, then $\{ x \mid \psi(x) \le \alpha \} \subseteq \{ x \mid \| x \|_1 \le (\alpha - \underline{f})/\gamma \} $. The latter is bounded for any $\alpha$.  Assumption~\ref{As: basic}(2) holds if $f$ is Lipschitz smooth. 
\end{remark}

 First of all, we need to show that the algorithm is well-defined. Specifically, we need to show that the backtracking linesearch (inner loop of Algorithm~\ref{2m-proj}) can stop.  We begin with the following two lemmas which hold under Assumption~\ref{As: basic}.

 \begin{lemma}
    \label{lm: tbd+}
  Consider the $k$th iteration of Algorithm~\ref{2m-proj}.  If $t \leq \frac{ 2(1-\sigma)}{L_g}$, we have
    \begin{equation*}
        \psi(x_k(t)) - \psi(x_k) \leq \sum_{i \in I_k^-} \left( (g_k^i+\omega_{k}^i)(x_k^i(t)-x_k^i)+\frac{L_g}{2}(x_k^i(t)-x_k^i)^2 \right)-\sigma  t \|G_t (x_k^+;g_k^+)\|_{2}^{2}.
    \end{equation*}
\end{lemma}
\begin{proofs}
    According to (\ref{gradient map}), we have
    \begin{align*}
        x_k^+(t) = S_{t \gamma}(x_k^+ - t g_k^+) = x_k^+ -t G_{t}(x_k^+;g_k^+).
    \end{align*}
    Then by the definition of proximal operator, we have
    \begin{align}
    \notag
       & 0 \in x_k^+(t) - (x_k^+ - t g_k^+) + t \partial \gamma \|x_k^+(t)\|_{1} \\
       \label{inclus0}
       \implies & G_t(x_k^+;g_k^+) - g_k^+ \in \partial \gamma \|x_k^+(t)\|_{1}.
    \end{align}
    Without loss of generality, we assume that
    \begin{align*}
            I_k^+ = \left\{1,2,\ldots,n_1\right\},\;
            I_k^{-+} = \left\{n_1+1,\ldots,n_2\right\},\;
            I_k^{--} = \left\{n_2+1,\ldots,n\right\}.
    \end{align*} 
    Therefore, by $h(x_k(t)) = \gamma \|x_k^+(t)\|_{1} + \gamma \sum_{i \in I_k^{-+}} | x_k^i(t) | + \gamma \sum_{i \in I_k^{--}} | x_k^i(t) |$ and the fact that
    for any $i \in I_k^{-+}$, $x^i_k(t) \geq 0$ and for any $i \in I_k^{--}$, $x^i_k(t) \leq 0$, we have
%    \begin{equation*}
%        \begin{pmatrix}  
%            G^k_t(x_k^+)-g_k^+  \\  
%          \gamma \\
%            \vdots  \\  
%            \gamma \\
%          -\gamma \\
%            \vdots  \\  
%            -\gamma 
%          \end{pmatrix}  \in \partial h(x_k(t))
%    \end{equation*}
    \begin{equation*}
       \left(
            ( G_t(x_k^+;g_k^+)-g_k^+)^T ,  
          \gamma,
            \hdots ,
            \gamma,
          -\gamma,
            \hdots,
            -\gamma 
          \right)^T \in \partial h(x_k(t)).
    \end{equation*}
    Since $h(x)$ is convex, then
    \begin{align}\label{ineq: hc}
    \begin{aligned}
&        h(x_k)  \geq h(x_k(t))  \\
   &     + \left( G_t(x_k^+;g_k^+)-g_k^+ \right)^{T}t G_t(x_k^+;g_k^+)+\gamma \sum_{i \in I_k^{-+}}(x_k^i -x_k^i(t))-\gamma \sum_{i \in I_k^{--}}(x_k^i -x_k^i(t)).
   \end{aligned}
    \end{align}
    Therefore, by Lipschitz continuity of $f$ and (\ref{ineq: hc}), we have
    \begin{align*}
    \psi(x_k(t)) & = f(x_k(t)) + h(x_k(t)) \\
    & \le f(x_k) + h(x_k) + g_k^T (x_k(t) - x_k) + \frac{L_g}{2} \| x_k(t) - x_k \|^2 \\
    & + \left( g_k^+ - G_t(x_k^+;g_k^+) \right)^{T}t G_t(x_k^+;g_k^+)  - \gamma \sum_{i \in I_k^{-+}}(x_k^i -x_k^i(t)) + \gamma \sum_{i \in I_k^{--}}(x_k^i -x_k^i(t)) \\
    & \leq \psi(x_k) - t ( G_{t}(x_k^+;g_k^+) )^T G_{t}(x_k^+;g_k^+) + \frac{L_g t^2}{2}\|G_t(x_k^+;g_k^+)\|_{2}^{2} \\
    & + \sum_{i \in I_k^-} \left( (g_k^i+\omega_{k}^i)(x_k^i(t)-x_k^i)+\frac{L_g}{2}(x_k^i(t)-x_k^i)^2 \right) \\
    & = \psi(x_k) + \left( \frac{L_gt^2}{2} - t \right) \|G_t (x_k^+;g_k^+)\|_{2}^{2} \\
    & + \sum_{i \in I_k^-} \left( (g_k^i+\omega_{k}^i)(x_k^i(t)-x_k^i)+\frac{L_g}{2}(x_k^i(t)-x_k^i)^2 \right).
    \end{align*}
    If $t \leq \frac{2(1-\sigma)}{L_g}$, then 
    \begin{equation*}
        \psi(x_k(t)) - \psi(x_k) \leq \sum_{i \in I_k^-} \left( (g_k^i+\omega_{k}^i)(x_k^i(t)-x_k^i)+\frac{L_g}{2}(x_k^i(t)-x_k^i)^2 \right)-\sigma  t \|G_t (x_k^+;g_k^+)\|_{2}^{2}.
    \end{equation*}
This concludes the proof.
\end{proofs}

We assume that ${g_k}$ is bounded without loss of generality and will demonstrate this in Remark~\ref{rmk: bdg} after Theorem~\ref{thm: suffdescent}.  
Note that by \eqref{def: pk} - \eqref{residual},
    \begin{align*}
        \|\bar{p}_k\| & = \|(H_k + \mu_k I)^{-1}(\bar{g}_k+\bar{\omega}_{k} + \bar{r}_k)\| \le \mu_k^{-1} (\| \bar{g}_k + \bar{\omega}_{k} \| + \| \bar{r}_k \| ) \\
        & \leq c^{-1} \|\bar{g}_k+\bar{\omega}_{k}\|^{1-\delta} + \tau \|\bar{p}_k\|,
    \end{align*}
    which indicates that
    \begin{equation}\label{ineq: pkbd}
        \|\bar{p}_k\| \leq \frac{c^{-1}\|\bar{g}_k+\bar{\omega}_{k}\|^{1-\delta}}{1-\tau}.
    \end{equation}
Therefore, we let $G$ be a constant such that for any $k$,
\begin{align}\label{ineq: pkbd2}
    \|\bar{p}_k\| \le G.
\end{align}

\begin{lemma}\label{lm: tbd-}
    Consider the $k$th iteration of Algorithm~\ref{2m-proj}.  If $t$ satisfies $$t \leq \min\left(\frac{ 2(1-\sigma)(1-\tau) \mu_k}{L_g}, \frac{\epsilon_k}{G} \right) $$ then we have:
    \begin{equation*}
    \sum_{i \in I_k^-} \left( (g_k^i+\omega_{k}^i)(x_k^i(t)-x_k^i)+\frac{L_g}{2}(x_k^i(t)-x_k^i)^2 \right) \leq -\sigma t (1-\tau) \mu_k \|\bar{p}_{k}\|^2.
    \end{equation*}
\end{lemma}
\begin{proofs}
    If $x_k$ is critical, the result is immediate. Now suppose that $x_k$ is not critical. Then $\epsilon_k > 0 $.\\
    Case 1: For $x_k^i > \epsilon_k$ or $x_k^i < -\epsilon_k$, we have 
    \begin{equation*}
        x_k^i(t) = x_k^i- t p_k^i, \quad \text{if } t \leq \frac{\epsilon_k}{G}.
    \end{equation*}
    Therefore, 
    \begin{equation*}
        (g_k^i+\omega_{k}^i)(x_k^i(t)-x_k^i)+\frac{L_g}{2}(x_k^i(t)-x_k^i)^2 = \frac{L_g}{2}(t p_k^i)^2 -(g_k^i+\omega_{k}^i)(t p_k^i).
    \end{equation*}
    Case 2: For $0 \le x_k^i \le \epsilon_k, g_k^i \le - \gamma$,  we have $i \in I_k^{-+}$, $x_k^i(t) \triangleq \max\{ x_k^i - tp_k^i,  0 \}$, and $\omega_k^i = \gamma$.  Therefore,
    \begin{equation*}
        x_k^i(t)\geq x_k^i- t p_k^i \Longrightarrow (g_k^i+\omega_{k}^i)(x_k^i(t)-x_k^i) \leq -t p_k^i(g_k^i+\omega_{k}^i).
    \end{equation*}
    Since $| x_k^i(t) - x_k^i| \leq |t p_k^i|$, we also have
    \begin{equation*}
        \frac{L_g}{2}(x_k^i(t)-x_k^i)^2 \leq \frac{L_g}{2}(t p_k^i)^2.
    \end{equation*}
    Therefore, in this case
    \begin{equation*}
        (g_k^i+\omega_{k}^i)(x_k^i(t)-x_k^i)+\frac{L_g}{2}(x_k^i(t)-x_k^i)^2 \leq \frac{L_g}{2}(t p_k^i)^2 -(g_k^i+\omega_{k}^i)(t p_k^i).
    \end{equation*}
    Case 3: For $-\epsilon_k \le x_k^i \le 0, g_k^i \ge \gamma$, similar to Case 2 we have
    \begin{equation*}
        (g_k^i+\omega_{k}^i)(x_k^i(t)-x_k^i)+\frac{L_g}{2}(x_k^i(t)-x_k^i)^2 \le \frac{L_g}{2}(t p_k^i)^2 -(g_k^i+\omega_{k}^i)(t p_k^i).
    \end{equation*}
    To conclude, for $i \in I_k^-$ we have,
    \begin{equation*}
        (g_k^i+\omega_{k}^i)(x_k^i(t)-x_k^i)+\frac{L_g}{2}(x_k^i(t)-x_k^i)^2 \le \frac{L_g}{2}(t p_k^i)^2 -(g_k^i+\omega_{k}^i)(t p_k^i),
    \end{equation*}
    i.e.,
    \begin{align*}
        \sum_{i \in I_k^-} \left( (g_k^i+\omega_{k}^i)(x_k^i(t)-x_k^i)+\frac{L_g}{2}(x_k^i(t)-x_k^i)^2 \right) \le \frac{L_g}{2}t^2 \| \bar p_k \|^2 - t(\bar g_k + \bar \omega_{k})^T \bar p_k.
    \end{align*}
    We consider $(\bar{g}_k+\bar{\omega}_{k})^T \bar{p}_{k}$,
    \begin{equation*}
        \begin{aligned}
            (\bar{g}_k+\bar{\omega}_{k})^T \bar{p}_{k} & =  \bar{p}_{k}^T (H_k + \mu_k I) \bar{p}_{k} - \bar{p}_{k}^T \bar{r}_k \geq \mu_k \|\bar{p}_{k}\|^2 - \|\bar{p}_{k}\|\| \bar{r}_{k}\| \\
            & \geq \mu_k \|\bar{p}_{k}\|^2 - \tau \mu_k \|\bar{p}_{k}\|^2 = (1-\tau)\mu_k \|\bar{p}_{k}\|^2.
        \end{aligned}
    \end{equation*}
    Therefore,
    \begin{equation*}
        \sum_{i \in I_k^-} \left( (g_k^i+\omega_{k}^i)(x_k^i(t)-x_k^i)+\frac{L_g}{2}(x_k^i(t)-x_k^i)^2 \right)  \leq \frac{L_g}{2}t^2\|\bar{p}_{k}\|^2 -t(1-\tau)\mu_k \|\bar{p}_{k}\|^2.
    \end{equation*}
    When $t \leq \frac{2(1-\sigma)(1-\tau) \mu_k}{L_g}$,
    \begin{equation*}
        \frac{L_g}{2}t^2\|\bar{p}_{k}\|^2 -t(1-\tau)\mu_k \|\bar{p}_{k}\|^2 = t\left(\frac{L_g}{2}t - (1-\tau)\mu_k \right)\|\bar{p}_{k}\|^2 \leq -\sigma t (1-\tau)\mu_k\|\bar{p}_{k}\|^2.
    \end{equation*}
    This concludes the proof.
\end{proofs}

These two lemmas immediately indicate the following theorem.  The result of Theorem~\ref{thm: suffdescent} guarantees that the linesearch in the inner loop of Algorithm~\ref{2m-proj} is well-defined.
 \begin{theorem}
    \label{thm: suffdescent}
    Consider the $k$th iteration of Algorithm~\ref{2m-proj}.  Suppose that Assumption~\ref{As: basic} holds. Then if $$t \leq \min \left(\frac{ 2(1-\sigma)}{L_g}, \frac{ 2(1-\sigma)(1-\tau) \mu_k}{L_g}, \frac{\epsilon_k}{G}\right), $$ we have
    \begin{equation}
        \label{armijo}
        \psi \left(x_{k}\right)-\psi \left(x_{k}\left(t\right)\right) 
            \geq \sigma t (1-\tau) \mu_k \|\bar{p}_{k}\|^2 +\sigma t\|G_t(x_k^+;g_k^+)\|^{2}.
    \end{equation}
    Therefore, $t_k > \beta \min \left\{\frac{ 2(1-\sigma)}{L_g}, \frac{ 2(1-\sigma)(1-\tau) \mu_k}{L_g}, \frac{\epsilon_k}{G}, 1\right\} $.
\end{theorem}

\begin{remark}\label{rmk: bdg}
    $x_0$ is in the bounded sublevel set $\mathcal{L}_\psi (x_0)$, so we can naturally find $G$ such that (\ref{ineq: pkbd2}) holds for $k = 0$. Therefore Theorem~\ref{thm: suffdescent} makes sense for $k = 0$ and we have $\psi(x_1) \le \psi(x_0)$, 
    which indicates that $x_1 \in \mathcal{L}_\psi (x_0)$. Then by induction, $\{x_k\}$ lies in the sublevel set and there exists $G < +\infty$ such that (\ref{ineq: pkbd2}) holds for any $k$.
\end{remark}

 \subsection{Global convergence}
To demonstrate the correctness of Algorithm~\ref{2m-proj} to solve problem \eqref{l1min}, we show that the limit point is always stationary regardless of the initial choice $x_0$.
 \begin{theorem}[Global Convergence]
    \label{global-convergence}
  Denote $\{ x_k \}$ as the sequence generated by Algorithm~\ref{2m-proj} and suppose the sequence does not contain a stationary point ($\mu_k >0$, $\forall k$).  Under Assumption~\ref{As: basic}, any limit point of $\{ x_k \}$ is a stationary point. 
%    \begin{equation*}
%        \lim_{k \to \infty} \|x_{k} - \operatorname{prox}_h(x_{k} - g_{k})\| = 0.
%    \end{equation*}
\end{theorem}

\begin{proofs}%[Proof of \ref{global-convergence}]
    By Assumption~\ref{As: basic} we have lower boundedness of $\psi(x)$.  This fact and Theorem~\ref{thm: suffdescent} indicate that $\psi(x_k) - \psi(x_{k+1}) \to 0$. Therefore,
    \begin{equation}
        \label{limit-linesearch}
        \lim_{k \to \infty} t_k \mu_k \|\bar{p}_{k}\|^2=0
        \text{ and }
        \lim_{k \to \infty} t_k\|G_{t_k}(x_k^+;g_k^+)\|_{2}^{2}=0.
    \end{equation}
    Suppose that $x^*$ is a limit point but not stationary, therefore there exists a subsequence $\left\{x_{k_m}\right\}$ converging to $x^*$. 
    We first show that $\left\{t_{k_m}\right\}$ is bounded away from zero. For $i \in I_k^-$ and $t$ small enough, we have
    \begin{align}\label{ineq: beck}
        g_k^i + \omega_{k}^i = | x_k^i - \operatorname{prox}_{t \gamma | \cdot |}(x_k^i - t g_k^i) |/t \ge | x_k^i - \operatorname{prox}_{\gamma  | \cdot |}(x_k^i - g_k^i) |,
    \end{align}
    where the inequality holds because of Theorem 10.9 in \citep{beck2017first}.  Together with definition of $\mu_k$ in \eqref{def: muk}, \eqref{ineq: beck} indicates the following: for any $k$,
   \begin{align}\label{ineq: muk}
           \mu_k \geq c\|x_k - \operatorname{prox}_{\gamma  \| \cdot \|_1}(x_k - g_k)\|^\delta.
   \end{align}
Since $$\| x^* - \operatorname{prox}_{h} (x^* - \nabla f(x^*)) \| \ne 0,$$
\eqref{ineq: muk} holds and $\lim\limits_{m \to \infty} x_{k_m} = x^*$, there exist $\bar{\epsilon}, \bar{\mu}>0$ and $M$ large enough such that $\epsilon_{k_m} \ge \bar{\epsilon}$,  $\mu_{k_m} \ge \bar{\mu}$,  $\forall m \ge M$. 
    Therefore, we have that $\left\{t_{k_m}\right\}$ is bounded away from zero by Theorem~\ref{thm: suffdescent}. 
    Then, by (\ref{limit-linesearch}) we have
    \begin{equation*}
        \lim_{m \to \infty} \|\bar{p}_{k_m}\| = 0\;
    \mbox{and}\;
        \lim_{m \to \infty} \|G_{t_{k_m}}(x_{k_m}^+;g_{k_m}^+)\| = 0.
    \end{equation*}   
From \eqref{ineq: beck} and \eqref{def: pk} - \eqref{residual}, we know that
    \begin{equation*}
        \|[x_k - \operatorname{prox}_h(x_k - g_k)]_{I_k^-}\| \leq\|\bar{g}_k+\bar{\omega}_{k}\| = \| (H_k + \mu_k I)\bar{p}_k - \bar{r}_k \| \leq (\lambda_H+\mu_k+ \tau \mu_k) \|\bar{p}_k\|.
    \end{equation*}
Also, by \citep[Theorem 10.9]{beck2017first}, we have
    \begin{equation*}
        \|[x_k - \operatorname{prox}_h(x_k - g_k)]_{I_k^+}\| \leq\|G_{t_k}(x_k^+;g_k^+)\|.
    \end{equation*}
    Thus,
    \begin{equation*}
        \lim_{m \to \infty} \|x_{k_m} - \operatorname{prox}_h(x_{k_m} - g_{k_m})\| = 0,
    \end{equation*}
    which indicates that $x^*$ is a stationary point. Contradiction.
\end{proofs}
%\begin{remark}\label{rmk: beck}
%%    For $i \in I_k^-$ and $t$ small enough, we have
%%    \begin{align*}
%%        g_k^i + \omega_{k}^i = | x_k^i - \operatorname{prox}_{t \gamma | \cdot |}(x_k^i - t g_k^i) |/t \ge | x_k^i - \operatorname{prox}_{\gamma  | \cdot |}(x_k^i - g_k^i) |,
%%    \end{align*}
%%    where the inequality holds because of Theorem 10.9 in \citep{beck2017first}. Therefore,
%By definition of $\mu_k$ in \eqref{def: muk}, \eqref{ineq: beck} indicates the following:
%   $
%        \mu_k \geq c\|x_k - \operatorname{prox}_{\gamma  \| \cdot \|_1}(x_k - g_k)\|^\delta.
%   $
%\end{remark}

\subsection{Manifold Identification and local superlinear convergence rate}

In this subsection,  we exploit manifold identification and Newton steps for acceleration.
%our discussion is on the premise that $f$ is convex and twice continuously differentiable.  Therefore, we can take $H_k = [\nabla^2 f(x_k)]_{I_k^- \times I_k^-} \succeq 0$. 
Under appropriate assumptions and choice of $H_k$, we show that Algorithm~\ref{2m-proj} is able to identify the active manifold of the optimal solution in finite number of steps. Such a property takes
the calculation of the Newton step to a potentially low-dimensional subspace, thus effectively accelerating convergence in iteration while maintaining scalability.  We specify the necessary assumptions as follows.
% According to Proposition \ref{prop: eb-ebman}, we have Assumption \ref{asp: ebman} holds, we call this condition EB throughout the paper.

\begin{assumption}[Strict complementarity or SC]\label{asp: sc}
A critical point $x$ satisfies the strict complementarity condition if and only if $0 \in \operatorname{relint}(\nabla f(x)+ \partial h(x))$. $\operatorname{relint} (S)$ denotes the relative interior of a set $S$.
\end{assumption}

\begin{remark}
Assumption~\ref{asp: sc} is very common in literature to guarantee manifold identification for active set methods, \citep{hare2004identifying,bareilles2023newton,lee2023accelerating,hale2008fixed,nocedal1999numerical}  to name a few.
It is an if and only if condition for the partially smooth functions to be identifiable \citep{lewis2022partial,lewis2024identifiability,hare2007identifying,drusvyatskiy2014optimality}.
Discussion on applicability of this assumption is ongoing for decades.  Some positive results include \citep{drusvyatskiy2011generic},  in which Assumption~\ref{asp: sc} is proven to hold almost surely under linear perturbation at minimizers of convex objective function, and \citep{hale2008fixed}, in which Assumption~\ref{asp: sc} is empirically observed for LASSO when the data matrix is randomly generated. 
\end{remark}

\begin{assumption}[Error bound or EB]\label{asp: eb}
    Given an optimal point $x^* \in X^*$, there exist a neighborhood $B\left(x^{*}, b_{1}\right)$ and a constant $c_1 > 0$ such that for all $x \in B\left(x^{*}, b_{1}\right)$,
    \begin{equation}
        \label{eq: eb}
        c_1 \|x - \operatorname{prox}_h(x-\nabla f(x))\| \geq \operatorname{dist}\left(x, X^{*}\right).
    \end{equation}
\end{assumption}

\begin{remark}[Comparison with Luo-Tseng EB Property]
    The EB condition in Assumption~\ref{asp: eb} is similar to the Luo-Tseng EB property \citep{luo1993error, tseng2009coordinate}, which has been widely used in the convergence rate analysis of first-order methods for structured optimization problems and is currently extended for analysis of second-order methods \citep{yue2019family, xiao2018regularized}.
    Assumption~\ref{asp: eb} was also used in the analysis of semi-smooth Newton methods \citep{hu2022local}.
    Luo-Tseng EB property requires that for any $\zeta \ge \psi(x^*)$, \eqref{eq: eb} holds for all $x$ in a test set $\{x \mid \psi(x) \leq \zeta , \|x - \operatorname{prox}_h(x-\nabla f(x))\| \leq \epsilon\}$, where $\epsilon>0$ and $c_1$ depend on $\zeta$. In the convex case and under mild conditions,  Luo-Tseng EB property could be checked by verifying the calmness of the solution map which is equivalent to \eqref{eq: eb} holding in a neighborhood of $X^*$ \citep{zhou2017unified},  so Luo-Tseng EB implies Assumption~\ref{asp: eb}.  Luo-Tseng EB property holds for a wide range of structured convex optimization problems, including LASSO and logistic regression with $\ell_1$-regularization.  More details can be found in \citep{tseng2009coordinate, zhou2017unified}.
\end{remark}

% EB condition has been used in the convergence rate analysis of various first-order and second-order methods. For example, it has been utilized to establish the linear convergence 
% of the projected gradient descent \citep{luo1993error} and proximal gradient descent \citep{hou2013linear}, the local superlinear convergence of the semi-smooth Newton method 
% \citep{hu2022local} and the inexact SQA \citep{yue2019family}.

\begin{assumption}[Error bound on Manifold]\label{asp: ebman}
Given an optimal point $x^* \in X^*$ and the corresponding $\sM_*$, there exist constants ${b}_1 > 0$ and ${c}_1 > 0$  such that for all $x \in \mathcal{M}_{*} \cap B\left(x^{*},  {b}_{1}\right)$,
    \begin{equation*}
        {c}_1 \|\operatorname{grad}_{\mathcal{M}_{*}} \psi(x)\| \geq \operatorname{dist}\left(x, X^{*} \cap \mathcal{M}_{*}\right).
    \end{equation*}
\end{assumption}
The relation between Assumption~\ref{asp: eb} and \ref{asp: ebman} is as follows.

\begin{theorem}
    \label{prop: eb-ebman}
%Suppose that $f$ is twice continuously differentiable around $x^*$.  When Assumption~\ref{As: basic} holds and 
When SC (Assumption~\ref{asp: sc}) holds at an optimal point $x^*$,  EB (Assumption~\ref{asp: eb}) at $x^*$ implies EB on manifold (Assumption~\ref{asp: ebman}) at $x^*$. 
\end{theorem}
\begin{proofs}
Recall that $\partial h(x^*) = \{ v \mid v^i = {\rm sgn}([x^*]^i) \gamma \mbox{ if } [x^*]^i \neq 0; v^i \in [-\gamma,\gamma] \mbox{ if } [x^*]^i = 0 \}$.  Then by SC, $\nabla_i f(x^*) \in (-\gamma, \gamma)$ if $[x^*]^i = 0$. By continuity of $\nabla f(x)$, there exists a neighborhood $B\left(x^{*}, 2b_1\right)$ (w.l.o.g we suppose $b_1$ is same as Assumption~\ref{asp: eb}) such that $ B\left(x^{*}, 2b_1\right) \cap X^* \subseteq \mathcal{M}_{*}$.
    Define $ \Pi_{X^{*}}(x) \triangleq   \operatorname{argmin}_{y \in X^{*}}\|y-x\|$, and let $\tilde{x} \triangleq \Pi_{X^{*}}(x)$. Then we have for all $x \in B\left(x^{*}, b_1\right)$,
    \begin{equation*}
        \|x - \tilde{x}\| = \operatorname{dist}\left(x, X^{*}\right) \leq \|x - x^*\| < b_1.
    \end{equation*}
    Therefore, $\|\tilde{x}-x^*\| \leq \|x - \tilde{x}\|+ \|x - x^*\| < 2b_1$ and $\tilde{x} \in \mathcal{M}_{*}$.
    Since
    \begin{equation*}
        \|x - \tilde{x}\| = \operatorname{dist}\left(x, X^{*}\right) \leq \operatorname{dist}\left(x, X^{*} \cap \mathcal{M}_{*}\right) \leq \|x - \tilde{x}\|,
    \end{equation*}
we have that $\operatorname{dist}\left(x, X^{*}\right) = \operatorname{dist}\left(x, X^{*} \cap \mathcal{M}_{*}\right)$. 
    If $x \in \mathcal{M}_{*}$ and $x$ is within a sufficiently small neighborhood of $x^*$ (w.l.o.g. suppose this neighborhood is $B\left(x^{*}, b_1\right)$), we have 
\begin{align*}
& i \in \mathcal{A}(x^*) \Leftrightarrow x^i = 0  \Leftrightarrow | x^i - \nabla_i f(x) | < \gamma, \\
& i \notin \mathcal{A}(x^*) \Leftrightarrow x^i \neq 0  \Leftrightarrow | x^i - \nabla_i f(x) | > \gamma,
\end{align*}    
by SC and continuity. Therefore, in this neighborhood we have
    $$\|x - \operatorname{prox}_h(x-\nabla f(x))\| = \left( \sum_{i \notin \mathcal{A}(x^*)} (\nabla_i f(x) + \gamma {\rm sign}(x^i))^2 \right)^{1/2}  = \|\operatorname{grad}_{\mathcal{M}_{*}} \psi(x)\|,$$ 
which concludes the proof.
\end{proofs}

\begin{assumption}[Lipschitz Hessian]\label{asp: liphess}
    Given an optimal point $x^*$, there exists a constant $b_2 > 0$ such that $\operatorname{Hess}_{\mathcal{M}_{*}} f$ is Lipschitz continuous over $B\left(x^{*}, b_2\right) \cap \mathcal{M}_{*}$.
\end{assumption}

\begin{remark}
    Assumption~\ref{asp: liphess} holds when $\nabla^2 f$ is Lipschitz continuous around $x^*$,  which can be derived from the definition \eqref{def: RHes} and nonexpansiveness of the projection operator. This assumption is crucial for establishing a fast convergence of Newton-type methods.  We would like to point out that $f$ in many practical problems has a locally Lipschitz continuous Hessian matrix, e.g., least squares, logistic regression and Poisson regression.
% If the smooth part $f$ satisfies additional structural assumptions, \citep{mordukhovich2023globally} shows that proximal Newton-type methods can achieve local superlinear convergence without Lipschitz continuity of Hessian.
%    However, we would like to point out that many practical problems that satisfy this structural assumptions also have local Lipschitz continuous Hessian, e.g., least squares, logistic regression and Poisson regression.
\end{remark}

The identification property and convergence guarantees of Algorithm~\ref{2m-proj} are as follows.
 \begin{theorem}
    \label{local-convergence-rate}
    Suppose that $f$ is convex and twice continuously differentiable.  Assumption~\ref{As: basic} holds. Let $ H_k \triangleq [\nabla^2 f(x_k)]_{I_k^- \times I_k^-} \succeq 0$ in Algorithm~\ref{2m-proj}. Assume that the sequence generated by Algorithm~\ref{2m-proj} admits a limit point $x^* \in X^*$ such that
    Assumption~\ref{asp: sc},~\ref{asp: eb}, \ref{asp: liphess} hold at $x^*$. Then after finite number of iterations the sequence $\left\{x_k\right\}$ lies in $\mathcal{M}_{*}$ and $x_k \rightarrow x^*$ superlinearly with parameter $1+\delta$.
\end{theorem}

Next we discuss Lemma~\ref{manifold-identification} - \ref{lm: Cauchy}, which can be viewed as six steps to prove the theorem.  We implicitly use the same assumptions and notations as in Theorem~\ref{local-convergence-rate} for these lemmas.

    \begin{lemma}
        \label{manifold-identification}
There exists a neighborhood $B(x^*, b_3)$ such that for all $x_k \in B(x^*, b_3)$,
        $I_k^+ = \mathcal{A}(x^*) = \mathcal{A}(x_{k+1})$, and
 \begin{align*}
            I_k^+ = \left\{ i \mid | x_k^i | \le \epsilon_k, | g_k^i | < \gamma \right\},\;
            I_k^{-+} = \left\{ i \mid x_k^i > \epsilon_k\right\},\;
                I_k^{--} = \left\{ i \mid x_k^i < -\epsilon_k \right\}.
        \end{align*}
    \end{lemma}

    \begin{remark}\label{rmk: wlog}
%        We then rewrite the index sets for $x_k \in B(x^*, b_3)$ as follows,
%        \begin{align*}
%            I_k^+ = \left\{ i \mid | x_k^i | \le \epsilon_k, | g_k^i | < \gamma \right\},\;
%            I_k^{-+} = \left\{ i \mid x_k^i > \epsilon_k\right\},\;
%                I_k^{--} = \left\{ i \mid x_k^i < -\epsilon_k \right\}.
%        \end{align*} 
    By Lemma~\ref{manifold-identification}, \eqref{optcon}, \eqref{ineq: pkbd} and continuity,  there exists a neighborhood $B(x^*, \bar{b}_3) \subseteq B(x^*, b_3)$ such that 
        $
            \|\bar{p}_k\| \leq \frac{b_3}{2}.
        $
        Therefore, for all $x_k \in B(x^*, \min\{\frac{b_3}{2}, \bar{b}_3\})$,
        \begin{equation*}
            \|x_{k+1}-x^*\| = \| [x_{k}-x^*]_{I_k^-} + t_k \bar{p}_{k}\| \le \| [x_{k}-x^*]_{I_k^-} \| + \| t_k \bar{p}_{k}\| \leq \frac{b_3}{2} + \frac{b_3}{2} = b_3.
        \end{equation*}
        Thus, for all $x_k \in B(x^*, \min\{\frac{b_3}{2}, \bar{b}_3\})$,
        \begin{equation*}
            I_k^+ = \mathcal{A}(x^*) = \mathcal{A}(x_{k+1}) = I_{k+1}^+ = \mathcal{A}(x_{k+2}),
        \end{equation*}
        which means that $x_{k+1}, x_{k+2} \in \mathcal{M}_*$.  With similar argument, we can show that for any positive $b \le b_3$, there exists $\bar b > 0$ such that if $x_k \in B(x^*,\bar b)$, 
        then $x_{k+1}, x_{k+2} \in \mathcal{M}_*$ and $x_{k+1} \in B(x^*,b)$.
    \end{remark}

    By Lemma~\ref{manifold-identification} and Remark~\ref{rmk: wlog}, if $x^*$ satisfies SC then for $x_k$ sufficiently close to $x^*$, our algorithm identifies the active set.  {Therefore, } Algorithm~\ref{2m-proj} generates infinitely many consecutive pairs of iterates on the manifold $\mathcal{M}_{*}$.

    By Theorem~\ref{prop: eb-ebman}, we know that if SC and EB hold at $x^*$, then the manifold EB condition holds at $x^*$.
    Thus, for Lemma~\ref{eb-pk} - \ref{eb-point}, we treat manifold EB at $x^*$ as a given assumption.  We also suppose that $x_k,x_{k+1} \in \mathcal{M}_*$.  Based on our discussions in Remark~\ref{rmk: wlog},  this is a reasonable assumption since such a pair exists in any small neighborhood of $x^*$.
    \begin{lemma}
        \label{eb-pk}
        There exists a neighborhood $B(x^*, b_4)$ with $b_4 \le \min\{b_1,b_2/2,b_3\}$ such that if $x_k \in B(x^*, b_4)$, then
        \begin{equation*}
            \|\bar{p}_k\| \leq \mathcal{O}(\operatorname{dist}\left(x_k, X^{*} \cap \mathcal{M}_{*}\right))
        \end{equation*}
    \end{lemma}

    Lemma~\ref{eb-pk} claims that $\bar{p}_k$ has a better upper bound than \eqref{ineq: pkbd} which will be used later.  The next lemma states that the linesearch of Algorithm~\ref{2m-proj} will eventually output a unit stepsize $t_k=1$.
    
    \begin{lemma}
        \label{line-search-tk}
        There exists a neighborhood $B(x^*, b_5)$ with $b_5 \le b_4$ such that if $x_k \in B(x^*, b_5)$,  then $t_k=1$.
    \end{lemma}

    From the above discussions, we know that Algorithm~\ref{2m-proj} will degenerate to an inexact Newton method on the manifold $\mathcal{M}_*$ at iteration $k$ for $x_k$ sufficiently close to $x^*$ given that $x_k,x_{k+1} \in \sM_*$ .  Leveraging this fact,  the next lemma shows that the distance between $x_{k+1}$ and the optimal set $X^*$ is bounded by the distance between $x_k$ and $X^*$ raised to the power of $1+\delta$.
    
    \begin{lemma}
        \label{eb-point}
If $x_k \in B(x^*, b_5)$, then
        \begin{equation*}
            \operatorname{dist}\left(x_{k+1}, X^{*} \cap \mathcal{M}_{*}\right) \leq \mathcal{O}(\operatorname{dist}\left(x_k, X^{*} \cap \mathcal{M}_{*}\right)^{1+\delta}).
        \end{equation*}
    \end{lemma}

    In the following lemma, we show that $x_{k}$ will be confined in a small neighborhood around $x^*$ for any $k$ large enough by induction.
    This indicates that Lemma \ref{manifold-identification} - \ref{eb-point} hold for any $k$ large enough.
    
    \begin{lemma}
        \label{approaching}
        For any $ b_6 \in (0, b_5)$, there exists a neighborhood $B(x^*, r)$ such that 
        for all $x_k \in B(x^*, r)$, we have
        \begin{equation*}
            x_{k+i} \in B(x^*, b_6) \quad \forall i = 1,2,\ldots
        \end{equation*}
    \end{lemma}

    Finally, these results indicate that $\{x_k\}$ is a Cauchy sequence.
    %thus converging sequentially to $x^*$ superlinearly with parameter $1+\delta$. We summarize the results in the following lemma.
    \begin{lemma}\label{lm: Cauchy}
        $\{x_k\}$ is a Cauchy sequence and hence converges to $x^*$ superlinearly with parameter $1+\delta$.
    \end{lemma}
        
\begin{remark}
We would like to emphasize that the assumptions of Theorem~\ref{local-convergence-rate} are either mild or commonly uesd in literature. In particular, convexity, second-order smoothness together with Assumption~\ref{As: basic}, Assumption~\ref{asp: eb} and \ref{asp: liphess} hold for a wide range of problems, e.g., least squares and logistic regression with $\ell_1$-regularization. Assumption~\ref{asp: eb} and \ref{asp: liphess} are also used in the SQA method \citep{yue2019family} and the semi-smooth Newton method \citep{hu2022local} to establish superlinear convergence.  Assumption~\ref{asp: sc} is also a common assumption in literature \citep{de2016fast,bertsekas1982projected} to guarantee manifold identification for active set methods.
\end{remark}

\section{Numerical experiment}\label{sec: num}

In this section, we study the practical performance of the two-metric adaptive projection method and compare it with some existing algorithms. All experiments are coded in MATLAB (R2024a) and run on a desktop
with a 9700X CPU and 32 GB of RAM. The source code can be downloaded from \url{https://github.com/Hanju-Wu/TMAP}. We will compare the performance of the two-metric adaptive projection method with some state-of-the-art algorithms on logistic regression problems and LASSO problems.

\subsection{Numerical results on logistic regression problems}
We consider the ($\ell_1$-norm regularized)
logistic regression problem:
\begin{equation*}
    \min _{x \in \mathbb{R}^{n}}\frac{1}{m} \sum_{i=1}^{m} \log \left(1+\exp \left(-b_{i} \cdot a_{i}^{T} x\right)\right)+\gamma\|x\|_{1} .
\end{equation*}
Here, $a_1, a_2, \ldots, a_m \in \mathbb{R}^n$ are given data samples; $b_1, b_2, \ldots, b_m \in \{-1,1\}$ are given
labels; $\gamma$ is a given regularization parameter. Let $A \triangleq [a_1,a_2,...,a_m]^T$ denotes the data matrix. This problem is a classic classification model in machine learning. It is a standard problem
for testing the efficiency of different algorithms in solving composite optimization problems. In our experiment, we use the LIBSVM data sets\footnote{\url{https://www.csie.ntu.edu.tw/~cjlin/libsvmtools/datasets/}} \citep{chang2008libsvm}: \textbf{rcv1.train}, \textbf{rcv1.test}, \textbf{news20} and \textbf{real-sim}, and set $\gamma \triangleq \frac{1}{m}$ same as the setting in \citep{lee2023accelerating}.
The sizes of these data sets are listed in Table~\ref{tab:data}.
\begin{table}[h] \fontsize{11}{15}\selectfont
	\centering
	\caption{\fontsize{11}{15}\selectfont 
	Tested data sets.} 
	\label{tab:data}
        \begin{tabular}{c c c c c c}
		\hline
		& Data set & $n$ & $m$ & nnz of $A$ & sparsity of $x^*$\\ 
		\hline \hline
		\multirow{2}{*}{$n>m$}& \textbf{rcv1.train} & 47,236 & 20,242 & 1,498,952 & 1.18\% \\ 
		& \textbf{news20} & 1,355,191 & 19,996 &  9,097,916 & 0.0373\% \\  \hline \hline
        %& \textbf{colon-cancer} & 2,000 & 62 &  124,000 & 1.3\% \\ \hline \hline
        \multirow{2}{*}{$n<m$}&\textbf{real-sim} & 20,958 & 72,309 & 3,709,083 & 8.18\% \\  
        & \textbf{rcv1.test} & 47,236 & 677,399 & 49,556,258 & 10.79\% \\ \hline
		\hline
	\end{tabular}
    nnz represents number of nonzeros in data matrix $A$.
\end{table}

It is worth mentioning the computation of the linear system \eqref{def: pk}, \eqref{def: Dk}.  We let $H_k \triangleq [\nabla^2 f(x_k)]_{I_k^- \times I_k^-}$.  Yet we do not need to calculate Hessian matrix exactly. We only need to calculate Hessian-vector products when applying conjugate gradient method to solve the linear equations.

For the logistic loss function $f$, we have $\nabla^2 f(x) = \frac{1}{m}A^{T}DA$, where $D$ is a diagonal matrix with:
\begin{align*}
    D_{i i} =\frac{\exp \left(-b_{i} \cdot a_{i}^{T} x\right)}{\left(\exp \left(-b_{i} \cdot a_{i}^{T} x\right)+1\right)^{2}}, \quad i=1, \ldots, m.
\end{align*}
Note that $A$ is sparse in many real data sets, therefore $\nabla^2 f(x)p$ can be computed efficiently with worst-case complexity $\mathcal{O}(nm)$.

Note that the performance of several fast algorithms on the datasets \textbf{rcv1.train}, \textbf{news20} and \textbf{colon-cancer} have been studied in \citep{yue2019family}, where it is shown that \textbf{IRPN} outperforms \textbf{FISTA} \citep{beck2009fast}, \textbf{SpaRSA} \citep{wright2009sparse}, \textbf{CGD} \citep{tseng2009coordinate,yun2011coordinate}, and is comparable to \textbf{newGLMNET} \citep{yuan2012improved}.
We choose to exclude the dataset \textbf{colon-cancer} in our experiments because its size is too small, making the computation time too short to meaningfully compare the performance of different algorithms.  Moreover, the parameter $\gamma$ used in our experiments is different from that in \citep{yue2019family}, which may affect the relative performance of these methods.
For these reasons, we also include \textbf{FISTA}, \textbf{SpaRSA} and \textbf{newGLMNET} in our comparison.
We also include \textbf{AltN} \citep{bareilles2023newton}, since it is a second-order method that exploits the manifold identification property to accelerate convergence.
The detailed settings of these algorithms are as follows:
\begin{enumerate} \fontsize{11}{15}\selectfont
\item {\bf FISTA} This is the fast iterative shrinkage-thresholding algorithm proposed in \citep{beck2009fast}. We use the constant step size $t = 1/L_g$, where $L_g$ is the Lipschitz constant of $\nabla f$. For the logistic regression problem, $L_g = \frac{1}{4m} \lambda_{\max}(A^TA)$. Other settings are the same as the experiment in \citep{yue2019family}.
\item {\bf SpaRSA} This is the sparse reconstruction by the separable approximation algorithm proposed in \citep{wright2009sparse}. We use the same setting as in \citep{wright2009sparse,yue2019family}.
\item {\bf newGLMNET\footnote{\url{https://www.csie.ntu.edu.tw/~cjlin/liblinear/}} (newG)} This is the improved GLMNET algorithm proposed in \citep{yuan2012improved}. We use the same setting as in \citep{yuan2012improved,yue2019family}.
\item {\bf Alternating Newton acceleration\footnote{\url{https://github.com/GillesBareilles/NewtonRiemannAccel-ProxGrad}} (AltN)} This is an accelerated proximal gradient algorithm \citep{bareilles2023newton}, which alternates between proximal gradient step and Riemannian Newton step on the current
manifold $\mathcal{M}_k$. In $\ell_1$-regularized problems, the current manifold is determined by the current support set at iterate $x_k$, i.e., $\mathcal{M}_k \triangleq \{ x \in \mathbb{R}^n: x_i = 0, \forall i \in \mathcal{A}(x_k) \}$, where $\mathcal{A}(x_k) \triangleq \{ i: x_k^i = 0 \}$. Each iteration has the following updates:
\begin{align*}
    x_{k} & \triangleq \operatorname{prox}_{t h}(y_{k-1} - t \nabla f(y_{k-1})), \\
    y_{k} & \triangleq \operatorname{ManAcc}_{\mathcal{M}_k}(x_{k}),
\end{align*}
where $\operatorname{ManAcc}_{\mathcal{M}_k}(x_{k})$ denotes the Riemannian (truncated) Newton step on manifold $\mathcal{M}_k$ starting from point $x_k$. Our implementation is based on the authors' code.  To improve performance, we replace the original stepsize with a Barzilai–Borwein (BB) stepsize combined with a backtracking linesearch.
\item {\bf IRPN\footnote{\url{https://github.com/ZiruiZhou/IRPN}}} This is the family of inexact SQA methods proposed in \citep{yue2019family}, and we use the same setting, i.e. $\theta = \beta = 0.25, \zeta = 0.4, c = 10^{-6}$ and $\eta= 0.5$. We choose $\rho=0.5$ since \textbf{IRPN} has superlinear convergence in this setting and performs best in \cite{yue2019family}.
The coordinate descent algorithm is used for solving the inner problems of \textbf{IRPN}.
\item {\bf Two-metric adaptive projection (TMAP)} This is Algorithm~\ref{2m-proj} we proposed in this paper. We set $c = 10^{-4}$, $\beta = 0.2 $, $\sigma = \tau = 0.1$, $\epsilon = 10^{-3}$ and $\delta = \frac{1}{2}$. We start all linesearch with step size $t = 1$.
% to avoid high computation cost for Lipschitz constant in our algorithm.
\end{enumerate}

We initialize all the tested algorithms at the same point $x_0 = 0$. All the algorithms are terminated if the iterate $x_k$ satisfies $r(x_k) = \| x_k - \operatorname{prox}_{ h} (x_k - \nabla f(x_k)) \| \leq \texttt{tol}$, where $\texttt{tol} \in \{10^{-6}, 10^{-8}, 10^{-10}\}$. Note that for \textbf{SpaRSA} and \textbf{FISTA}, we set a maximum number of iterations as 10,000.

The computational results on the data sets \textbf{rcv1.train}, \textbf{rcv1.test}, \textbf{news20} and \textbf{real-sim}
are presented in Table~\ref{tab:rcv1-train}, \ref{tab:news20}, \ref{tab:real-sim}, \ref{tab:rcv1-test}, respectively. 
In these tables, we report the number of outer iterations, inner iterations, and CPU time (in seconds) required by each algorithm to achieve different accuracy levels.
In Figure~\ref{fig:convergence}, we plot the convergence behavior of the tested algorithms on the data sets \textbf{rcv1.train}, \textbf{news20}, \textbf{real-sim} and \textbf{rcv1.test}. The $y$-axis represents the logarithm of the residual $r(x_k) = \| x_k - \operatorname{prox}_{ h} (x_k - \nabla f(x_k)) \|$, while the $x$-axis represents the iteration number.
\begin{table}[h]
    \centering    
    \caption{Computational results on {\bf rcv1.train}}
    \label{tab:rcv1-train}                                      
    \begin{tabular}{cccccccc}   
        \hline                          
        Tol. & & SpaRSA & FISTA & AltN & IRPN & newG & TMAP \\ 
        \hline \hline      
        \multirow{3}{*}{$10^{-6}$}& outer iter. & - & - & - & 7 & 19 & - \\          
        & inner iter. & 687 & 444 & 63 & 292 & 206 & 11 \\ 
        & time & 2.66 & 2.54 & 0.80 & 0.92 & 0.71 & 0.16 \\
       \hline \hline
       \multirow{3}{*}{$10^{-8}$}& outer iter. & - & - & - & 8 & 25 & - \\          
        & inner iter. & 1377 & 828 & 63 & 392 & 343 & 12 \\
        & time & 5.58 & 4.84 & 0.80 & 1.30 & 1.24 & 0.17 \\
       \hline \hline
       \multirow{3}{*}{$10^{-10}$}& outer iter. & - & - & - & 9 & 31 & - \\            
        & inner iter. & 2083 & 1284 & 64 & 492 & 492 & 13 \\
        & time & 8.04 & 7.07 & 0.83 & 1.53 & 1.63 & 0.19 \\  
       \hline \hline
    \end{tabular}
\end{table}

\begin{table}[h]
    \centering    
    \caption{Computational results on {\bf news20}}
	\label{tab:news20}
    \begin{tabular}{cccccccc}   
        \hline                          
        Tol. & &  SpaRSA & FISTA & AltN & IRPN & newG & TMAP \\       
        \hline \hline      
        \multirow{3}{*}{$10^{-6}$}& outer iter. & - & - & - & 6 & 22 & - \\                
        & inner iter. & 2235 & 1596 & 240 & 497 & 567 & 15 \\    
        & time & 107.28 & 99.92 & 32.86 & 16.73 & 20.24 & 1.53 \\
       \hline \hline
       \multirow{3}{*}{$10^{-8}$}& outer iter. & - & - & - & 9 & 32 & - \\                 
        & inner iter. & 5941 & 2552 & 241 & 797 & 910 & 17 \\     
        & time & 283.08 & 158.77 & 33.42 & 26.60 & 31.40 & 1.62 \\
       \hline \hline
       \multirow{3}{*}{$10^{-10}$}& outer iter. & - & - & - & 13 & 99 & - \\                 
        & inner iter. & 10000 & 3398 & 241 & 1197 & 6970 & 18 \\   
        & time & 501.80 & 221.20 & 33.42 & 41.93 & 259.35 & 1.84 \\
       \hline \hline
    \end{tabular}
\end{table}

\iffalse
\begin{table}[h]
    \centering    
    \caption{Computational results on {\bf colon-cancer}}
	\label{tab:colon-cancer}
    \begin{tabular}{cccccccc}   
        \hline                          
        Tol. & &  SpaRSA & FISTA & AltN & IRPN & newG & TMAP \\       
        \hline \hline      
        \multirow{3}{*}{$10^{-6}$}& outer iter. & - & - & - & 7 & 23 & - \\           
        & inner iter. & 1096 & 2895 & 92 & 197 & 170 & 42 \\
        & time & 0.14 & 0.51 & 0.04 & 0.04 & 0.03 & 0.05 \\ 
       \hline \hline
       \multirow{3}{*}{$10^{-8}$}& outer iter. & - & - & - & 8 & 29 & - \\           
        & inner iter. & 1885 & 4584 & 93 & 297 & 238 & 43 \\
        & time & 0.24 & 0.80 & 0.04 & 0.05 & 0.04 & 0.05 \\ 
       \hline \hline
       \multirow{3}{*}{$10^{-10}$}& outer iter. & - & - & - & 8 & 35 & - \\           
        & inner iter. & 2514 & 6375 & 94 & 297 & 306 & 43 \\
        & time & 0.32 & 1.12 & 0.04 & 0.06 & 0.06 & 0.06 \\ 
       \hline \hline
    \end{tabular}
\end{table}
\fi

\begin{table}[h]
	\centering
	\caption{Computational results on {\bf real-sim}}
	\label{tab:real-sim}
    \begin{tabular}{cccccccc}   
        \hline                          
        Tol. & &  SpaRSA & FISTA & AltN & IRPN & newG & TMAP \\       
        \hline \hline      
        \multirow{3}{*}{$10^{-6}$}& outer iter. & - & - & - & 9 & 18 & - \\          
        & inner iter. & 321 & 454 & 59 & 247 & 63 & 14 \\  
        & time & 3.28 & 5.60 & 1.56 & 2.00 & 0.69 & 0.63 \\
       \hline \hline
       \multirow{3}{*}{$10^{-8}$}& outer iter. & - & - & - & 10 & 24 & - \\            
        & inner iter. & 4847 & 1903 & 76 & 347 & 228 & 19 \\ 
        & time & 44.08 & 23.29 & 2.72 & 2.90 & 2.10 & 0.91 \\
       \hline \hline
       \multirow{3}{*}{$10^{-10}$}& outer iter. & - & - & - & 12 & 33 & - \\             
        & inner iter. & 10000 & 5044 & 79 & 547 & 548 & 21 \\
        & time & 88.13 & 60.12 & 2.92 & 4.55 & 4.85 & 1.17 \\ 
       \hline \hline
    \end{tabular}
\end{table}

\begin{table}[h]
    \centering    
    \caption{Computational results on {\bf rcv1.test}}
    \label{tab:rcv1-test}                                      
    \begin{tabular}{cccccccc}   
        \hline                          
        Tol. & & SpaRSA & FISTA & AltN & IRPN & newG & TMAP \\ 
        \hline \hline      
        \multirow{3}{*}{$10^{-6}$}& outer iter. & - & - & - & 10 & 28 & - \\                  
        & inner iter. & 1645 & 1345 & 183 & 706 & 759 & 27 \\       
        & time & 224.33 & 200.49 & 99.04 & 86.48 & 93.58 & 20.30 \\
       \hline \hline
       \multirow{3}{*}{$10^{-8}$}& outer iter. & - & - & - & 18 & 45 & - \\                     
        & inner iter. & 10000 & 5271 & 273 & 1506 & 1609 & 34 \\      
        & time & 1343.19 & 772.10 & 242.96 & 179.12 & 191.88 & 52.60 \\
       \hline \hline
       \multirow{3}{*}{$10^{-10}$}& outer iter. & - & - & - & 28 & 57 & - \\
        & inner iter. & 10000 & 7885 & 289 & 2506 & 2312 & 37 \\
        & time & 1343.19 & 1169.59 & 257.45 & 315.32 & 286.32 & 73.28 \\
       \hline \hline
    \end{tabular}
\end{table}

\begin{figure}[ht]
    \centering
    \subfigure[\textbf{rcv1.train}]{
        \includegraphics[width=0.45\linewidth]{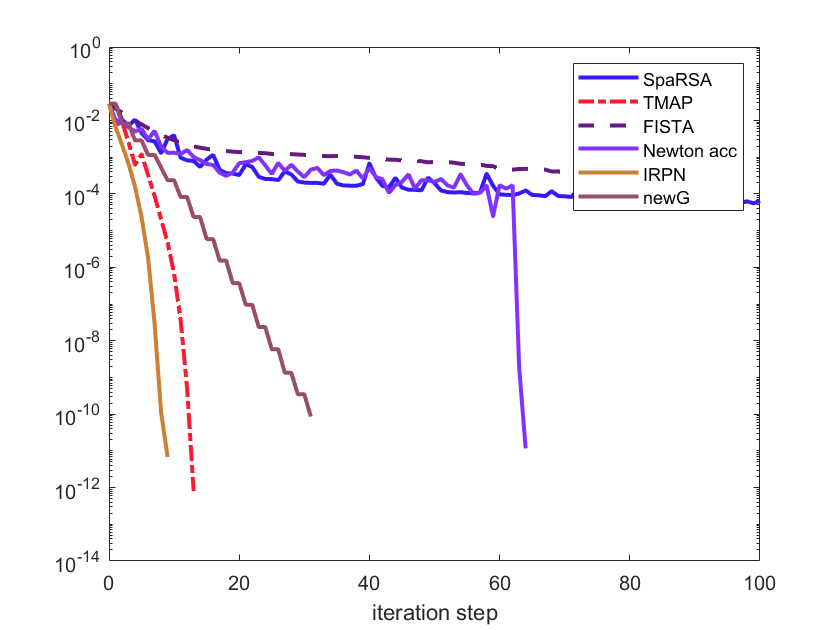}
        \label{rcv1.train}
    }
    \subfigure[\textbf{news20}]{
        \includegraphics[width=0.45\linewidth]{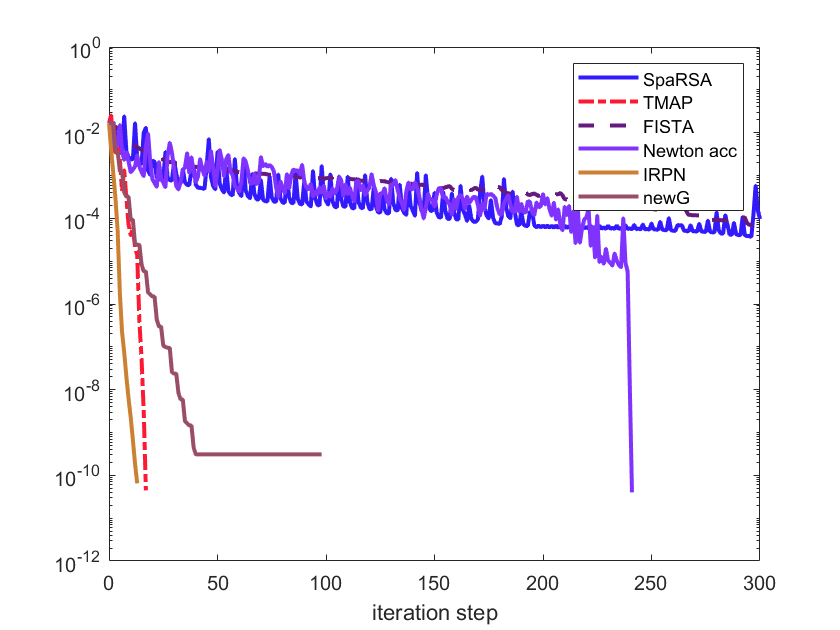}
        \label{news20}
    }
    \subfigure[\textbf{real-sim}]{
        \includegraphics[width=0.45\linewidth]{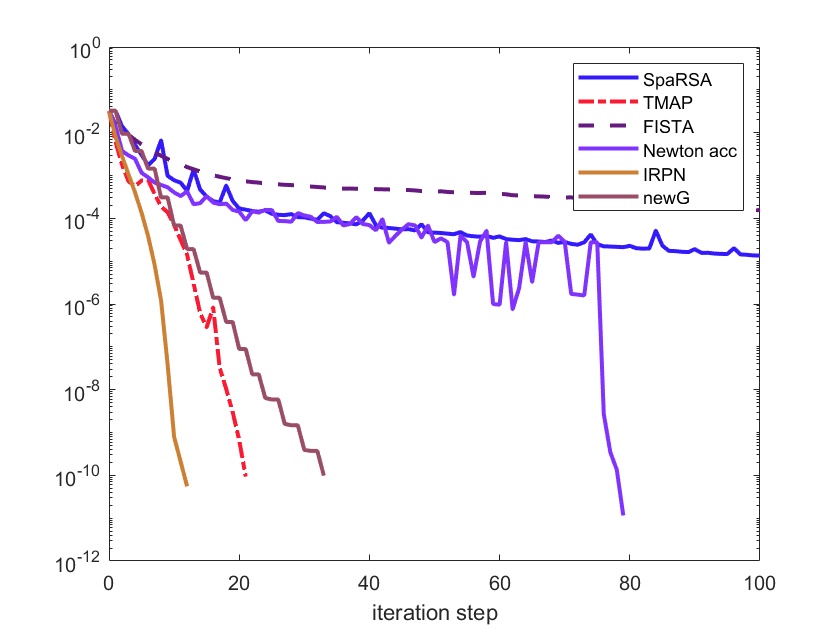}
        \label{real-sim}
    }
    \subfigure[\textbf{rcv1.test}]{
        \includegraphics[width=0.45\linewidth]{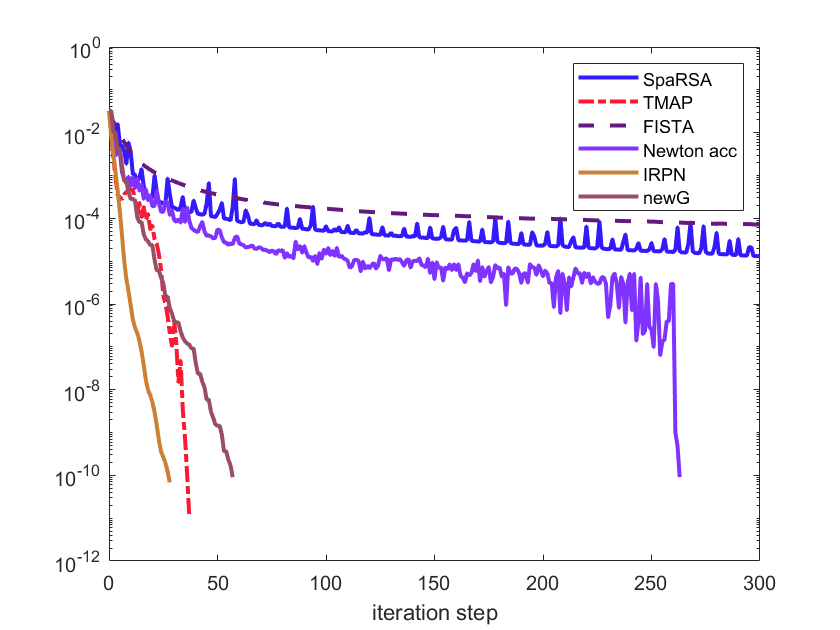}
        \label{rcv1.test}
    }
    \caption{Convergence behavior of the sequence of $\{\|x_k - \operatorname{prox}_h(x_k - \nabla f(x_k))\|\}_{k\geq 0}$}
    \label{fig:convergence}
\end{figure}

From these results, we observe that our proposed \textbf{TMAP} method outperforms all other algorithms on these large-scale data sets across all considered accuracy levels in terms of CPU time. 
Notably, \textbf{AltN} requires more iterations and a longer initial phase to reach a high-accuracy solution, primarily because its proximal gradient step is less efficient at identifying the active manifold and locating a suboptimal point at the early stages. Once the active manifold is correctly identified, \textbf{AltN} enjoys rapid local convergence (can be observed in Figure~\ref{fig:convergence}).
Compared to \textbf{AltN}, \textbf{IRPN} and \textbf{newGLMNET} perform slightly better or comparable in the early stage of optimization, but their runtimes grow more rapidly as the required tolerance becomes higher, so they ultimately require more time to reach highly accurate solutions.

From the above discussion, we see that the key to the fast convergence of \textbf{TMAP} and \textbf{AltN} in practice is that both methods explicitly exploit the manifold identification property in their design. Once the active manifold is identified, the algorithms locally reduce to a Newton method restricted to this manifold, which, thanks to sparsity, results in a low computational cost.

\iffalse
\yx{(The discussions below seems irrelevant to the main content.)} The other algorithms similar to \textbf{AltN} are the $\mathcal{VU}$-algorithm \citep{doi:10.1137/24M1697001} and \textbf{ISQA+} \citep{lee2023accelerating}. The former one alternates between proximal gradient steps ($\mathcal{V}$ step) and Newton steps on the $\mathcal{U}$-space ($\mathcal{U}$ step) and the latter one employs SQA for initial manifold identification and subsequently switches to an alternating Newton method for accelerated optimization. 

Ideally, to exploit this valuable identification property, we would like to switch to a more sophisticated method, such as a Riemannian Newton method, once we know that the iterates lie on the active manifold. However, although theory guarantees that the sequence generated by the algorithm will be captured by the active manifold in a finite number of iterations, in practice we never know for sure whether the current iterate $x_k$ is already on this manifold.

Therefore, \textbf{AltN} always performs a proximal gradient step before the Newton step to ensure automatic identification of the active manifold. Similarly, \textbf{ISQA+} uses the alternating Newton method after the SQA phase produces an iterate that is likely to be on the active manifold. Thus, instead of saying that SQA identifies the manifold, it is more accurate to say that SQA provides a low-accuracy approximate solution, allowing AltN to complete the identification steps very quickly.
\fi

\subsection{Numerical results on the LASSO problem}
We consider the LASSO problem as follows:
\begin{equation}\label{lasso}
    \min _{x \in \mathbb{R}^{n}}\frac{1}{2} \|Ax-b\|_{2}^{2}+\gamma\|x\|_{1} .
\end{equation}
Here, $A \in \mathbb{R}^{m \times n}$ and $b \in \mathbb{R}^m$ are given data; $\gamma > 0$ is a given regularization parameter.
Closely related to the LASSO problem is the Basis Pursuit problem which is a convex optimization problem and seeks the sparsest solution to an underdetermined linear system of equations. The Basis Pursuit (BP) problem can be formulated as
\begin{equation*}
    \min _{x \in \mathbb{R}^{n}}\|x\|_{1} \quad \text { s.t. } \quad Ax=b.
\end{equation*}
The theory for penalty functions implies that the solution of the LASSO problem goes to the solution of BP problem as $\gamma$ goes to zero. 
Numerical experiments suggest that in practice, most S-sparse signals are in fact recovered exactly via BP once $m \geq 4S$ or so and Fourier measurements are used \citep{candes2006robust}.

In our experiment, we compare our algorithm with other state-of-the-art algorithms on the large-scale reconstruction/LASSO problem \eqref{lasso}.

The problem setting is from \citep{wright2009sparse,de2016fast}.
The matrix $A \in \mathbb{R}^{m \times n}$ is a random matrix with $n \in \{2^{14}, 2^{15}, 2^{16}\}$ and $m \triangleq n/4$, with i.i.d Gaussian entries of zero mean and variance $\frac{1}{2n}$, i.e., $A_{i j} \sim \mathcal{N}\left(0, \frac{1}{2 n}\right)$.
The true signal $\bar{x} \in \mathbb{R}^{n}$ is generated with $k$ nonzero entries, where $k \triangleq \left\lfloor \rho m \right\rfloor$ and $\rho \in \{0.1,0.05,0.01\}$. The $k$ different indices are uniformly chosen from $\{1,2, \ldots ,n\}$ and the magnitude of each nonzero element is uniformly randomly chosen from $\{-1,1\}$. We choose $b \triangleq A\bar{x} + \epsilon$, where components of $\epsilon$ are i.i.d Gaussian
noises with a standard deviation $\bar{\sigma} \triangleq 0.01$. The regularization parameter is set to $\gamma \triangleq 0.1\|A^\top b\|_{\infty}$.

Our tested algorithms include \textbf{SpaRSA} \citep{wright2009sparse}, \textbf{FAST-2CDA-E} \citep{de2016fast}, \textbf{ASSN} \citep{xiao2018regularized}, \textbf{ProjectionL1} \citep{schmidt2007fast} and \textbf{PSSgb} \citep{schmidt2010graphical}.
The settings and review of these algorithms are as follows:
\begin{enumerate} \fontsize{11}{15}\selectfont
\item {\bf SpaRSA} {This is %the sparse reconstruction by 
the separable approximation algorithm proposed in \citep{wright2009sparse}. } We use the same setting as in \citep{wright2009sparse}. This method is famous and efficient for reconstruction problems among first-order methods, and it has been widely used as a benchmark in the literature.
\item {\bf FAST-2CDA-E\footnote{\url{https://sites.google.com/site/activesetlibrary/asm-active-set-methods-library}}} {This is one of the variants of the Fast Active SeT Block Coordinate Descent Algorithm (\textbf{FAST-BCDA}) proposed in \citep{de2016fast}. } We use the same setting as in \citep{de2016fast}. Also note that \textbf{FAST-2CDA-E} performs the best among the four variants of \textbf{FAST-BCDA}.
\item {\bf ASSN\footnote{\url{https://github.com/optsuite/ssm-Newton-l1}}} This is the semismooth Newton method proposed in \citep{xiao2018regularized}. It is highly efficient for solving reconstruction problems whose data are generated by the discrete cosine transform.
We follow the same setting as \citep{xiao2018regularized}, except that we adopt a new warm-start strategy (adaptive continuation) suggested by \citep{wright2009sparse}.
\item {\bf ProjectionL1\footnote{\url{https://www.cs.ubc.ca/~schmidtm/Software/L1General.html}} (TMP)} This is the two-metric projection method proposed by Bertsekas \citep{bertsekas1982projected}. This method was used by \citep{schmidt2007fast} to solve the equivalent formulation~\eqref{eq-bound}, and we adopt the same setting as in \citep{schmidt2007fast}. \textbf{ProjectionL1} approximates the Hessian via L-BFGS, unlike the original two-metric projection method.
\item {\bf PSSgb\footnote{\url{https://www.cs.ubc.ca/~schmidtm/Software/L1General.html}}} This is the improved version of \textbf{ProjectionL1} proposed in \citep{schmidt2010graphical}. It directly solves the Problem~\ref{l1min} without reformulation.
We use the same setting as in \citep{schmidt2010graphical} except that the parameter \texttt{options.quadraticInit} is set to 1. The Hessian matrix is approximated by L-BFGS in \textbf{PSSgb}.
\item {\bf TMAP} This is Algorithm~\ref{2m-proj} we proposed in this paper. We set $c = 0.1$, $\beta = 0.2 $, $\sigma = \tau = 0.1$, $\epsilon = 10^{-3}$ and $\delta = 0.7$.  We start all linesearch with step size $t = 1$. We also use a warm-start strategy (Adaptive Continuation) as suggested by \citep{wright2009sparse}.
\end{enumerate}

We initialize all the tested algorithms by the same point $x_0 = 0$. All the algorithms are terminated if the iterate $x_k$ satisfies $$ \| x_k - \operatorname{prox}_{ h} (x_k - \nabla f(x_k)) \| \leq \texttt{tol},$$
where $\texttt{tol} \in \{10^{-6}, 10^{-8}, 10^{-10}\}$ or the maximum number of iterations (500) is reached.
In warm-start strategy (Adaptive Continuation), we start from a large value of $\gamma_0 = \zeta\|A^\top b\|_{\infty}$ and solve a sequence of LASSO problems with decreasing regularization parameters $\gamma_{i+1} = \max\{\zeta\|A^\top (b - A y_i)\|_{\infty},\gamma\}$ until we reach the target value $\gamma$. 
$y_i$ is the solution of the previous problem, which is used as the initial point for the next problem. 
As suggested by \citep{wright2009sparse}, we set $\zeta = 0.2$, i.e., $\gamma_0 = 2 \gamma$.
We run each algorithm on 100 randomly generated problem instances for each pair of $n, \rho$ and present box plots showing the CPU time distribution (in seconds) in Figures~\ref{fig:lasso-1e-6}, \ref{fig:lasso-1e-8}, and \ref{fig:lasso-1e-10} for different accuracy levels.
In each box, the central mark is the median, the edges of the box are the 25th and 75th percentiles, the whiskers extend to the most extreme data points except for outliers, and outliers are plotted individually. 

\begin{figure}[htbp]
    \centering
    \subfigure[\textbf{$n = 2^{14}, \rho = 0.01$}]{
        \includegraphics[width=0.47\linewidth]{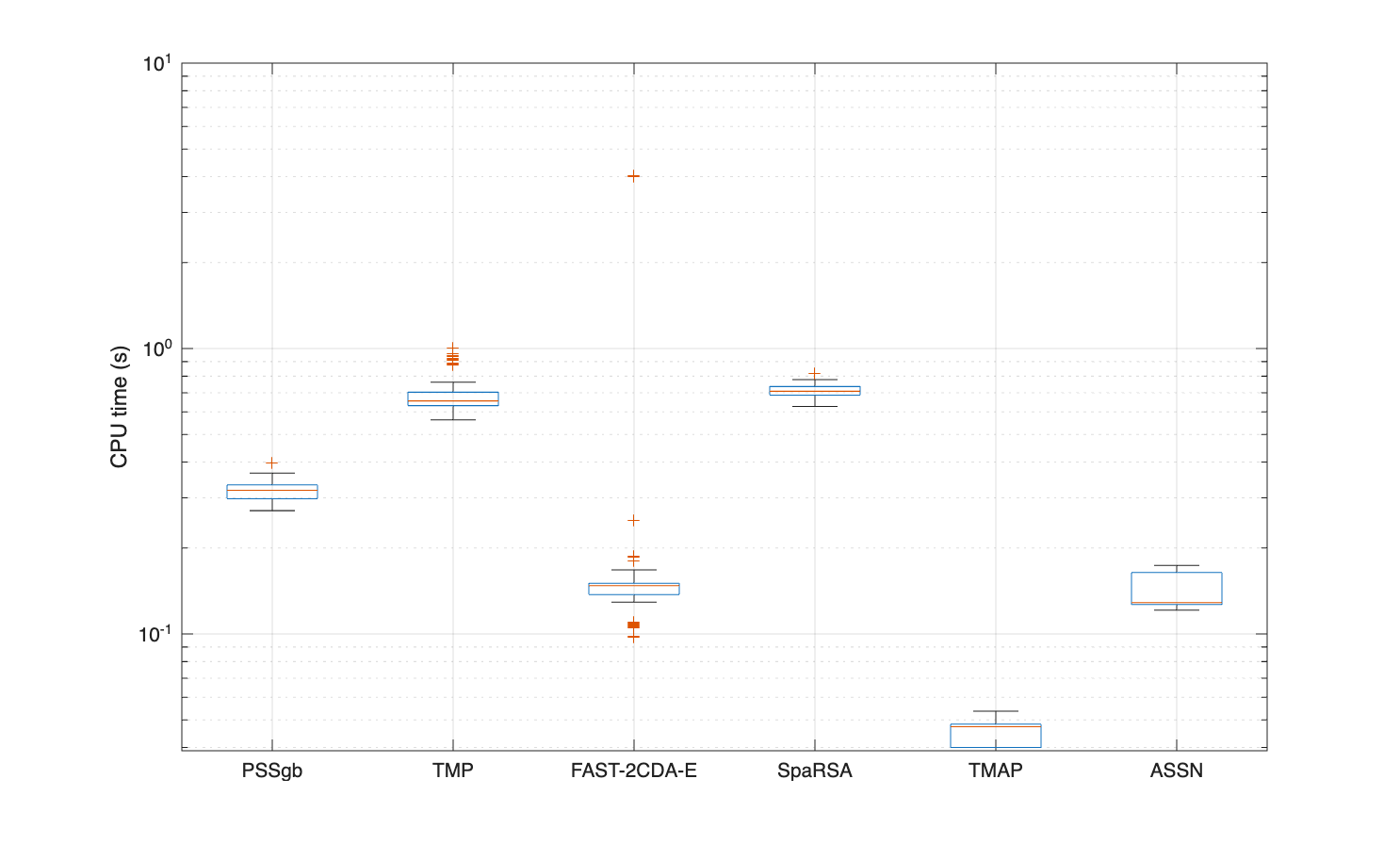}
        \label{2^14-0.01-1e-6}
    }
    \subfigure[\textbf{$n = 2^{14}, \rho = 0.05$}]{
        \includegraphics[width=0.47\linewidth]{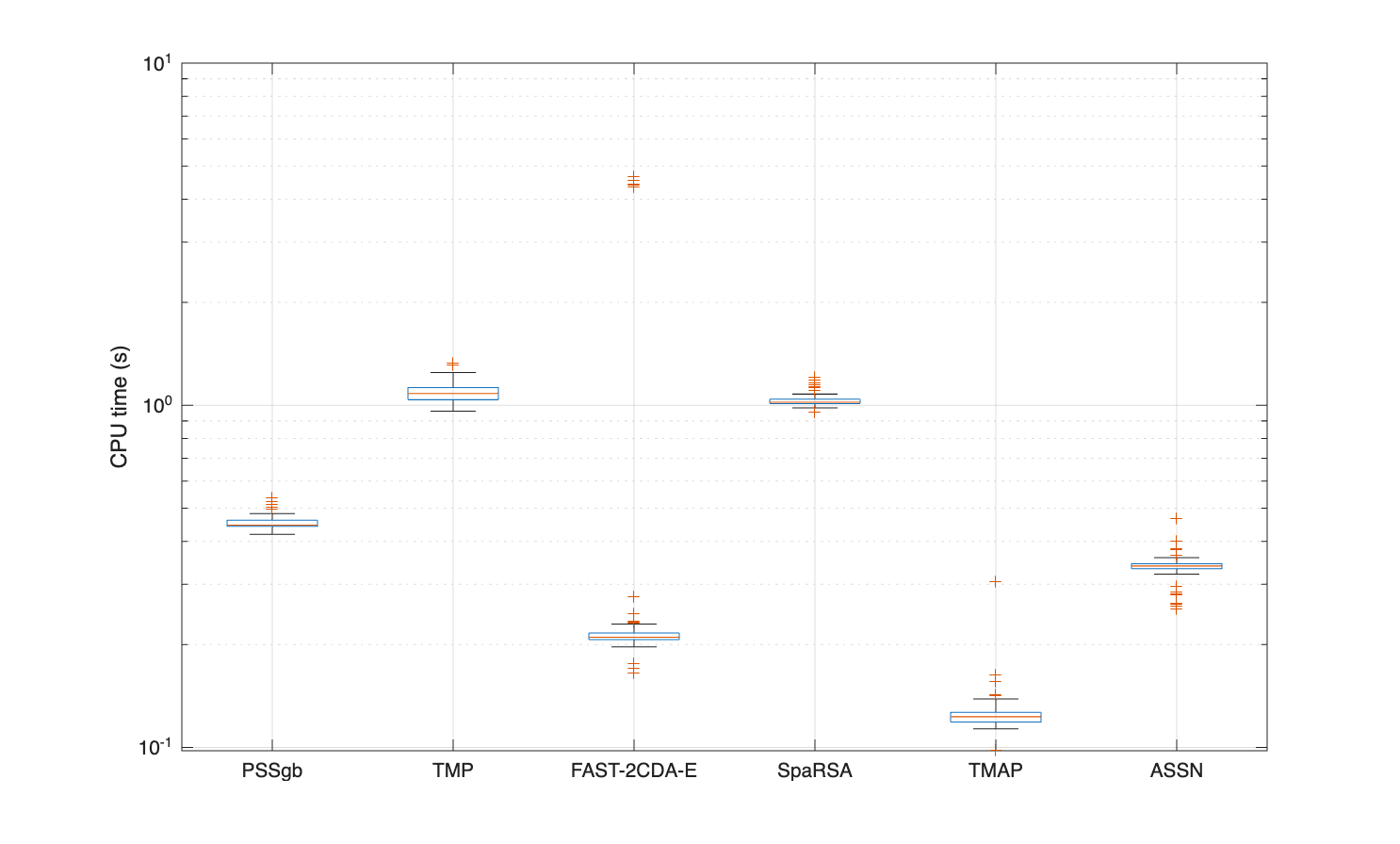}
        \label{2^14-0.05-1e-6}
    }
    \subfigure[\textbf{$n = 2^{14}, \rho = 0.1$}]{
        \includegraphics[width=0.47\linewidth]{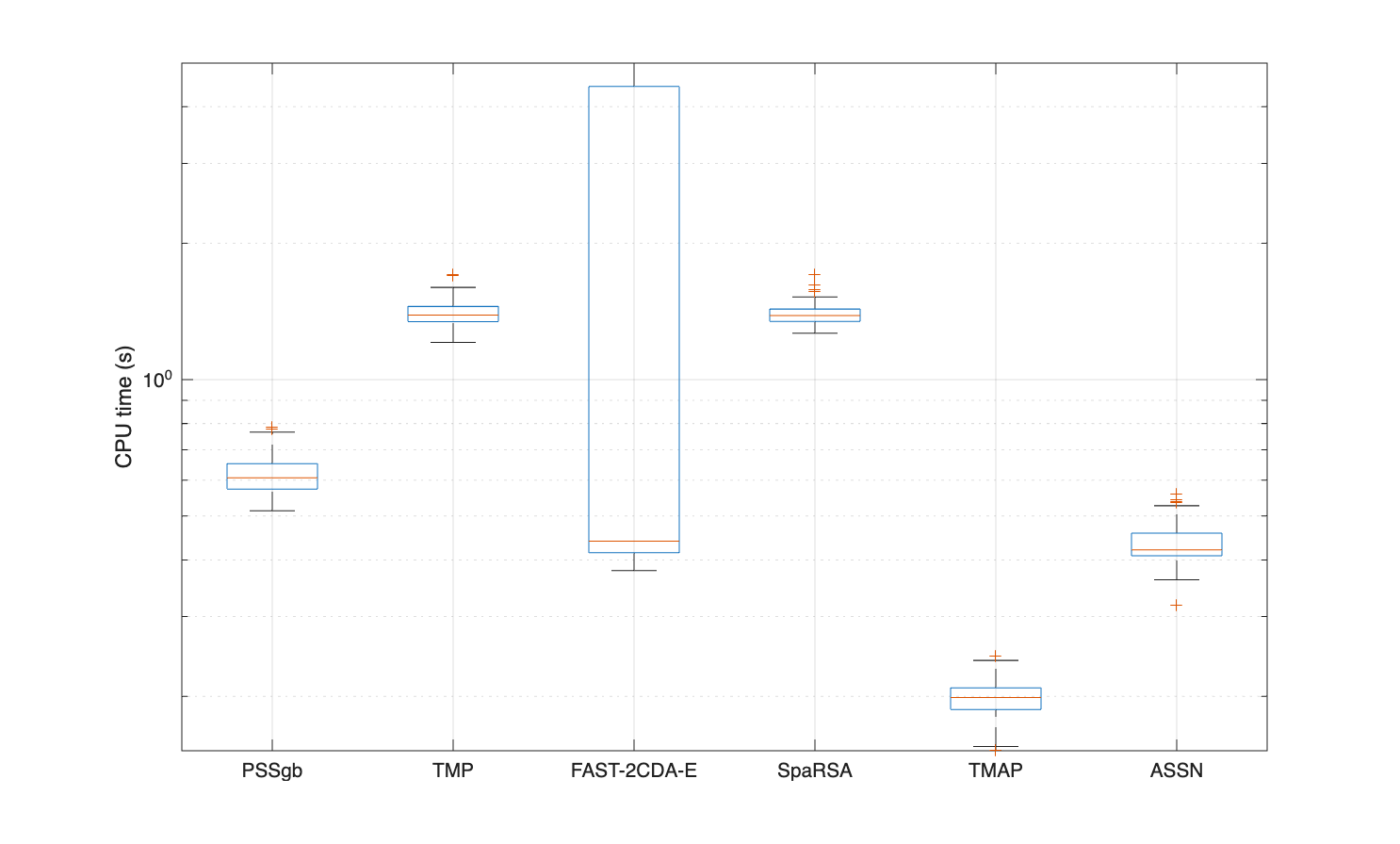}
        \label{2^14-0.1-1e-6}
    }
    \subfigure[\textbf{$n = 2^{15}, \rho = 0.01$}]{
        \includegraphics[width=0.47\linewidth]{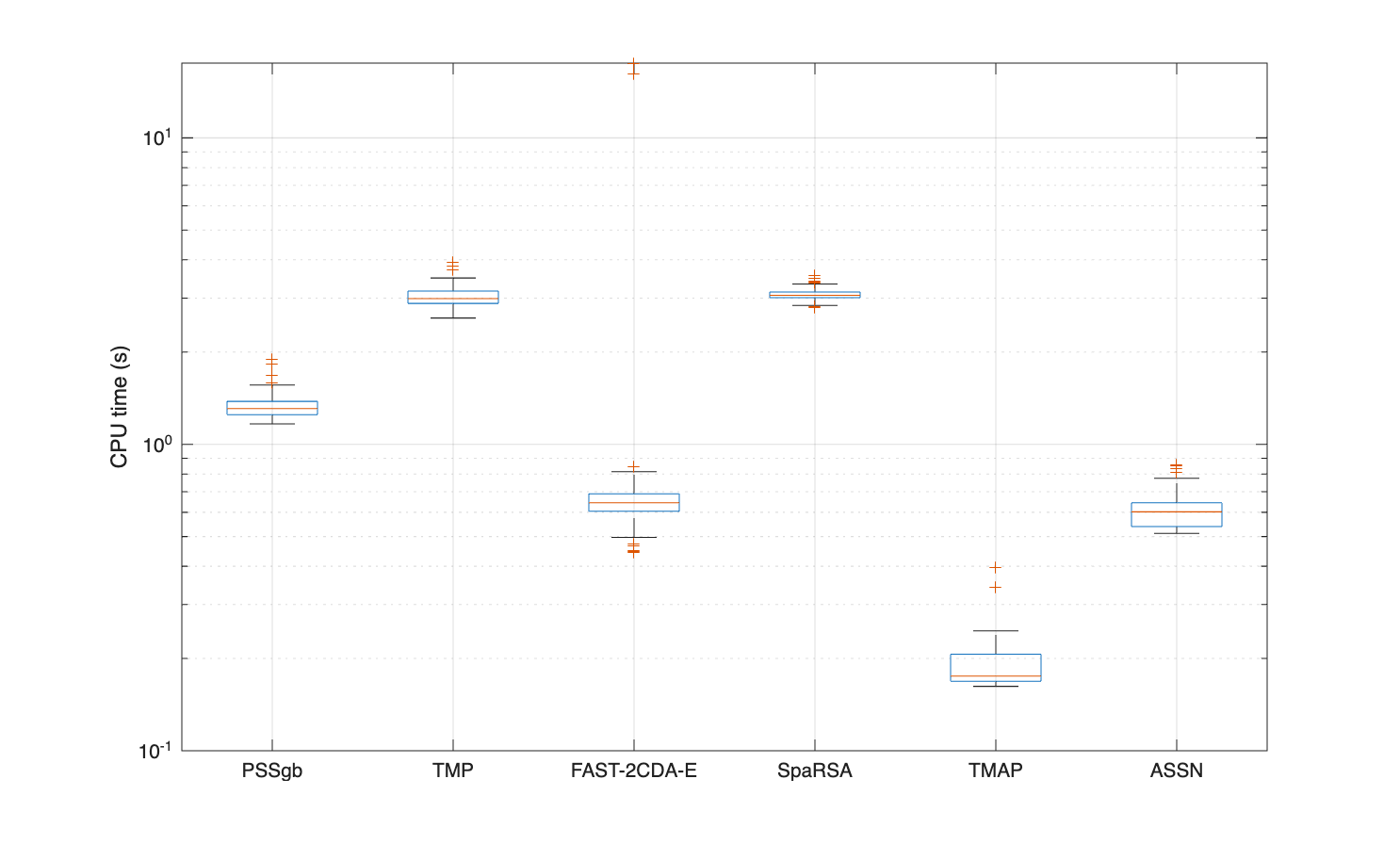}
        \label{2^15-0.01-1e-6}
    }
    \subfigure[\textbf{$n = 2^{15}, \rho = 0.05$}]{
        \includegraphics[width=0.47\linewidth]{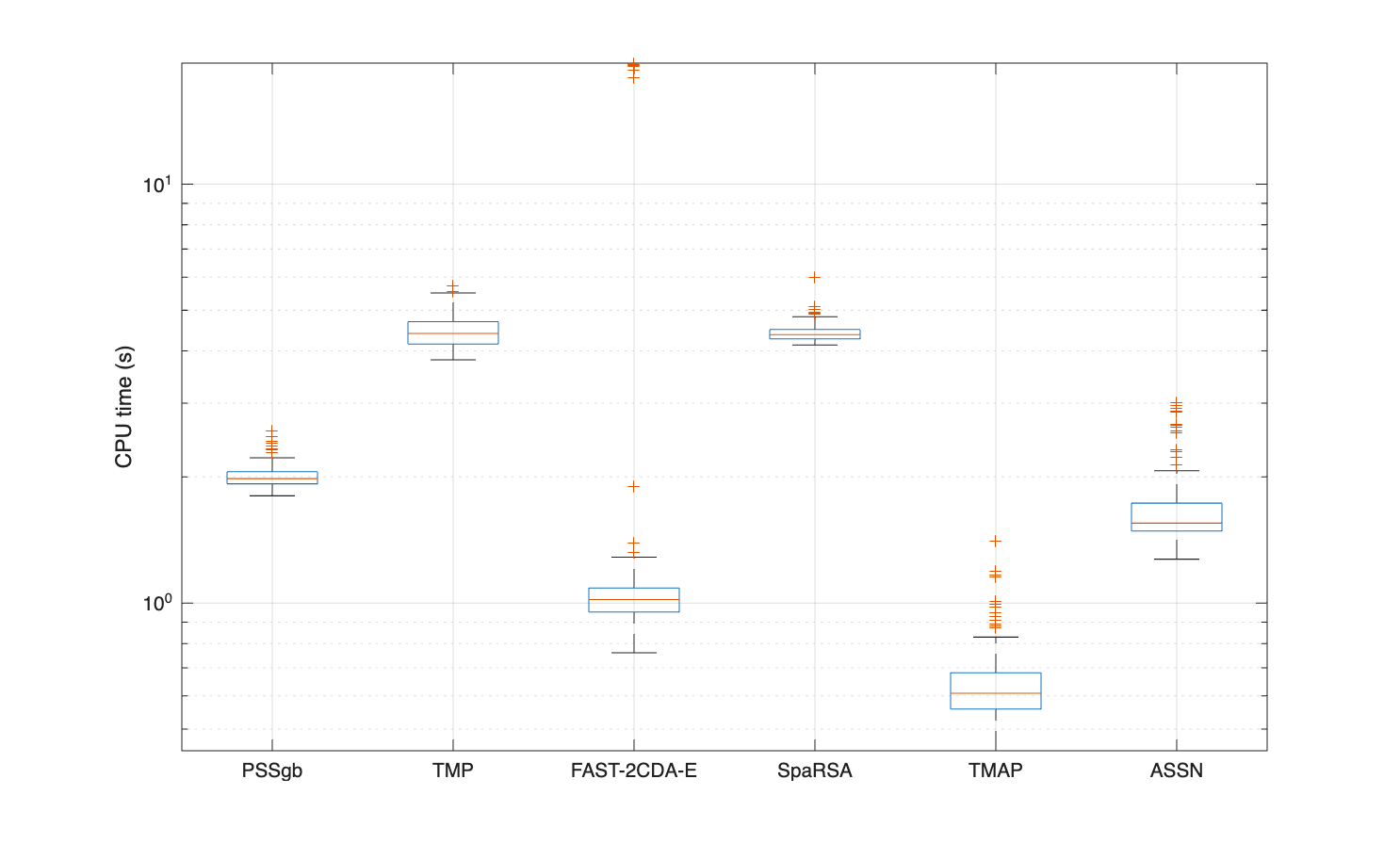}
        \label{2^15-0.05-1e-6}
    }
    \subfigure[\textbf{$n = 2^{15}, \rho = 0.1$}]{
        \includegraphics[width=0.47\linewidth]{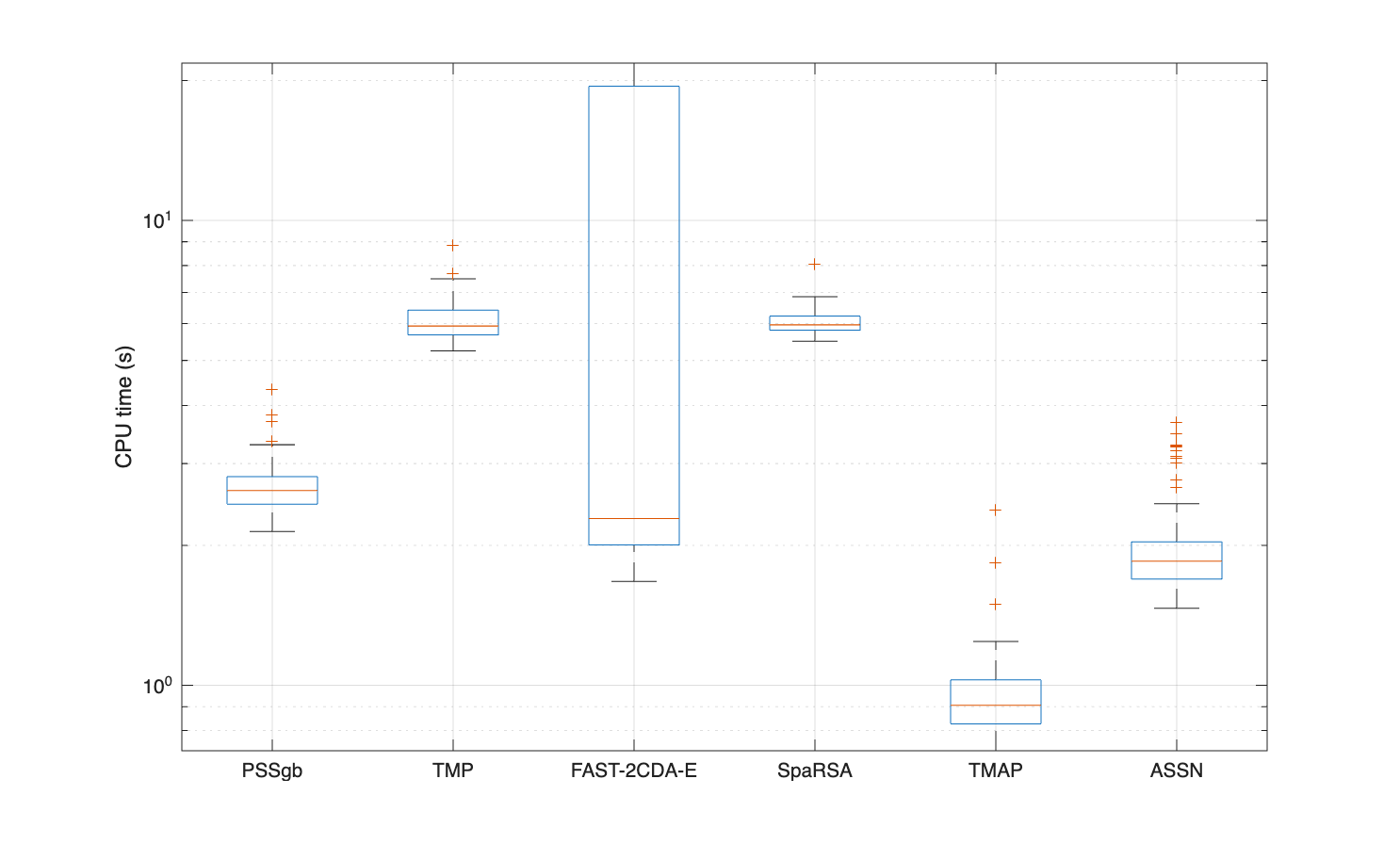}
        \label{2^15-0.1-1e-6}
    }
    \caption{\texttt{tol} = $10^{-6}$}
%        \yx{larger fontsize in each plot?}
    \label{fig:lasso-1e-6}
\end{figure}

\begin{figure}[ht]
    \centering
    \subfigure[\textbf{$n = 2^{14}, \rho = 0.01$}]{
        \includegraphics[width=0.47\linewidth]{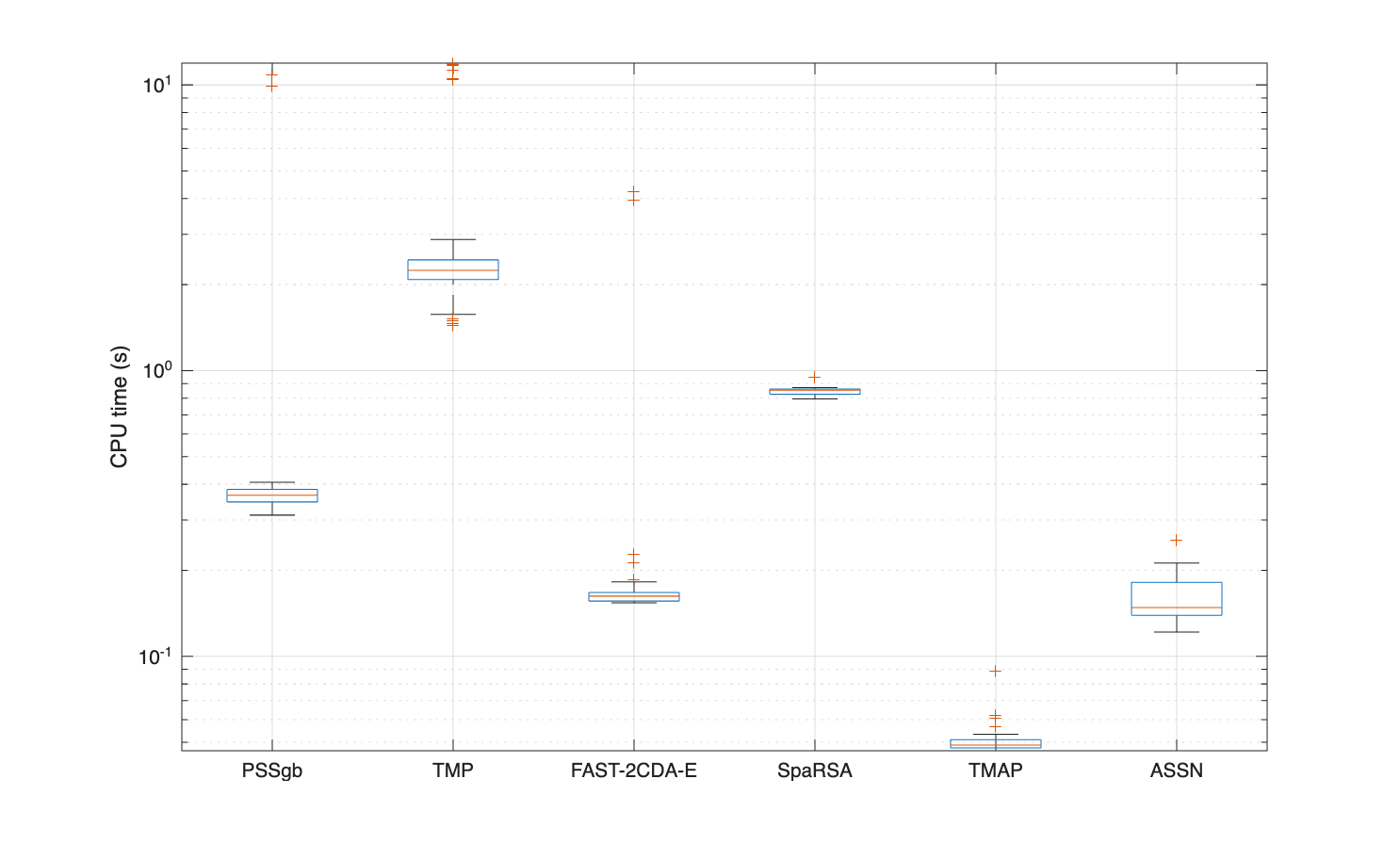}
        \label{2^14-0.01-1e-8}
    }
    \subfigure[\textbf{$n = 2^{14}, \rho = 0.05$}]{
        \includegraphics[width=0.47\linewidth]{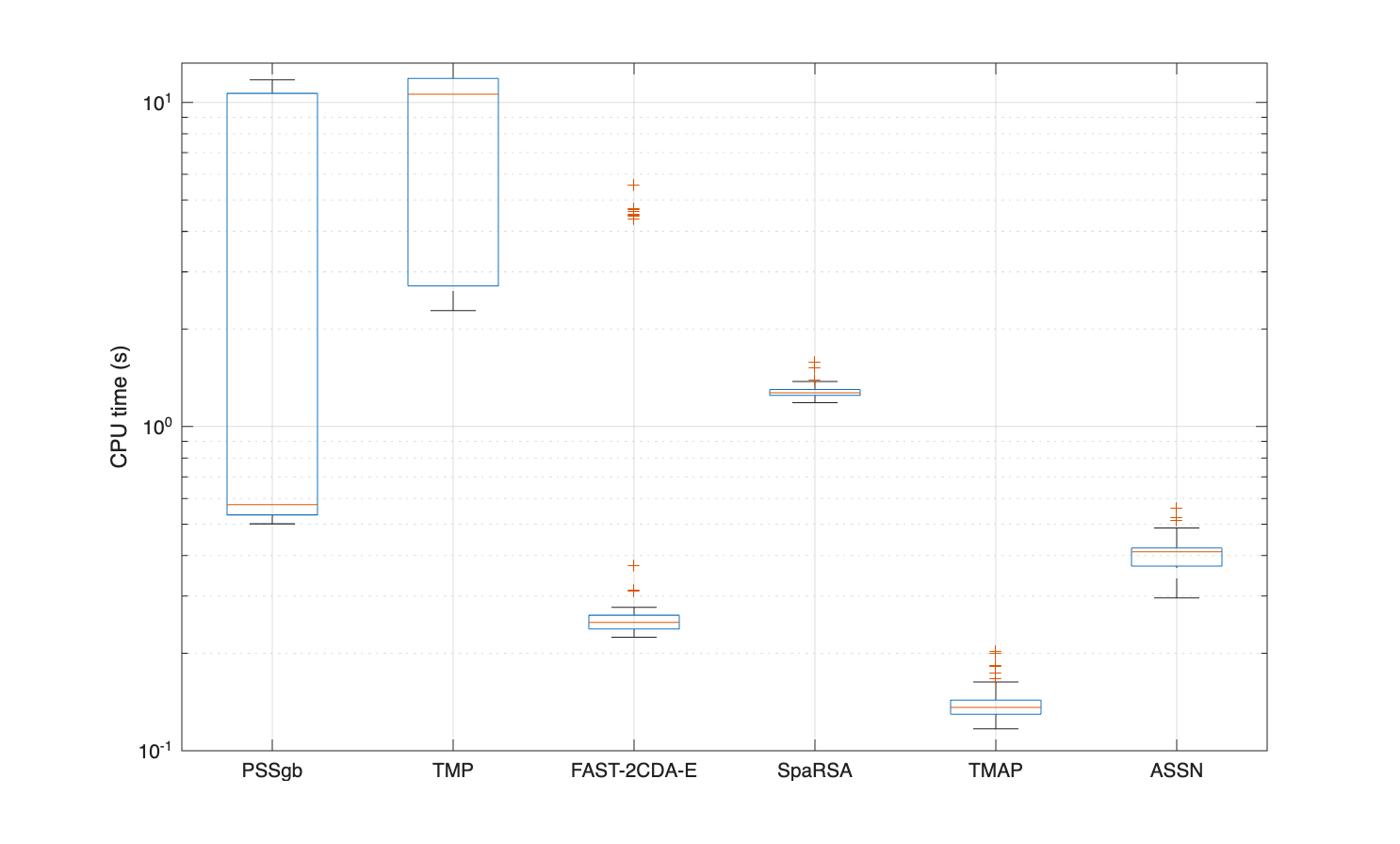}
        \label{2^14-0.05-1e-8}
    }
    \subfigure[\textbf{$n = 2^{14}, \rho = 0.1$}]{
        \includegraphics[width=0.47\linewidth]{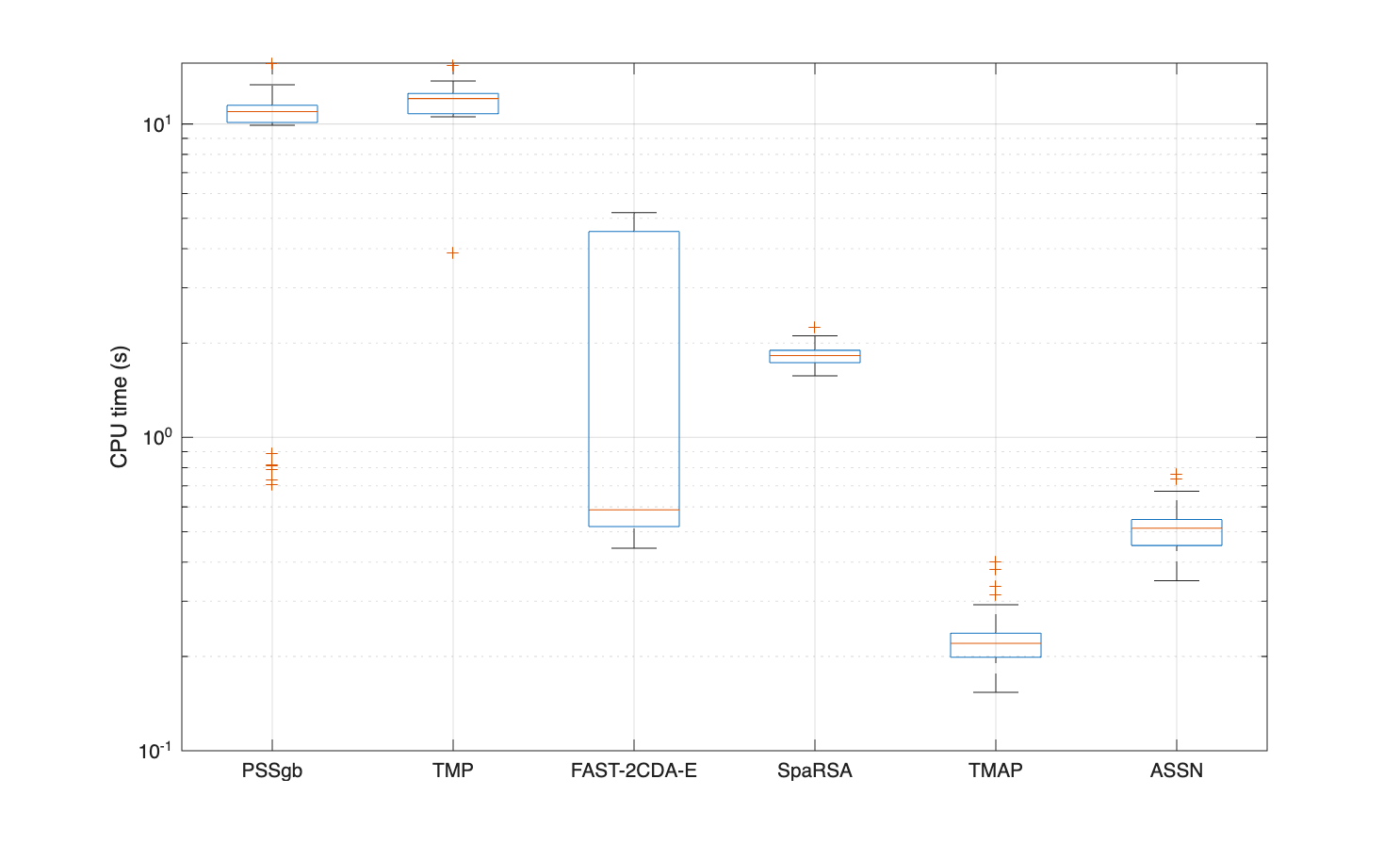}
        \label{2^14-0.1-1e-8}
    }
    \subfigure[\textbf{$n = 2^{15}, \rho = 0.01$}]{
        \includegraphics[width=0.47\linewidth]{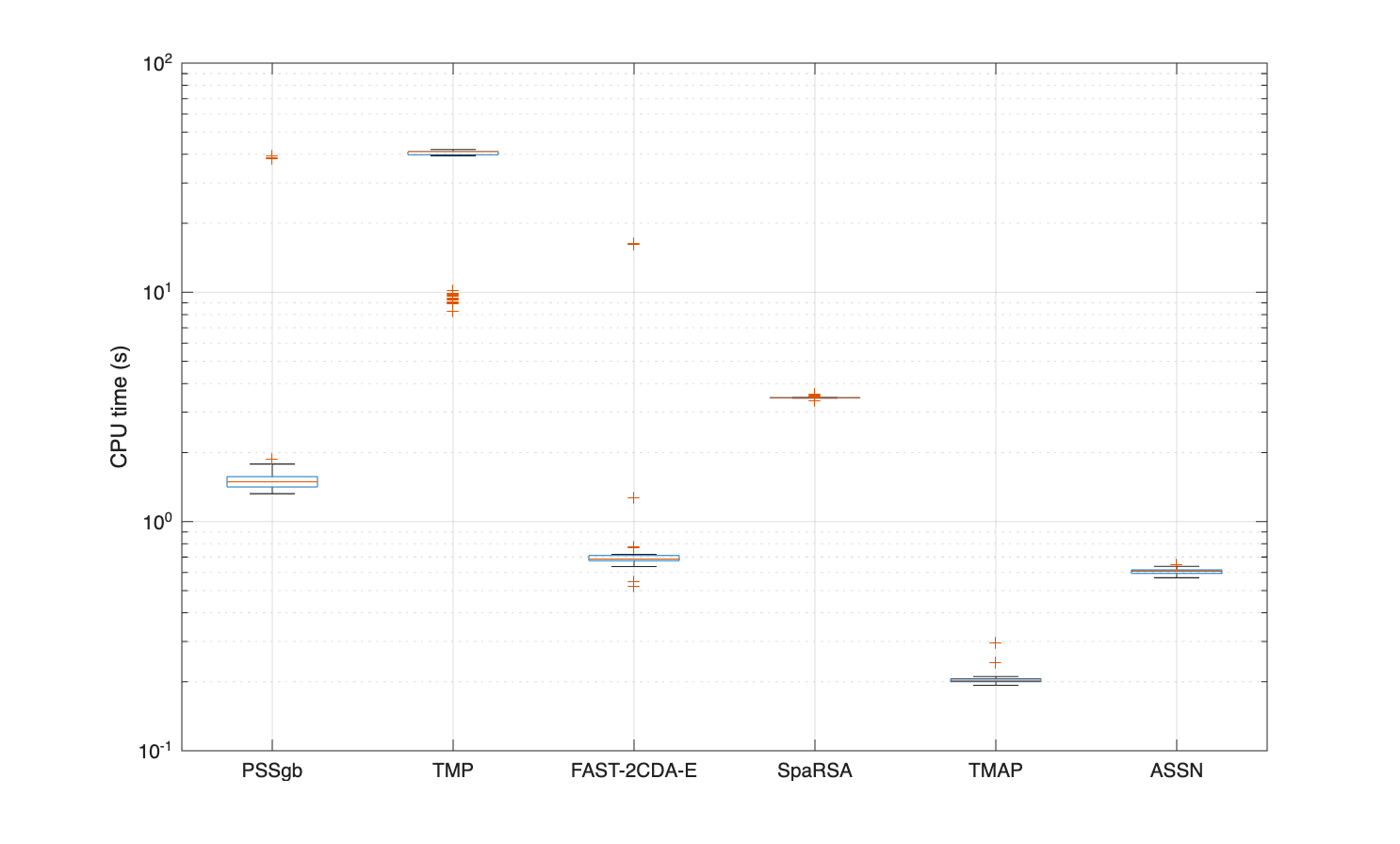}
        \label{2^15-0.01-1e-8}
    }
    \subfigure[\textbf{$n = 2^{15}, \rho = 0.05$}]{
        \includegraphics[width=0.47\linewidth]{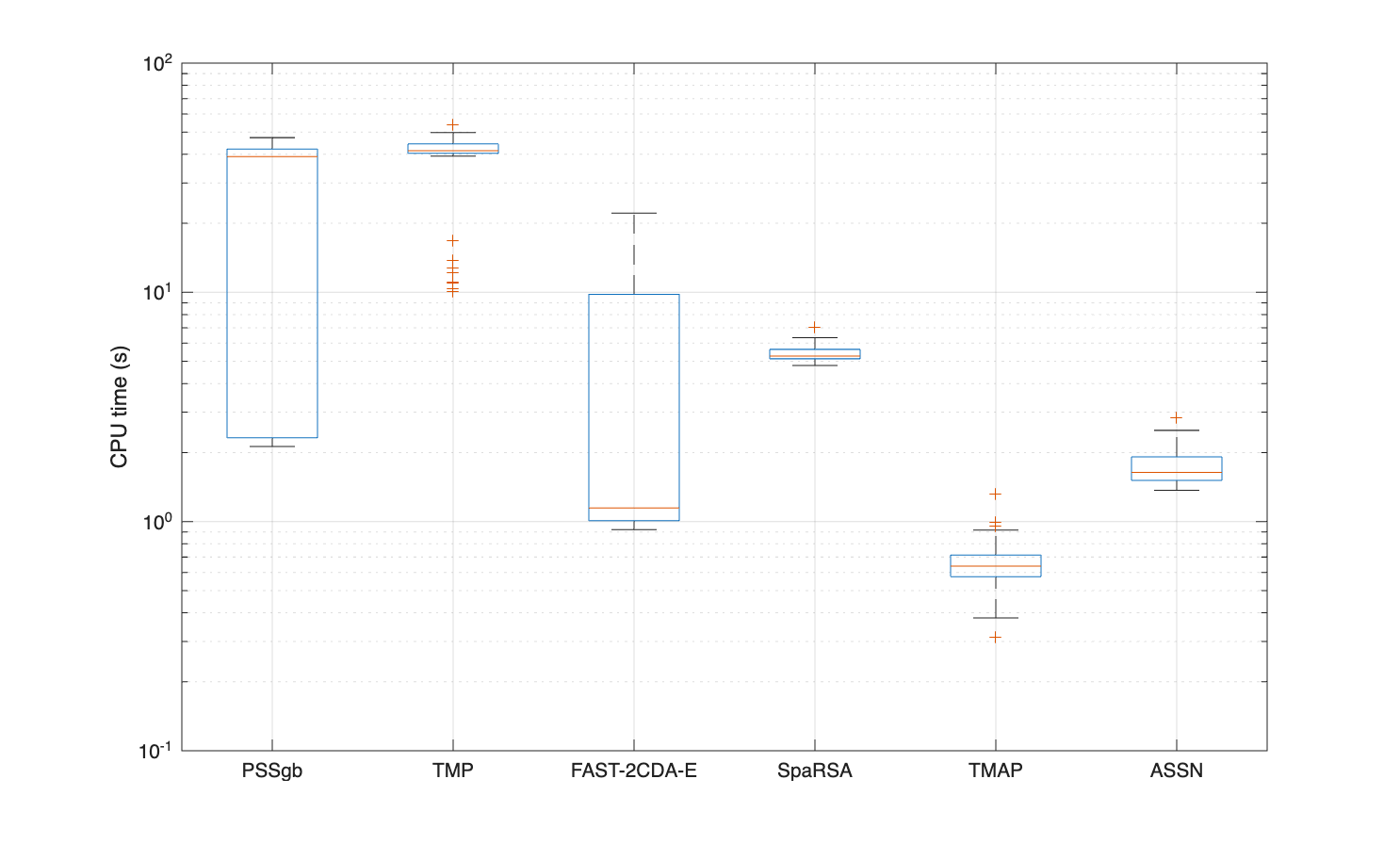}
        \label{2^15-0.05-1e-8}
    }
    \subfigure[\textbf{$n = 2^{15}, \rho = 0.1$}]{
        \includegraphics[width=0.47\linewidth]{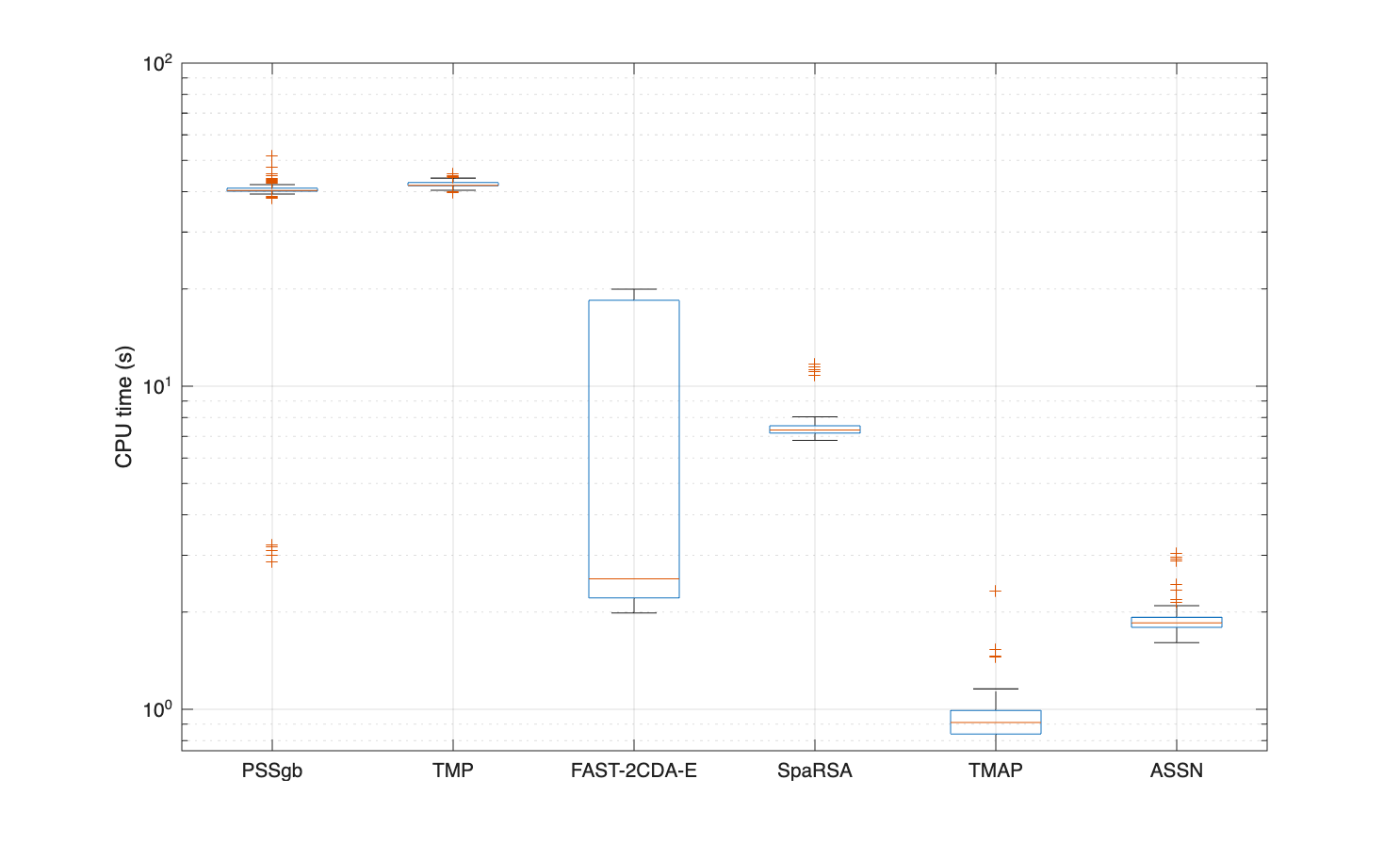}
        \label{2^15-0.1-1e-8}
    }
    \caption{\texttt{tol} = $10^{-8}$}
    \label{fig:lasso-1e-8}
\end{figure}

\begin{figure}[ht]
    \centering
    \subfigure[\textbf{$n = 2^{14}, \rho = 0.01$}]{
        \includegraphics[width=0.47\linewidth]{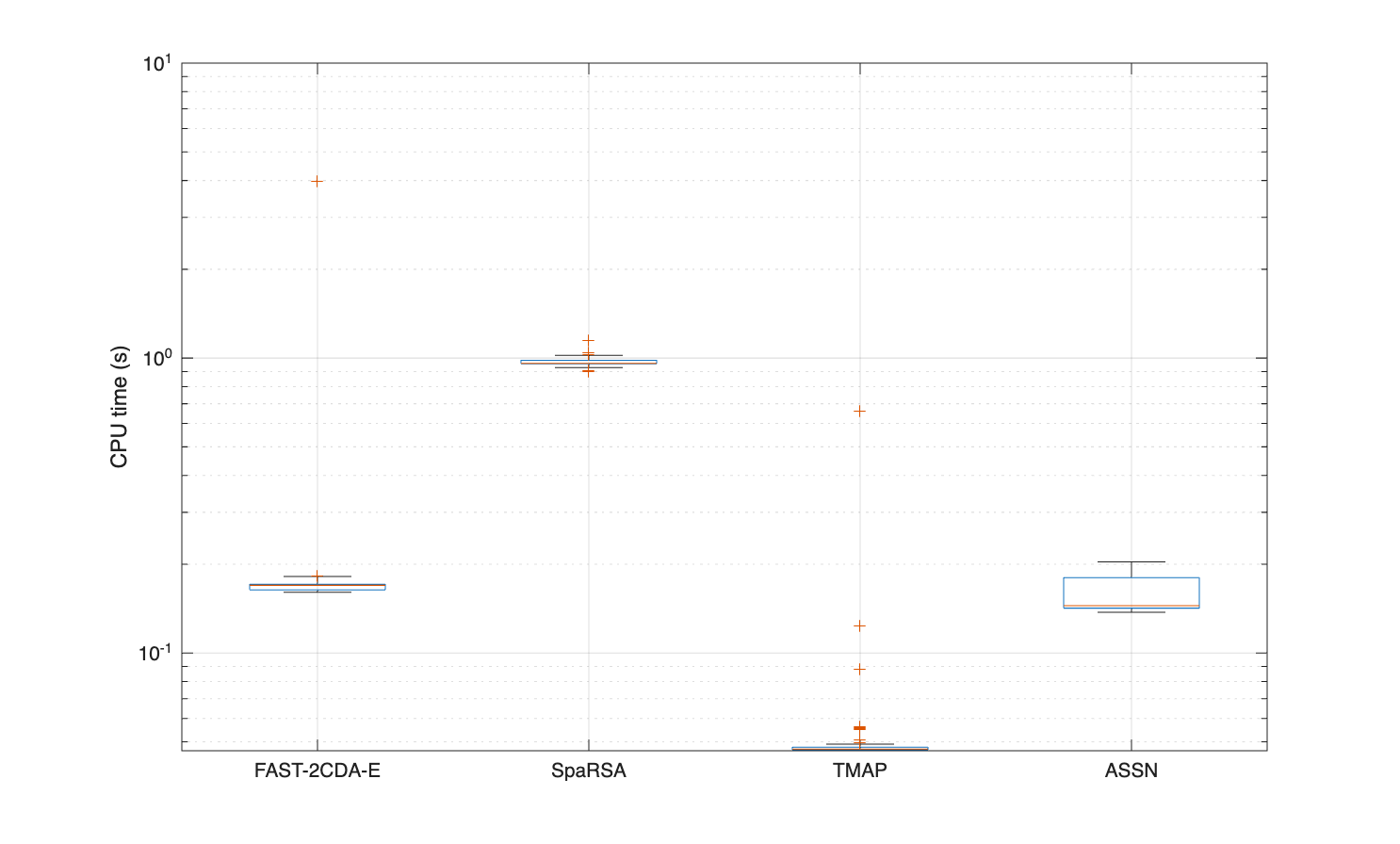}
        \label{2^14-0.01-1e-10}
    }
    \subfigure[\textbf{$n = 2^{14}, \rho = 0.05$}]{
        \includegraphics[width=0.47\linewidth]{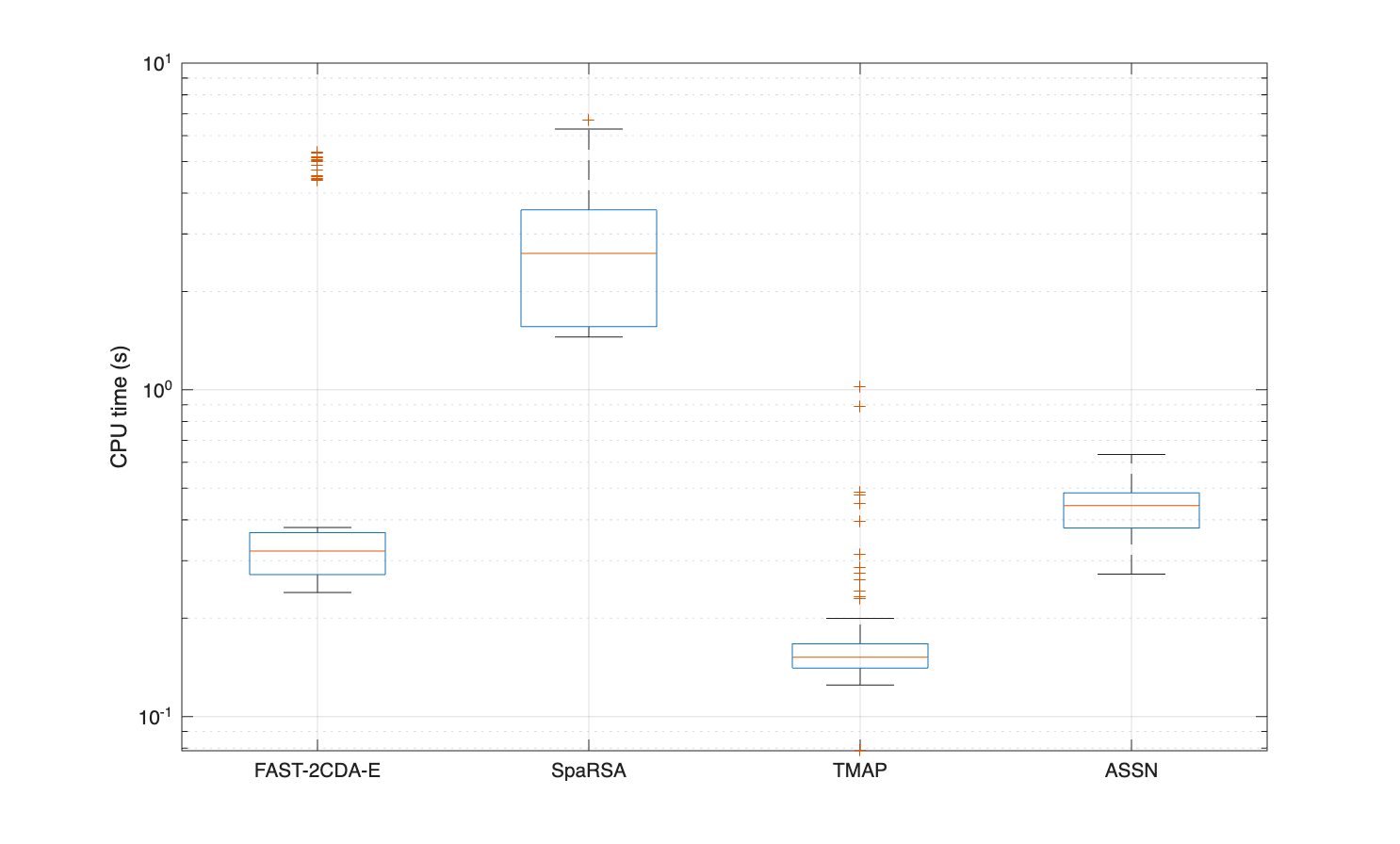}
        \label{2^14-0.05-1e-10}
    }
    \subfigure[\textbf{$n = 2^{14}, \rho = 0.1$}]{
        \includegraphics[width=0.47\linewidth]{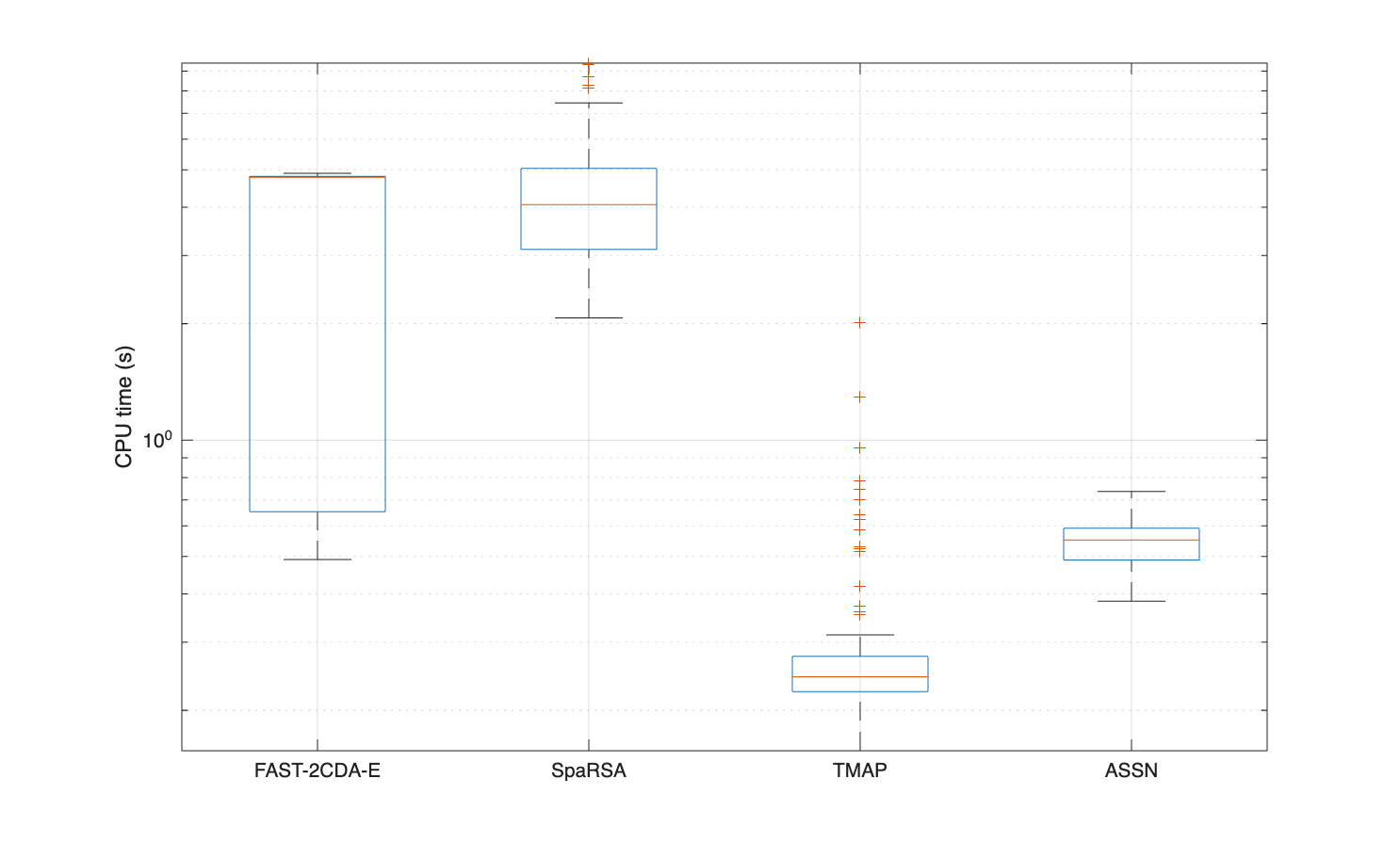}
        \label{2^14-0.1-1e-10}
    }
    \subfigure[\textbf{$n = 2^{15}, \rho = 0.01$}]{
        \includegraphics[width=0.47\linewidth]{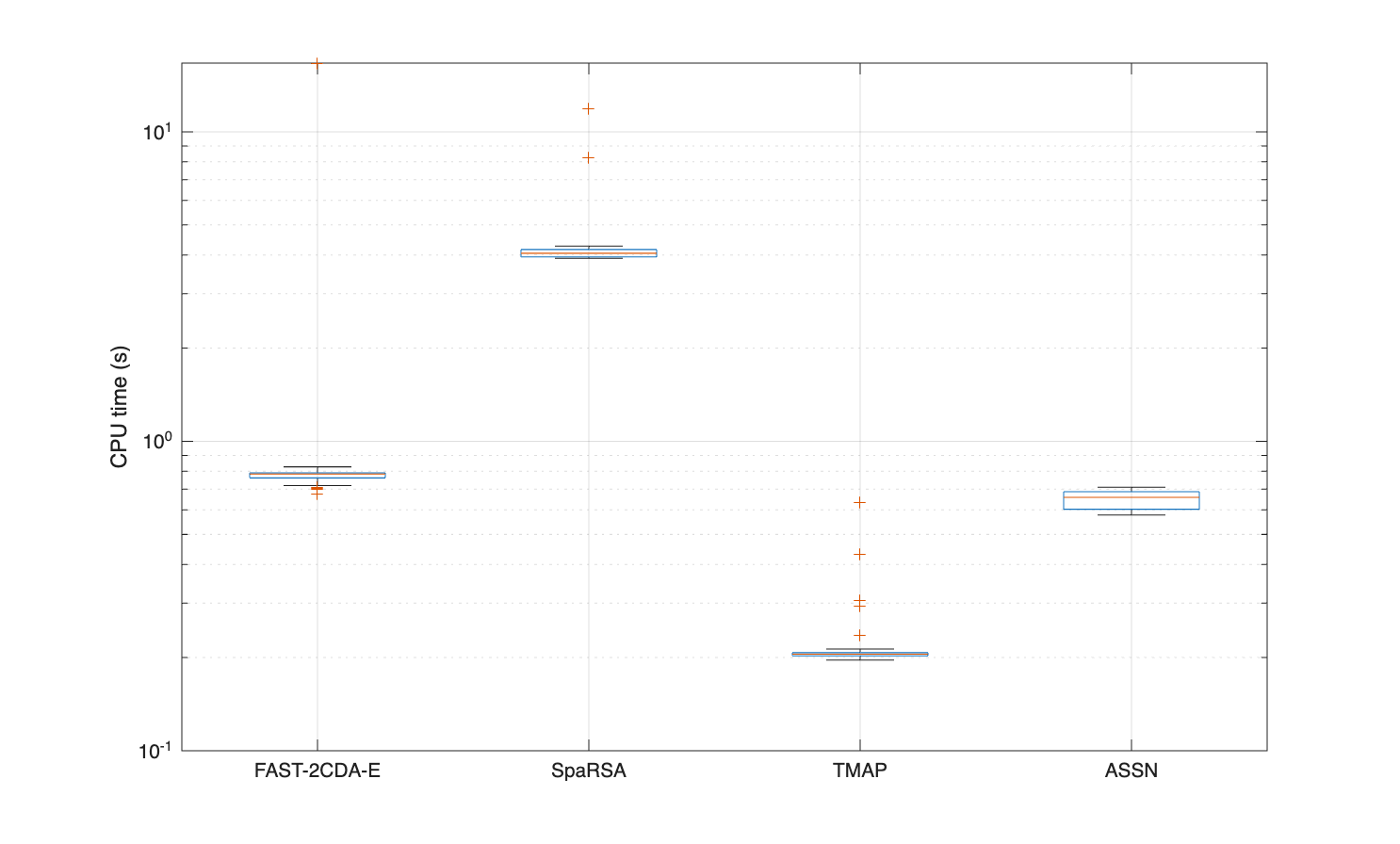}
        \label{2^15-0.01-1e-10}
    }
    \subfigure[\textbf{$n = 2^{15}, \rho = 0.05$}]{
        \includegraphics[width=0.47\linewidth]{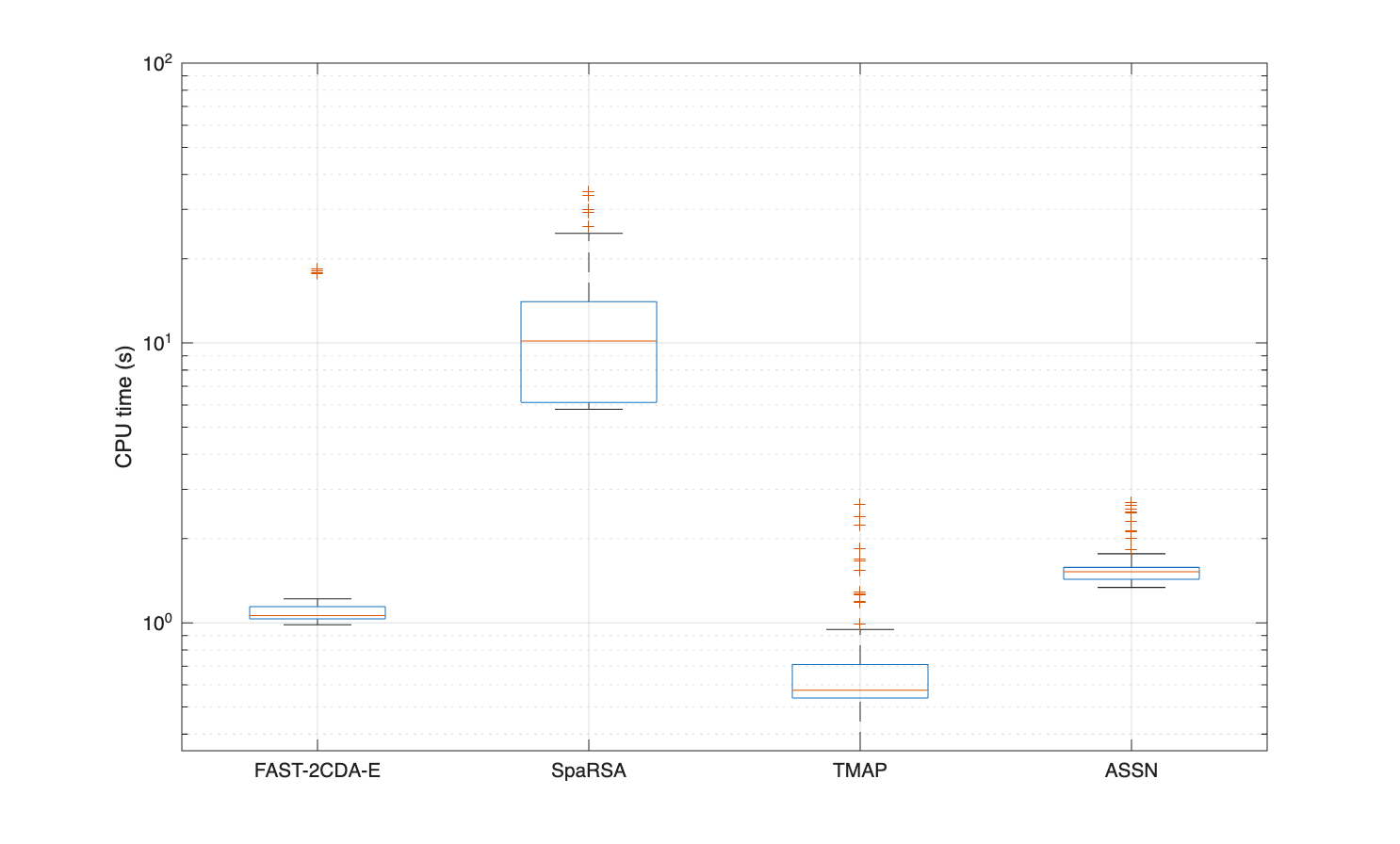}
        \label{2^15-0.05-1e-10}
    }
    \subfigure[\textbf{$n = 2^{15}, \rho = 0.1$}]{
        \includegraphics[width=0.47\linewidth]{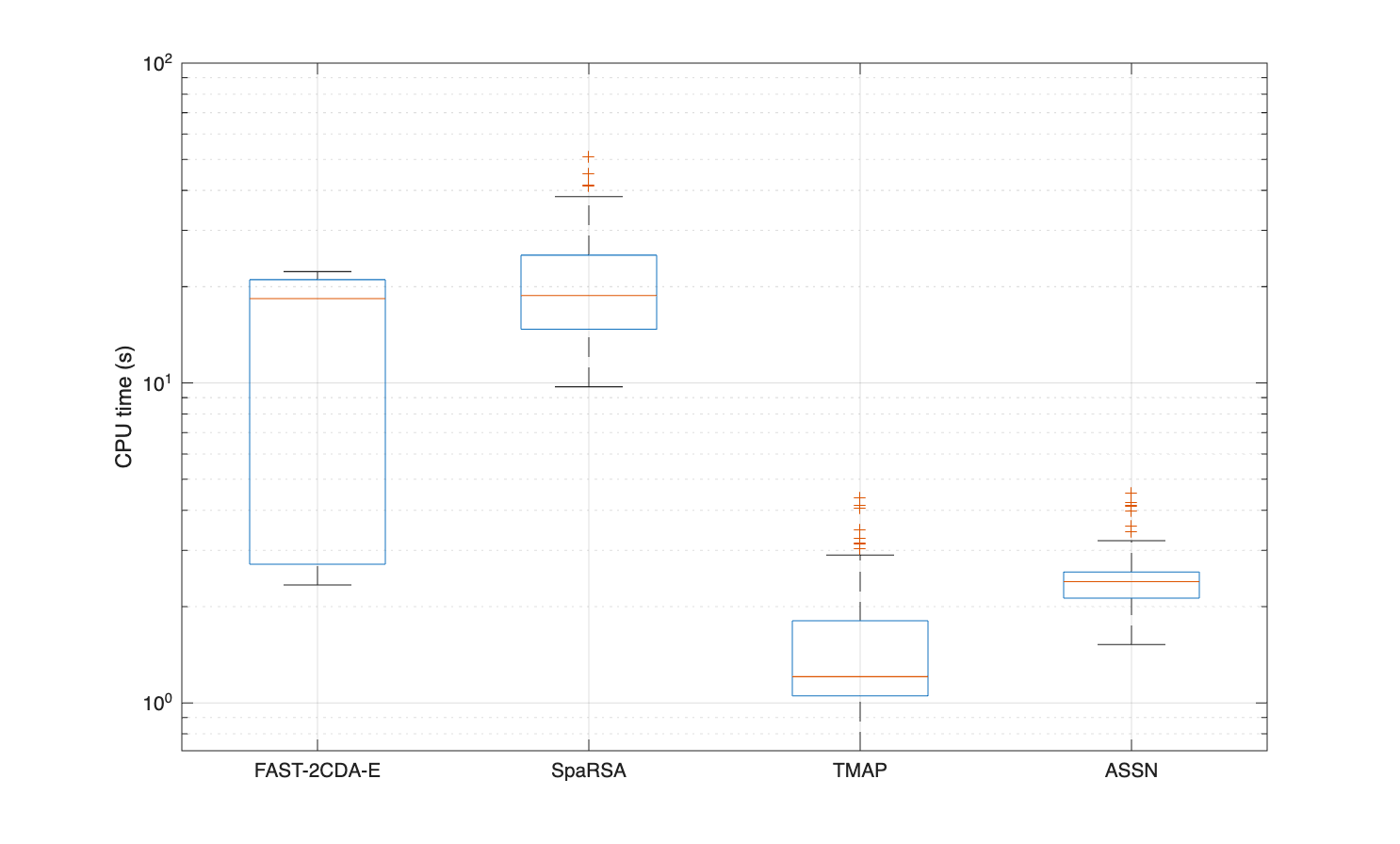}
        \label{2^15-0.1-1e-10}
    }
    \caption{\texttt{tol} = $10^{-10}$}
    \label{fig:lasso-1e-10}
\end{figure}

From Figures~\ref{fig:lasso-1e-6}, \ref{fig:lasso-1e-8}, and \ref{fig:lasso-1e-10}, we observe that \textbf{TMAP} consistently outperforms all competing algorithms in terms of computational time across all problem sizes, sparsity levels, and accuracy levels.
Notably, \textbf{ProjectionL1} and \textbf{PSSgb}—which implement the two-metric projection method and its variant—exhibit significant performance degradation as accuracy levels are raised,  especially on larger problems or {less} sparse signals (larger $\rho$ or $n$). Specifically, \textbf{PSSgb} and \textbf{ProjectionL1} require at least 3 times, sometimes over 10 times the computational time of \textbf{TMAP},  to achieve a tolerance of $10^{-6}$.
%This empirical observation directly validates the theoretical discussion in Section~\ref{sec: prelim}, which identified the inherent numerical instability arising from applying the original two-metric projection method to the equivalent bound-constrained reformulation.
We further observe that \textbf{FAST-2CDA-E} shows notable performance degradation as problem size increases or signal sparsity decreases (larger $\rho$ or $n$). When applicable, it generally outperforms \textbf{ProjectionL1} and \textbf{PSSgb}, and achieves performance comparable to \textbf{ASSN}. 
%yet still remains slower than \textbf{TMAP}.
Collectively, these results demonstrate that \textbf{TMAP} is robust to variations in problem scale, sparsity structure, and accuracy levels, showcasing superior efficiency and stability when solving large-scale $\ell_1$-regularized reconstruction problems to high precision.

\section{Conclusion}\label{sec: con}

In this paper, we propose a two-metric adaptive projection (TMAP) method for solving the $\ell_1$-norm minimization problem. 
Compared with the two-metric projection method \cite{bertsekas1982projected} and the ProjectionL1 method \cite{schmidt2007fast} for the equivalent bound-constrained formulation, our algorithm directly solves the $\ell_1$-norm minimization problem. 
%require an exact solution of the Newton system nor the strong convexity of $f$. 
Our index-set selection strategy helps avoid singularity issues and numerical instabilities that may arise when applying the original two-metric projection method to the equivalent bound-constrained formulation in practice, making our algorithm more robust. 
The proposed \textbf{TMAP} method can be shown to converge globally to an optimal solution and enjoys a manifold identification property, as well as a locally superlinear convergence rate. 
A key ingredient in our analysis is an error bound condition, which is weaker than either the strong convexity of $f$ or the sufficient second-order optimality condition of problem~\eqref{l1min}.

We conduct extensive numerical experiments, comparing our algorithm with several competitive first-order and second-order methods from the literature on large-scale $\ell_1$-regularized logistic regression and LASSO problems. 
The results demonstrate that our algorithm is highly efficient and robust in solving such problems to high accuracy. 
These findings show that \textbf{TMAP} not only has strong theoretical guarantees but also delivers superior numerical performance: maintaining high efficiency and robustness in practice.  {The future directions include extension to nonconvex regularizers such as the $\ell_0$-norm.  Despite nonconvexity,  $\ell_0$ norm is separable and facilitate partition of the index set hence different types of updates across the coordinates.}

\medskip

\noindent{\bf Acknowledgements.} We would like to acknowledge the valuable suggestions provided by Professor Stephen J. Wright, Professor Kim-Chuan Toh and Professor Chingpei Lee when writing this manuscript.  The research of YX is supported by Hong Kong RGC General Research Fund (Project number: 17300824),  Guangdong Province Fundamental and Applied Fundamental Research Regional Joint Fund (Project 2022B1515130009) and HKU start-up fund.

%\bibliographystyle{informs2014} % outcomment this and next line in Case 1
\bibliography{sn-bibliography} % if more than one, comma separated

\newpage
\begin{appendices}

\section{Preliminaries on Riemannian geometry}\label{app: prelim}
Throughout this paper, all smooth manifolds $\mathcal{M}$ are assumed to be embedded in $\mathbb{R}^n$ and we consider the tangent and normal spaces to $\mathcal{M}$ as subspaces of $\mathbb{R}^n$.
A non-empty set $\mathcal{M} \subset \mathbb{R}^{n}$ is a $C^{p}$-smooth embedded submanifold of dimension $d$ if around any $x \in \mathcal{M}$ there exists an open neighborhood $\mathcal{U} \subset \mathbb{R}^{n}$ of $x$ and a $C^{p}$-smooth map $F: \mathcal{U} \rightarrow \mathbb{R}^{n-d}$ such that the Jacobian matrix $\nabla F(x)$ has full rank $n-d$ and $\mathcal{M} \cap \mathcal{U}=F^{-1}(0)$.
The tangent space of $\mathcal{M}$ at $x$ is defined as $T_{x} \mathcal{M}=\operatorname{ker} \nabla F(x)$.
Another definition of the tangent space is that $T_{x} \mathcal{M}$ consists of all tangent vectors to smooth curves on $\mathcal{M}$ passing through $x$, i.e., $\dot{\gamma}(0)$ for all $C^{1}$-smooth curves $\gamma:(-\epsilon, \epsilon) \rightarrow \mathcal{M}$ with $\gamma(0)=x$. The normal space of $\mathcal{M}$ at $x$ is defined as the orthogonal complement of $T_{x} \mathcal{M}$, i.e., $N_{x} \mathcal{M}=\left(T_{x} \mathcal{M}\right)^{\perp}=\operatorname{Range}\left(\nabla F(x)^{T}\right)$.

For a $C^{p}$-smooth manifold $\mathcal{M}$, a function $f: \mathcal{M} \rightarrow \mathbb{R}$ is said to be $C^{p}$-smooth if for each $x \in \mathcal{M}$ there exists an open neighborhood $\mathcal{U} \subset \mathbb{R}^{n}$ of $x$ and a $C^{p}$-smooth function $\tilde{f}: \mathcal{U} \rightarrow \mathbb{R}$ such that $\tilde{f}|_{\mathcal{M} \cap \mathcal{U}}=f|_{\mathcal{M} \cap \mathcal{U}}$. The function $\tilde{f}$ is called a $C^{p}$-smooth extension of $f$ around $x$.

It is convenient to equip each tangent space of the manifold $\mathcal{M}$ with an inner product $\langle\cdot, \cdot\rangle_{x} : T_{x} \mathcal{M} \times T_{x} \mathcal{M} \rightarrow \mathbb{R}$. If this inner product varies smoothly with $x$, then $\mathcal{M}$ is called a Riemannian manifold and the inner product is called a Riemannian metric. When $\mathcal{M}$ is an embedded submanifold of $\mathbb{R}^{n}$, the tangent space $T_{x} \mathcal{M}$ is a subspace of $\mathbb{R}^{n}$. It is natural to choose the Riemannian metric as the restriction of the Euclidean inner product on $\mathbb{R}^{n}$ to $T_{x} \mathcal{M}$, i.e., $\langle\eta, \xi\rangle_{x}=\eta^{T} \xi$ for all $\eta, \xi \in T_{x} \mathcal{M}$ and we call $\mathcal{M}$ a Riemannian submanifold of $\mathbb{R}^{n}$ in this case.

For a Riemannian manifold $\mathcal{M}$, we can define the Riemannian gradient of a smooth function $f: \mathcal{M} \rightarrow \mathbb{R}$ as follows. The Riemannian gradient $\operatorname{grad}_{\mathcal{M}} f(x) \in T_{x} \mathcal{M}$ is the unique tangent vector satisfying
\begin{equation*}
    \langle\operatorname{grad}_{\mathcal{M}} f(x), \xi\rangle_{x}=\mathrm{D} f(x)[\xi] = \frac{d}{dt} f(\gamma(t))\big|_{t=0}, \quad \forall \xi \in T_{x} \mathcal{M}
\end{equation*}
where $\mathrm{D} f(x)[\xi]$ is the directional derivative of $f$ at $x$ along $\xi$ and $\gamma:(-\epsilon, \epsilon) \rightarrow \mathcal{M}$ is any smooth curve satisfying $\gamma(0)=x$ and $\dot{\gamma}(0)=\xi$.
If $\mathcal{M}$ is a Riemannian submanifold of $\mathbb{R}^{n}$, then the Riemannian gradient can be computed via the projection of the Euclidean gradient onto the tangent space, i.e.,
\begin{equation*}
    \operatorname{grad}_{\mathcal{M}} f(x)=P_{T_{x} \mathcal{M}} \nabla \tilde{f}(x)
\end{equation*}
where $\tilde{f}: \mathcal{U} \rightarrow \mathbb{R}$ is any smooth extension of $f$ around $x$.
The Riemannian Hessian of $f$ at $x$ is defined as the linear operator $\operatorname{Hess}_{\mathcal{M}} f(x): T_{x} \mathcal{M} \rightarrow T_{x} \mathcal{M}$ such that
\begin{equation*}
    \operatorname{Hess}_{\mathcal{M}} f(x)[\xi] = \nabla_{\xi} \operatorname{grad}_{\mathcal{M}} f(x), \quad \forall \xi \in T_{x} \mathcal{M}
\end{equation*}
where $\nabla$ is the Riemannian connection on $\mathcal{M}$.
When $\mathcal{M}$ is a Riemannian submanifold of $\mathbb{R}^{n}$, the Riemannian Hessian can be computed via the projection of the Euclidean Hessian onto the tangent space, i.e.,
\begin{equation*}
    \operatorname{Hess}_{\mathcal{M}} f(x)[\xi]=P_{T_{x} \mathcal{M}} (\nabla^{2} \tilde{f}(x)[\xi]), \quad \forall \xi \in T_{x} \mathcal{M}
\end{equation*}
where $\tilde{f}: \mathcal{U} \rightarrow \mathbb{R}$ is any smooth extension of $f$ around $x$.

\section{Proofs}
\label{sec: appendixA}
%\paragraph{Proof of Lemma~\ref{lm: tbd+}}
%
%\paragraph{Proof of Lemma~\ref{lm: tbd-}}
%    \begin{lemma}
%        \label{manifold-identification}
%        For an optimal point $x^* \in X^*$ that satisfies SC, there exists a neighborhood $B(x^*, b_3)$ such that for all $x_k \in B(x^*, b_3)$,
%        $I_k^+ = \mathcal{A}(x^*) = \mathcal{A}(x_{k+1})$.
%    \end{lemma}
\paragraph{Proof of Lemma~\ref{manifold-identification}}
    \begin{proofs}%[Proof of \ref{manifold-identification}]
        If $x_k$ is a stationary point, then it never moves away, i.e., $x_k = x^*$, and the lemma holds. W.l.o.g, suppose that $x_k$ is
        not stationary. Since $x^*$ is a limit point, then there exists $k$ large enough such that $\exists \delta_1 > 0$ small enough and $x_k \in B(x^*,\delta_1)$, 
        $\pi_k = \epsilon_k$ is small enough, $|x_k^i|$ is small enough and $|g_k^i|< \gamma$ for any $i \in \mathcal{A}(x^*)$.
        Therefore,
        \begin{equation*}
            \operatorname{prox}_{\gamma |\cdot|}(x_k^i- g_k^i)=0, \quad \forall i \in \mathcal{A}(x^*), x_k \in B(x^*,\delta_1),
        \end{equation*}
        which indicates that
        \begin{equation*}
            |x_k^i|=|x_k^i-\operatorname{prox}_{\gamma |\cdot|}(x_k^i- g_k^i)| \leq \pi_k = \epsilon_k, \quad \forall i \in \mathcal{A}(x^*), x_k \in B(x^*,\delta_1).
        \end{equation*}
        Then
        $
            \mathcal{A}(x^*) \subseteq I_k^+.
        $
        There exists $\delta_2$, such that when $x_k \in B(x^*,\delta_2)$ and $\forall i \notin \mathcal{A}(x^*)$,
        \begin{equation*}
            |x_k^i| > \epsilon_k=\pi_k.
        \end{equation*}
        Then
        $
            I_k^+ \subseteq \mathcal{A}(x^*).
        $
        Therefore, $I_k^+ = \mathcal{A}(x^*)$. Next we prove that $\mathcal{A}(x_{k+1}) = I_k^+$. Indeed, let $b_3 \leq \min\{\delta_1, \delta_2\}$ such that for all $x_k \in B(x^*,b_3)$, $I_k^+ = \mathcal{A}(x^*)$ and $\| \bar g_k + \bar \omega_{k} \|$ small enough. 
        Combine (\ref{ineq: pkbd2}) with the fact that $\| \bar g_k + \bar \omega_{k} \|$ small enough and $t_k \leq 1$, we have $t_k \| \bar p_k \|$ can be arbitary small.
    
        Also note that, if $b_3$ is small enough, there exists a constant $\bar{\epsilon} > 0$ such that
        \begin{equation*}
            |x_k^i| > \bar{\epsilon}, \quad \forall i \notin \mathcal{A}(x^*) \text{ i.e. } i \in I_k^-
        \end{equation*}
        Therefore, if $b_3$ is small enought, for any $i \in I_k^-$, $\operatorname{sign}(x_k^i) = \operatorname{sign}(x_k^i-t_k p_k^i)$, which indicates that
        \begin{equation*}
            I_k^- \subseteq \operatorname{supp}(x_{k+1}) \implies \mathcal{A}(x_{k+1}) \subseteq I_k^+.
        \end{equation*}
        Last but not least, we need to prove that $I_k^+ \subseteq \mathcal{A}(x_{k+1})$. By similar arguments from Lemma~\ref{lm: tbd-} and Theorem~\ref{thm: suffdescent} we know that
        \begin{align}\label{ineq: tlb1}
            t_k > \beta \min \left\{\frac{ 2(1-\sigma)}{L_g}, \frac{2}{L_g}(1-\sigma)(1-\tau)\mu_k, \frac{\bar{\epsilon}}{G}, 1\right\}.
        \end{align}
        Note that when $b_3$ is small enough, by strict complementarity, for any $i \in I_k^+$, there exists $\bar \epsilon_i$ such that $| g_k^i | \le \gamma - \bar \epsilon_i$. Therefore, for
        any $i \in I_k^+$, as long as
        \begin{align}\label{ineq: tlb2}
            t_k \ge |x_k^i|/(\gamma - |g_k^i|),
        \end{align}
        we have that $x_{k+1}^i = 0$ by the update rule. Note that when $b_3$ is small enough, right-hand side in \eqref{ineq: tlb2} is less 
        than any constant. Therefore, we only need to argue that for any $i \in I_k^+$,
        \begin{align*}
            \frac{2\beta}{L_g}(1-\sigma)(1-\tau)\mu_k \ge |x_k^i|/(\gamma - |g_k^i|).
        \end{align*}
        Note that since $b_3 \le \delta_1$, we have that
        \begin{align*}
            \mu_k \ge c\| x_k^+ - \operatorname{prox}_h(x_k^+ - g_k^+) \|^\delta = c\| x_k^+ \|^\delta \ge c| x_k^i|^\delta, \quad \forall i \in I_k^+.
        \end{align*}
        Therefore we only need to show that 
        \begin{align*}
            c\frac{2\beta}{L_g}(1-\sigma)(1-\tau)| x_k^i |^\delta \ge |x_k^i|/(\gamma - |g_k^i|)
            \Longleftrightarrow |x_k^i|^{\delta-1} \ge \frac{L_g/(2\beta)}{c(1-\sigma)(1-\tau)(\gamma-|g_k^i|)}.
        \end{align*}
        This holds since $\delta \in (0,1)$ and when $b_3$ is small enough. Since there are finite elements in $I_k^+$, we have that
        $x_{k+1}^i = 0$ for any $i \in I_k^+$ if $b_3$ is small enough and we conclude $I_k^+ \subseteq \mathcal{A}(x_{k+1})$.
    \end{proofs}

Next we introduce a lemma useful in our later proofs. %\yx{(Yue: Shall we move the following Lemma to the appendix. It is not very useful in the main text.)}
\begin{lemma}[\citep{boumal2023intromanifolds}]\label{lm: liphess} 
    If $f: \mathcal{M} \rightarrow \mathbb{R}$ has Lipschitz continuous Hessian, then there exists $L_H > 0$ such that
    \begin{equation}\label{ineq: liphess1}
        \left|f\left(\operatorname{Exp}_{x}(s)\right)-f(x)-\langle s, \operatorname{grad} f(x)\rangle-\frac{1}{2}\langle s, \operatorname{Hess} f(x)[s]\rangle\right| \leq \frac{L_H}{6}\|s\|^{3}
    \end{equation}
    and
    \begin{equation}\label{ineq: liphess2}
        \left\|P_{s}^{-1} \operatorname{grad} f\left(\operatorname{Exp}_{x}(s)\right)-\operatorname{grad} f(x)-\operatorname{Hess} f(x)[s]\right\| \leq \frac{L_H}{2}\|s\|^{2}
    \end{equation}
    for all $(x, s)$ in the domain of the exponential map $\operatorname{Exp}$ and $P_{s}$ denotes parallel transport along
    $\gamma(t) = \operatorname{Exp}_{x}(ts)$ from $t = 0$ to $t = 1$. If $\sM = \sM_I$ is a Riemannian submanifold of $\mathbb{R}^n$, then ${\rm Exp}_x(s) = x+ s$ and $P_s = I$.
\end{lemma}

\paragraph{Proof of Lemma~\ref{eb-pk}}
    \begin{proofs}%[Proof of \ref{eb-pk}]
        W.l.o.g suppose that $x_k$ is not stationary. By Lemma~\ref{manifold-identification} and $b_4 \le b_3$,
        \begin{equation*}
            \mathcal{A}(x^*) = \mathcal{A}(x_{k}) = I_{k}^+ = \mathcal{A}(x_{k+1})
        \end{equation*}
        Therefore, $[x_k - \operatorname{prox}_h(x_k - g_k)]_{I_k^+} = 0$. Furthermore, when $b_4 \le b_1$, manifold EB holds and we have
        \begin{equation}
            \label{mu-lowerbound}
            \mu_k =c \left\|\begin{bmatrix}
                0 \\
                [g_{k}+\omega_k]_{I_k^-}
               \end{bmatrix} \right\|^\delta = c \|\operatorname{grad}_{\mathcal{M}_{*}} \psi(x_k)\|^\delta \geq \frac{c}{c_1^\delta} \operatorname{dist}\left(x_k, X^{*} \cap \mathcal{M}_{*}\right)^\delta.
        \end{equation}
        Let $ \tilde{X}^{*}:=X^{*} \cap \mathcal{M}_* $ and define 
\begin{align}\label{def: xkproj}
\Pi_{\tilde{X}^{*}}(x) \triangleq  \operatorname{argmin}_{y \in \tilde{X}^{*}}\|y-x\|, \;  x_k^* = \Pi_{\tilde{X}^{*}}(x_k).
\end{align}        
        Note that $\left\| x_k^* -x^{*}\right\| \leq\left\| x_k^* -x_{k}\right\|+\left\|x_{k}-x^{*}\right\| \leq 2\left\|x_{k}-x^{*}\right\| \leq b_2$. 
        Therefore, we can apply Assumption~\ref{asp: liphess} and results in Lemma~\ref{lm: liphess}.  Note that ${\rm grad}_{\mathcal{M}_*}\psi(x_k^*) = 0$ and $[\operatorname{Hess}_{\mathcal{M}_{*}} \psi(x_k)(x_k- x_k^* )]_{I_k^-} = H_k(x_k - x_k^* )_{I_k^-}$ since $x_k$ and $x_k^*$ 
        are both on $\mathcal{M}_*$. Therefore,
        \begin{equation*}
            \begin{aligned}
                \|\bar{p}_k\| & = \|(H_k + \mu_k I)^{-1}(\bar{g}_k+\bar{\omega}_{k} + \bar{r}_k )\| \\
                & \leq \|(H_k + \mu_k I)^{-1}\|\|\operatorname{grad}_{\mathcal{M}_{*}} \psi(x_k) - \operatorname{grad}_{\mathcal{M}_{*}} \psi( x_k^* ) - \operatorname{Hess}_{\mathcal{M}_{*}} \psi(x_k)(x_k- x_k^* )\|\\
                & + \|(H_k + \mu_k I)^{-1}H_k (x_k- x_k^* )_{I_k^-}\| +\|(H_k + \mu_k I)^{-1}\bar{r}_k \| \\
                & \overset{\eqref{ineq: liphess2}}{\le} \frac{1}{2} L_H \mu_k^{-1}\|x_k - x_k^* \|^2 + \|(H_k + \mu_k I)^{-1}H_k \|\|(x_k-x_k^* )\| + \mu_k^{-1}\| \bar{r}_k \|.
            \end{aligned}
        \end{equation*}
        Note that
        \begin{equation*}
            \|(H_k + \mu_k I)^{-1}H_k \| = \|I - (H_k + \mu_k I)^{-1}\mu_k I \| \leq \left(\|I\|+\left\|\left(H_{k}+\mu_{k} I\right)^{-1} \mu_{k} I\right\|\right) \leq 2.
        \end{equation*}
        Therefore, we have
        \begin{equation*}
            \begin{aligned}
                (1-\tau)\left\|\bar{p}_{k}\right\|
                & \leq \frac{1}{2} L_{H} c_{1}^{\delta} c^{-1} \operatorname{dist}\left(x_{k}, X^{*} \cap \mathcal{M}_{*}\right)^{2-\delta}+2 \operatorname{dist}\left(x_{k}, X^{*} \cap \mathcal{M}_{*}\right) \\
                & \leq\left(\frac{1}{2} L_{H} c_{1}^{\delta}c^{-1}+2\right) \operatorname{dist}\left(x_{k}, X^{*} \cap \mathcal{M}_{*}\right).
            \end{aligned}
        \end{equation*}
        This concludes the proof.
    \end{proofs}
%    Lemma~\ref{eb-pk} claims that $\bar{p}_k$ has a better upper bound than \eqref{ineq: pkbd} which will be used in the next lemma. Next, we show that the linesearch of Algorithm~\ref{2m-proj} will eventually output a unit step size $t_k=1$.
    
%    \begin{lemma}
%        \label{line-search-tk}
%        There exists a neighborhood $B(x^*, b_5)$ with $b_5 \le b_4$ such that for all $x_k \in B(x^*, b_5)$, $t_k=1$.
%    \end{lemma}
\paragraph{Proof of Lemma~\ref{line-search-tk}}
    \begin{proofs}%[Proof of \ref{line-search-tk}]
        Since $b_5 \leq b_4$, $\|\bar{p}_{k}\|$ small enough such that for all $x_k \in B(x^*, b_5)$, we have
        \begin{equation*}
            \operatorname{sign}(x_k^i) = \operatorname{sign}(x_k^i-p_k^i), \quad \forall i \in I_k^-.
        \end{equation*}
        Define
        \begin{align}\label{def: xhat}
        [s_k]_{I_k^+} = 0,\quad [s_k]_{I_k^-} = \bar p_k,\quad \hat{x}_{k+1} = x_k - s_k.
        \end{align}
        Then 
        \begin{align*}
            & \psi\left(\hat{x}^{k+1}\right)-\psi\left(x^{k}\right) \\
            \leq & \frac{L_{H}}{6}\left\|s_{k}\right\|^{3}-\langle s_k, \operatorname{grad}_{\mathcal{M}_*} f(x_k)\rangle+  \frac{1}{2}\langle s_k, \operatorname{Hess}_{\mathcal{M}_*} f(x_k)[s_k]\rangle+\omega_{k}^{T}\left(\hat{x}^{k+1}-x^{k}\right) \\
            = & \frac{L_{H}}{6}\left\|\bar{p}_k\right\|^{3}+\frac{1}{2}\bar{p}_k^{T} H_k \bar{p}_k - (\bar{g}_k+\bar{\omega}_{k})^{T} \bar{p}_k \\
            = & \frac{L_{H}}{6}\left\|\bar{p}_k\right\|^{3}+\frac{1}{2}\bar{p}_k^{T} H_k \bar{p}_k - \bar{p}_k^{T}\left((H_k+\mu_k I)\bar{p}_k -\bar{r}_k \right) \\
            \leq & \frac{L_{H}}{6}\left\|\bar{p}_k\right\|^{3} - (1-\tau) \mu_k \|\bar{p}_k\|^2
        \end{align*}
        Therefore a sufficient condition for linesearch to stop at $t_k = 1$ is 
        \begin{equation*}
            \frac{L_{H}}{6}\left\|\bar{p}_k\right\| - (1-\tau) \mu_k \leq -\sigma(1-\tau) \mu_k \Longleftrightarrow
            \frac{L_{H}}{6}\left\|\bar{p}_k\right\| \leq (1-\sigma)(1-\tau) \mu_k.
        \end{equation*}
        Note that from Lemma~\ref{eb-pk} and \eqref{mu-lowerbound}, for $\delta \in (0,1)$, this sufficient condition can be satisfied when $b_5$ is small enough.
    \end{proofs}

%    From the above discussion and lemmas, we know that Algorithm~\ref{2m-proj} will degenerate to an inexact Newton method on the manifold $\mathcal{M}_*$ at iteration $k$ for $x_k$ sufficiently close to $x^*$.
%    The next lemma shows that the distance between $x_{k+1}$ and the optimal set $X^*$ is bounded by the distance between $x_k$ and $X^*$ raised to the power of $1+\delta$.
    
%    \begin{lemma}
%        \label{eb-point}
%        For all $x_k \in B(x^*, b_5)$,
%        \begin{equation*}
%            \operatorname{dist}\left(x_{k+1}, X^{*} \cap \mathcal{M}_{*}\right) \leq \mathcal{O}(\operatorname{dist}\left(x_k, X^{*} \cap \mathcal{M}_{*}\right)^{1+\delta}).
%        \end{equation*}
%    \end{lemma}

\paragraph{Proof of Lemma~\ref{eb-point}}

    \begin{proofs}%[Proof of \ref{eb-point}]
        Note that we have $x_{k+1} = \hat{x}_{k+1}$ defined in \eqref{def: xhat}.  Moreover,
        \begin{align*}
            & \|\operatorname{grad}_{\mathcal{M}_{*}} \psi(x_k) - \operatorname{grad}_{\mathcal{M}_{*}} \psi(x_{k+1}) - \operatorname{Hess}_{\mathcal{M}_{*}} \psi(x_k)(x_k-x_{k+1})\| \leq \frac{1}{2} L_H \|\bar{p}_k\|^2,  \\
            & \mu_k = c \|\operatorname{grad}_{\mathcal{M}_{*}} \psi(x_{k})\|^\delta = c\|\operatorname{grad}_{\mathcal{M}_{*}} \psi(x_{k}) - \operatorname{grad}_{\mathcal{M}_{*}} \psi(x_k^*)\|^\delta \leq c L_\psi^\delta \|x_k- x_k^* \|^{\delta},
        \end{align*}
where $\bar x_k$ is defined as in \eqref{def: xkproj}.  Therefore, combine the above inequalities with manifold EB, Lemma~\ref{eb-pk}, and definition of $\bar{r}_k$, we have
        \begin{equation*}
            \begin{aligned}
                \operatorname{dist}\left(x_{k+1}, X^{*} \cap \mathcal{M}_{*}\right) & \leq c_1 \|\operatorname{grad}_{\mathcal{M}_{*}} \psi(x_{k+1})\| \\
                & \leq c_1 \|\bar{g}_k+\bar{\omega}_{k}-H_k \bar{p}_k\| + \frac{c_1}{2} L_H \|x_k-x_{k+1}\|^2\\
                & = c_1 \|\mu_k \bar{p}_k - \bar{r}_k \| + \frac{c_1}{2} L_H \|x_k-x_{k+1}\|^2\\
                & \leq \mathcal{O}(\operatorname{dist}\left(x_k, X^{*} \cap \mathcal{M}_{*}\right))^{1+\delta}+\mathcal{O}(\operatorname{dist}\left(x_k, X^{*} \cap \mathcal{M}_{*}\right))^2\\
                & \leq \mathcal{O}(\operatorname{dist}\left(x_k, X^{*} \cap \mathcal{M}_{*}\right))^{1+\delta}.
            \end{aligned}
        \end{equation*}
        This concludes the proof.
    \end{proofs}

%    In the following lemma, we show that the sequence $x_{k}$ will be confined in a small neighborhood around $x^*$ for any $k$ large enough by induction.
%    This means that \ref{manifold-identification,eb-pk,line-search-tk,eb-point} will hold for any $k$ large enough.
%    
%    \begin{lemma}
%        \label{approaching}
%        For any $ b_6 \in (0, b_5)$, there exists a neighborhood $B(x^*, r)$ such that 
%        for all $x_k \in B(x^*, r)$, we will have
%        \begin{equation*}
%            x_{k+i} \in B(x^*, b_6) \quad \forall i = 1,2,\ldots
%        \end{equation*}
%    \end{lemma}

\paragraph{Proof of Lemma~\ref{approaching}}
    \begin{proofs}%[Proof of \ref{approaching}]
        Let 
        \begin{equation*}
            r \triangleq \min \left\{\frac{1}{2 }c_3^{-\frac{1}{\delta} }, \frac{b_6}{1+c_2 +2c_2\frac{2^{-\delta}}{1-2^{-\delta}}} \right\} < b_6,
        \end{equation*}
        where $c_2$ is the constant in Lemma~\ref{eb-pk} and $c_3$ is the constant in Lemma~\ref{eb-point} that represent $\mathcal{O}(\cdot)$.
        Suppose that $x_{k} \in B(x^*, r)$ and $x_k, x_{k+1} \in \sM_*$, then by Lemma~\ref{eb-pk}
        \begin{equation*}
                \|x_{k+1}-x^*\| \leq \|x_{k}-x^*\| + \|\bar{p}_{k}\| \leq r + c_2 \|x_k - x_k^* \| \leq r + c_2 r \leq b_6,
        \end{equation*}
        where $x_k^*$ is defined as in \eqref{def: xkproj}.  Since $b_6 \le b_3$, $x_{k+2} \in \sM_*$. Suppose $x_{k+j} \in B(x^*, b_6)$ for $1 \leq j \leq i$ and $x_{k+j} \in \sM_*$ for $1 \leq j \leq i+1$.   From Lemma~\ref{eb-point} (let $q = 1+\delta$) we have
        \begin{equation*}
            \|x_{k+j}- {x}_{k+j}^* \| \leq c_3 \|x_{k+j-1}- {x}_{k+j-1}^* \|^{q} \leq \ldots \leq c_3^{\frac{1-q^j}{1-q}} \|x_{k}- {x}_{k}^* \|^{q^j} \leq c_3^{\frac{1-q^j}{1-q}} r^{q^j} \leq 2r 2^{-q^j},
        \end{equation*}
        where the last inequality is by the definition of $r$. Note
        \begin{align}
        \notag
                        \|x_{k+i+1}-x^*\| & \leq \|x_{k+1}-x^*\| + \sum_{j=1}^{i} \|x_{k+j+1}-x_{k+j} \| = \|x_{k+1}-x^*\| + \sum_{j=1}^{i} \|\bar{p}_{k+j}\|\\
                \label{ineq: bdseq}
                & \leq \|x_{k+1}-x^*\| + \sum_{j=1}^{i} c_2 \|x_{k+j}-{x}_{k+j}^* \| \leq r + c_2 r + 2c_2 r \sum_{j=1}^{i} 2^{-q^j}.
        \end{align}
        By noticing
        \begin{equation*}
            (1+ \delta)^j \geq 1+ \delta j \geq \delta j \implies  \sum_{j=1}^{i} 2^{-q^j} \leq \sum_{j=1}^{\infty} 2^{-\delta j},
        \end{equation*}
and combine \eqref{ineq: bdseq}, we have
                \begin{equation*}
            \|x_{k+i+1}-x^*\| \leq r + c_2 r + 2c_2 r \sum_{j=1}^{\infty} 2^{-\delta j} \leq r + c_2 r + 2c_2 r \frac{2^{-\delta}}{1-2^{-\delta}} \leq b_6,
        \end{equation*}
        which means that $x_{k+i+1} \in B(x^*, b_6)$ and $x_{k+i+2} \in \sM_*$ since $b_6 \le b_3$, and the proof is completed by induction.
    \end{proofs}

%    Finally, these results indicate that $\{x_k\}$ is a Cauchy sequence, thus converging sequentially to $x^*$ superlinearly with parameter $1+\delta$.
%    We summarize the results in the following lemma.
    
%    \begin{lemma}\label{lm: Cauchy}
%        $\{x_k\}$ is Cauchy sequence and hence converges to $x^*$ superlinearly with parameter $1+\delta$.
%    \end{lemma}

\paragraph{Proof of Lemma~\ref{lm: Cauchy}}
    \begin{proofs}%[Proof of \ref{lm: Cauchy}]
        From manifold EB and Theorem~\ref{global-convergence}, we have $\operatorname{dist}\left(x_k, X^{*} \cap \mathcal{M}_{*}\right) \rightarrow 0 .$
        Thus, for $\epsilon > 0$ and $\rho < 1$, there exists $K$ such that for all $k \geq K$, we have $x_k \in B(x^*,b_6)$,  and 
        \begin{equation*}
            \operatorname{dist}\left(x_k, X^{*} \cap \mathcal{M}_{*}\right) < \epsilon \quad \text{and} \quad \operatorname{dist}\left(x_{k+1}, X^{*} \cap \mathcal{M}_{*}\right) \leq \rho \operatorname{dist}\left(x_k, X^{*} \cap \mathcal{M}_{*}\right).
        \end{equation*}
        Therefore, for $k_1> k_2 \geq K$, we have
        \begin{equation}
            \label{eq: cauchy}
            \begin{aligned}
                \|x_{k_1}-x_{k_2}\| & \leq \sum_{i=k_2}^{k_1-1}\|x_{i+1}-x_{i}\| \leq \sum_{i=k_2}^{k_1-1}\|\bar{p}_i\| \leq \sum_{i=k_2}^{k_1-1}c_2 \operatorname{dist}\left(x_i, X^{*} \cap \mathcal{M}_{*}\right) \\
                & \leq \sum_{i=0}^{\infty}c_2 \rho^i \operatorname{dist}\left(x_{k_2}, X^{*} \cap \mathcal{M}_{*}\right) 
                 \leq \frac{c_2}{1-\rho}\operatorname{dist}\left(x_{k_2}, X^{*} \cap \mathcal{M}_{*}\right)
                \leq \frac{c_2}{1-\rho} \epsilon.
            \end{aligned}    
        \end{equation}
        Therefore, $\{x_k\}$ is Cauchy sequence and hence converges to $x^*$ since $x^*$ is a limit point. Then, by taking $k_2 = k$ and $k_1 \rightarrow \infty$ in (\ref{eq: cauchy}), we have
        \begin{equation*}
            \|x_{k}-x^*\| \leq \frac{c_2}{1-\rho}\operatorname{dist}\left(x_k, X^{*} \cap \mathcal{M}_{*}\right).
        \end{equation*}
        Together with Lemma~\ref{eb-point}, we have for all $k \ge K$,
        \begin{align*}
            \|x_{k+1}-x^*\| & \leq \mathcal{O}(\operatorname{dist}\left(x_{k+1}, X^{*} \cap \mathcal{M}_{*}\right)) \leq \mathcal{O}(\operatorname{dist}\left(x_{k}, X^{*} \cap \mathcal{M}_{*}\right))^{1+\delta} \\ 
            & \leq \mathcal{O}(\|x_{k}-x^*\|^{1+\delta}).
        \end{align*}
        This concludes the proof. 
    \end{proofs}

\iffalse
\begin{example}
        \begin{equation*}
    \min _{x \in \mathbb{R}^{n}}\frac{1}{m} \sum_{i=1}^{m} \log \left(1+\exp \left(-b_{i} \cdot a_{i}^{T} x\right)\right)+\gamma\|x\|_{1} .
    \end{equation*}
    \begin{equation*}
        A = \begin{bmatrix}
            87 &83 \\
            47 &-1
        \end{bmatrix} \in \mathbb{R}^{m \times n}, \quad b = \begin{bmatrix}
           1 \\
            -1
        \end{bmatrix} \in \mathbb{R}^{m}.
    \end{equation*}
    TMP begin at $x^+_0 = [0,0]^T$ and $x^-_0 = [0,0]^T$ with $\gamma = 1/m$. 
    
    At iteration $k=1$, $x^+_1 = [0,0.0665]^T$ and $x^-_1 = [0,0]^T$.

    At iteration $k=2$, $x^+_2 = [0,0.0662]^T$ and $x^-_2 = [0.0367,0]^T$.

    Then $I_k^1 = \{1,2\}$, $I_k^2 = \{1\}$.
\end{example}
\fi

\end{appendices}




\end{document}